\documentclass[10pt]{report}
 \usepackage{amssymb,latexsym,amsmath,mathrsfs,stmaryrd,amscd,multirow,pgf,tikz,amsthm,hyperref}
 \usepackage[square, numbers, comma, sort&compress]{natbib}  
 \usepackage[all]{xy}
 \usepackage{geometry}

 \newdir{ >}{{}*!/-9pt/\dir{>}}

 \theoremstyle{plain}

 \newtheorem{theorem}{Theorem}[chapter]
\newtheorem{example}[theorem]{Example}
\newtheorem{corollary}[theorem]{Corollary}
\newtheorem{conjecture}[theorem]{Conjecture}
\newtheorem{lemma}[theorem]{Lemma}
\newtheorem{proposition}[theorem]{Proposition}

\theoremstyle{definition}
\newtheorem{definition}[theorem]{Definition}
\theoremstyle{remark}

 \DeclareMathOperator{\Ext}{Ext}

\DeclareMathOperator{\MExt}{MExt}
\DeclareMathOperator{\AExt}{AExt}
\DeclareMathOperator{\Ker}{Ker}
\DeclareMathOperator{\Coker}{Coker}
\DeclareMathOperator{\Image}{Im}
\DeclareMathOperator{\Hom}{Hom}
\DeclareMathOperator{\Der}{Der}

\begin{document}

\title{\huge \bf{The cohomology of $\lambda$-rings and $\Psi$-rings}}       \vspace{1.5in}
\author{ {\large Thesis submitted for the degree of}\\\\
      {\large Doctor of Philosophy}\\
      {\large at the University of Leicester}\\\\\\
        \vspace{.5in}
      {\large by}\\
        \vspace{.5in}
      \LARGE Michael Robinson \\
\Large Department of Mathematics \\ University of Leicester}  \vspace{1.5in}
\date{
March 2010} 
\maketitle
\newpage

\pagenumbering{roman} \setcounter{page}{2}
\addcontentsline{toc}{chapter}{Abstract}
\begin{abstract}
In this thesis we develop the cohomology of diagrams of algebras and then apply this to the cases of the $\lambda$-rings and the $\Psi$-rings. A diagram of algebras is a functor from a small category to some category of algebras. For an appropriate category of algebras we get a diagram of groups, a diagram of Lie algebras, a diagram of commutative rings, etc.

We define the cohomology of diagrams of algebras using comonads. The cohomology of diagrams of algebras classifies extensions in the category of functors. Our main result is that there is a spectral sequence connecting the cohomology of the diagram of algebras to the cohomology of the members of the diagram.

$\Psi$-rings can be thought of as functors from the category with one object associated to the multiplicative monoid of the natural numbers to the category of commutative rings. So we can apply the theory we developed for the diagrams of algebras to the case of $\Psi$-rings. Our main result tells us that there is a spectral sequence connecting the cohomology of the $\Psi$-ring to the Andr\'{e}-Quillen cohomology of the underlying commutative ring.

The main example of a $\lambda$-ring or a $\Psi$-ring is the $K$-theory of a topological space. We look at the example of the $K$-theory of spheres and use its cohomology to give a proof of the classical result of Adams. We show that there are natural transformations connecting the cohomology of the $K$-theory of spheres to the homotopy groups of spheres. There is a very close connection between the cohomology of the $K$-theory of the $4n$-dimensional spheres and the homotopy groups of the $(4n-1)$-dimensional spheres.
\end{abstract}

\addcontentsline{toc}{chapter}{Acknowledgments}

\section*{Acknowledgments}
Firstly, I would like to thank my supervisor, Dr Teimuraz
Pirashvili, for all of his useful advice and support throughout this
project and for introducing me to the subject of $\lambda$-rings.
His passion for maths has been truly inspirational and encouraging.

I would also like to thank my first supervisor, Prof. Dietrich
Notbohm, who is now in Amsterdam. I appreciate the opportunity he
gave me to carry out research, for introducing me to the subject of
homological algebra, and also for introducing me to Teimuraz.

I would like to thank all of the staff in the Department of
Mathematics at the University of Leicester. I would also like to
thank the other research students that I met. I am especially
grateful to Dr. Mark Parsons for his beneficial advice during the
first year.

I am grateful for all the support received from my family; in
particular my Mum, Dad, and Brothers. I would like to say a big
thank you to Frieda Mann and Dorothy Kinnear for being like
Grandmothers to me. A special thanks to John Fergus Burns and John Burn for their
invaluable support, and for allowing me to escape Leicester from time to time.

I would like to acknowledge the EPSRC and the University of
Leicester for all of the financial support they have provided me.

\newpage

\pdfbookmark[0]{Contents}{toc}
\tableofcontents

\pagenumbering{arabic}
\setcounter{page}{5}
\chapter{Introduction}
\label{Chapter1}

$\lambda$-rings were first introduced in an algebraic-geometry
setting by Grothendieck in 1958, then later used in group theory by
Atiyah and Tall. A $\lambda$-ring $R$ is a commutative ring with identity, together
with operations $\lambda^{i}:R \rightarrow R$, for $i \geq 0$. We
require that $\lambda^{0}(r)=1$ and $\lambda^{1}(r) = r$ for all $r
\in R$. There are more complicated axioms describing
$\lambda^{i}(r_{1}+r_{2})$, $\lambda^{i}(r_{1}r_{2})$ and
$\lambda^{i}(\lambda^{j}(r))$. The $\lambda$-operations behave like
exterior powers. The more complicated axioms are difficult to work
with, and given a $\lambda$-ring $R$, it is difficult to prove that
it is actually a $\lambda$-ring.

In 1962 Adams introduced the operations $\Psi^{i}$ to study vector fields of spheres. These operations give us another type of ring, the $\Psi$-rings, which are related
to the $\lambda$-rings by the following formula.

\[
\Psi^{i}(r) -
\lambda^{1}(r)\Psi^{i-1}(r)+...+(-1)^{i-1}\lambda^{i-1}(r)\Psi^{1}(r)
+ (-1)^{i}i\lambda^{i}(r) = 0.
\]

A $\Psi$-ring is a commutative ring $R$, together with ring homomorphisms $\Psi^{i}: R
\rightarrow R$, for $i \geq 1$. We only require that $\Psi^{1}(r) =
r$ and $\Psi^{i}(\Psi^{j}(r))=\Psi^{ij}(r)$ for all $r \in R$. The
$\Psi$-rings are much easier to work with, and in several places we
will need to pass to $\Psi$-rings to be able to carry out
computations for $\lambda$-rings.

Homological algebra is a relatively young discipline, which arose
from algebraic topology in the early $20^{th}$ century. In 1956,
Cartan and Eilenberg published their book entitled ``Homological
Algebra" \cite{CE}, which was the first book on homological algebra
and still remains a standard book of reference today. They found
that the cohomology theories for groups, associative algebras and
Lie algebras could all be described by derived functors, defined by
means of projective and injective resolutions of modules. However
the method they used was not enough to define the cohomology of
commutative algebras. To overcome this problem, simplicial
techniques were developed in homological algebra.

In the 1950's Moore showed that every simplicial group $K$ is a Kan complex whose homotopy groups are the homology of a chain complex called the Moore complex of $K$.
Dold and Kan independently found that there is an equivalence between the category of simplicial abelian
groups and the category of non-negative chain complexes of abelian groups given by the Moore complex.
Using simplicial methods Dold and Puppe showed that one can define the derived functors of a non-additive functor, since simplicial homotopy doesn't involve addition.

The notion of a monad on a category traces back to R. Godement \cite{Godement}. Around 1965, Barr and Beck used comonads to define a resolution as a way to compute nonabelian derived functors.
In 1967, Andr\'{e} and Quillen independently developed what we now call Andr\'{e}-Quillen cohomology. The Andr\'{e}-Quillen cohomology is defined in general for algebras, using comonads. The $\lambda$-rings and $\Psi$-rings are particular examples which are included in this scheme, so the Andr\'{e}-Quillen cohomology is well defined for both $\lambda$-rings and $\Psi$-rings. The main difficulty is that the Andr\'{e}-Quillen cohomology is complicated and difficult to compute. Harrison had described a cohomology for commutative algebras in 1962 using a subcomplex of the Hochschild complex. The Harrison cohomology coincides with the Andr\'{e}-Quillen cohomology over a field of characteristic zero up to a dimension shift. Our aim is to develop tools which aid computation.

In 2004, Yau \cite{Yau} defined a cohomology for $\lambda$-rings in order to study deformations of the associated $\Psi$-operations. However, Yau's cohomology for $\lambda$-rings is different from the Andr\'{e}-Quillen cohomology.

\section{Outline of the thesis}
In Chapter 2, we give a short overview of some of the fundamental
concepts of homological algebra. We can trace the roots of these
concepts back to Cartan and Eilenberg in the 1950's. We provide the
definitions of additive categories, abelian categories and short
exact sequences in abelian categories. We outline the construction
of the right derived functor $\Ext^{i}$ using projective and
injective resolutions. The main references for this part of the
chapter are \cite{Weibel} and \cite{Mac}. We sketch the construction
of the cohomology of algebras in general using comonads \cite{Weibel}
and we give the example of the Andr\'{e}-Quillen cohomology for commutative rings which are the
right derived functors of the derivations functor \cite{Quillen}. We
provide an overview of the Harrison cohomology of commutative algebras \cite{Harr} and the Baues-Wirsching cohomology of a small
category with coefficients in a natural system \cite{BW}.

Chapter 2 only provides well known background material which will be
required later. It does not contain any original work. The original
research can be found in the remaining chapters of the thesis.

In Chapter 3, we turn our attention to $\Psi$-rings, which are
related to $\lambda$-rings via the Adam's operations. The first
section introduces the basic concept of a $\Psi$-ring which can be
found in \cite{Knut}. We then develop the $\Psi$-analogue of modules
and the semidirect product. These are then used to develop the
$\Psi$-analogue of derivations and extensions. The results from this
chapter are needed in chapter 4 to prove similar results for
$\lambda$-rings.

In 2005, Donald Yau published a paper entitled, ``Cohomology of
$\lambda$-rings" \cite{Yau}. In the paper he develops a cohomology of $\lambda$-rings in order to study the
deformations of the $\Psi$-ring structure. Yau's cohomology is different from the Andr\'{e}-Quillen cohomology. In the
last section of Chapter 3 I provide a definition of the deformation of a $\Psi$-ring which is different to Yau's definition. This alternative definition is
related to the Andr\'{e}-Quillen cohomology of $\Psi$-rings.

In Chapter 4, we introduce $\lambda$-rings. The first section
introduces the basic notions of $\lambda$-rings which can be found
in \cite{Knut}. We then develop the $\lambda$-analogue of modules
and the semidirect product. We then use these to develop the
$\lambda$-analogue of derivations and extensions. The last section
of Chapter 4 provides an overview of Yau's cohomology for
$\lambda$-rings.

In Chapter 5, we extend the Harrison cochain complex of a
commutative algebra to get a bicomplex whose cohomology we define to
be the Harrison cohomology of a diagram of a commutative algebra. We
then apply this theory to the case of $\Psi$-rings.

In Chapter 6, we develop a cohomology for diagrams of algebras in
general, using comonads. First, we fix a small category $I$. A
diagram of algebras is a functor $I \rightarrow
\mathfrak{Alg}(T)$, where $T$ is a monad on sets. For appropriate
$T$, we get a diagram of groups, a diagram of Lie algebras, a
diagram of commutative rings, etc. The adjoint pair
$\xymatrix{\mathfrak{Alg}(T) \ar@<0.5ex>[r] & \ar@<0.5ex>[l]
\mathfrak{Sets} }$ yields a comonad which we denote by $\mathbb{G}$.
We can also consider the category $I_{0}$, which has the same
objects as $I$, but only the identity morphisms. The inclusion
$I_{0} \subset I$ yields the functor $\mathfrak{Sets}^{I}
\rightarrow \mathfrak{Sets}^{I_{0}}$ which has a left adjoint given
by the left Kan extension. We also have the pair of adjoint functors
$\xymatrix{\mathfrak{Alg}(T)^{I} \ar@<0.5ex>[r] & \ar@<0.5ex>[l]
\mathfrak{Sets}^{I} }$ which comes from the adjoint pair
$\xymatrix{\mathfrak{Alg}(T) \ar@<0.5ex>[r] & \ar@<0.5ex>[l]
\mathfrak{Sets} }$. By putting these pairs together, we get another
adjoint pair $$\xymatrix{\mathfrak{Alg}(T)^{I} \ar@<0.5ex>[r] &
\ar@<0.5ex>[l] \mathfrak{Sets}^{I_{0}}}.$$ This adjoint pair yields
a comonad which we denote by $\mathbb{G}_{I}$. We can then take the
cohomology associated to the comonad $\mathbb{G}_{I}$. Now we have
both a global cohomology, $H^{*}_{\mathbb{G}_{I}}(A,M)$, and a local
cohomology, $H^{*}_{\mathbb{G}}(A(i),M(i))$. Our main result is that there exists
a local to global spectral sequence connecting the two: \[E^{pq}_{2}
= H^{p}_{BW}(I,\mathcal{H}^{q}(A,M)) \Rightarrow
H^{p+q}_{\mathbb{G}_{I}}(A,M),\] where $H^{p}_{BW}(I,\mathcal{H}^{q}(A,M))$ denotes the Baues-Wirsching cohomology of the small category $I$ with coefficients in the natural system $\mathcal{H}^{q}(A,M)$ on $I$ whose value on $(\alpha: i \rightarrow j)$
is given by \\$H^{q}_{\mathbb{G}}(A(i),\alpha^{*}M(j))$.

In Chapter 7, we apply our theory from Chapter 6 to the case of
$\Psi$-rings. A $\Psi$-ring can be considered as a diagram of a
commutative ring, so we can apply our results to get a cohomology
for $\Psi$-rings. We also define the cohomology of $\lambda$-rings using comonads. We
note that there are homomorphisms connecting the cohomology of
$\lambda$-rings, the cohomology of the associated $\Psi$-rings and
the Andr\'{e}-Quillen cohomology of the underlying commutative
rings.

The last Chapter looks at applications of the earlier developed
theory. Our main application is in algebraic topology.  For any topological space $X$
such that $K^{1}(X)=0$, there exists a homomorphism natural in $X$, $\tau :
\pi_{2n-1}(X) \rightarrow Ext_{\Psi}(K(X),\widetilde{K}(S^{2n}))$. We show that
the cohomology of $\lambda$-rings and $\Psi$-rings can be used to
prove the classical result of Adams. We also show that the
$\Psi$-ring cohomology of $K(S^{2n})$ is related to the stable homotopy groups
of spheres via the natural transformation $\tau$.

\chapter{Homological algebra}
\label{Chapter2}

\section{Category theory}
\subsection{Abelian categories}
The material in this section can be found in many textbooks,
including \cite{MacCW} and \cite{Weibel}. Before we introduce
abelian categories, we start by defining the notion of an additive
category.

An \emph{additive category} $\mathfrak{A}$ is a category such that
the following holds:
\begin{enumerate}
    \item for every pair of objects $X$ and $Y$ in $\mathfrak{A}$, the hom-set $\Hom_{\mathfrak{A}}(X,Y)$ has the
    structure of an abelian group such that morphism composition distributes over addition.
    \item $\mathfrak{A}$ has a zero object (an object which is both initial and terminal).
    \item for every pair of objects $X$ and $Y$ in $\mathfrak{A}$, their product $X \times Y$ exists.
\end{enumerate}

An abelian category is defined in terms of kernels and cokernels, so
first we will recall some other basic definitions from category
theory.

In a category $\mathfrak{C}$, a morphism $m: X \rightarrow Y$ is
called a \textit{monomorphism} if for all morphisms $f_{1}, f_{2}: V
\rightarrow X$ where $m \circ f_{1} = m \circ f_{2}$ we have
$f_{1}=f_{2}$. A morphism $e: Y \rightarrow X$ is called an
\textit{epimorphism} if for all morphisms $g_{1},g_{2}: X
\rightarrow V$ where $g_{1}\circ e = g_{2}\circ e$ we have
$g_{1}=g_{2}$.

In an additive category $\mathfrak{A}$, a \textit{kernel} of a
morphism $f:X \rightarrow Y$ is defined to be a map $i: X'
\rightarrow X$ such that $f \circ i =0$ and for any morphism $g: Z
\rightarrow X$ such that $f \circ g = 0$ there exists a unique
morphism $g': Z \rightarrow X'$ such that $i \circ g' = g$.
\[\xymatrix{&Z \ar@{.>}_{g'}[dl] \ar^{g}[d] \ar^{0}[dr] & \\ X' \ar_{i}[r] & X \ar_{f}[r] & Y}\]

Dually, in an additive category $\mathfrak{A}$, a \textit{cokernel}
of a morphism $f:X \rightarrow Y$ is defined to be a map $e: Y
\rightarrow Y'$ such that $e \circ f =0$ and for any morphism $g: Y
\rightarrow Z$ such that $g \circ f = 0$ there exists a unique
morphism $g': Y' \rightarrow Z$ such that $g' \circ e = g$.
\[\xymatrix{&Z   & \\ X \ar^{0}[ur] \ar_{f}[r] & Y \ar^{g}[u] \ar_{e}[r] & Y' \ar@{.>}_{g'}[ul]}\]

An \emph{abelian category} $\mathfrak{A}$ is an additive category
such that the following holds:
\begin{enumerate}
    \item every morphism in $\mathfrak{A}$ has a kernel and
    cokernel.
    \item every monomorphism in $\mathfrak{A}$ is the kernel of its
    cokernel.
    \item every epimorphism in $\mathfrak{A}$ is the cokernel of its
    kernel.
\end{enumerate}

The basic example of an abelian category is the category of abelian
groups, denoted by $\mathfrak{Ab}$. In the category $\mathfrak{Ab}$,
the objects are Abelian groups, and the morphisms are abelian group
homomorphisms. In general, module categories which appear throughout
algebra, are abelian categories.

If $\mathcal{I}$ is a small category and $\mathfrak{A}$ is an abelian category then the category of functors $\mathfrak{A}^{\mathcal{I}}$ as also an abelian category. The category of sets $\mathfrak{Sets}$ and the category of groups $\mathfrak{Grp}$ are not abelian categories. However, if $G$ is a group then the category of left (or right) $G$-modules, denoted by $G-\mathfrak{mod}$, is an abelian category. If $R$ is a ring then the category of left (or right) $R$-modules, denoted by $R-\mathfrak{mod}$, is an abelian category. If $R$ is a $\Psi$-ring then the category of $\Psi$-modules over $R$, denoted by $R-\mathfrak{mod}_{\Psi}$, is an abelian category. If $R$ is a $\lambda$-ring then the category of $\lambda$-modules over $R$, denoted by $R-\mathfrak{mod}_{\lambda}$, is an abelian category.

In an abelian category $\mathfrak{A}$, a \textit{short exact sequence} is a sequence
\[\xymatrix{0 \ar[r] & A \ar^{\alpha}[r] & B \ar^{\beta}[r] &C \ar[r] & 0 }\]
in which $\alpha$ is a monomorphism, $\beta$ is an epimorphism and $\Ker(\beta)=\Image(\alpha).$

In an abelian category $\mathfrak{A}$, a sequence
\[\xymatrix{\ldots \ar[r] & X^{n-1} \ar^{f^{n-1}}[r] & X^{n} \ar^{f^{n}}[r] &X^{n+1} \ar[r] & \ldots }\]
is said to be  \textit{exact at} $X^{n}$ if
$\Ker(f^{n})=\Image(f^{n-1}).$ The sequence is said to be
\textit{exact} if it is exact at $X^{n}$ for all $n \in \mathbb{Z}.$

\subsection{Modules}
Let $\mathfrak{C}$ be a (not necessarily abelian) category with finite limits, and $1$ denote a terminal object in $\mathfrak{C}$. An \textit{abelian group object} of $\mathfrak{C}$ is an object $A$ together with arrows $m: A \times A \rightarrow A$, $i:A \rightarrow A$ and $z:1 \rightarrow A$ such that the following diagrams commute.
\\(associativity of multiplication)
\[\xymatrix{A \times A \times A \ar_{id_{A} \times m}[dd] \ar^{m\times id_{A}}[rr] && A\times A \ar^{m}[dd] \\ && \\ A\times A \ar_{m}[rr] && A }\]
(left and right units)
\[\xymatrix{A \times 1 \ar_{\cong}[dr] \ar^{id_{A} \times z}[r] & A \times A \ar^{m}[d] & 1\times A \ar_{z\times id_{A}}[l] \ar^{\cong}[dl] \\ & A &  }\]
(left and right inverses)
\[\xymatrix{A  \ar^-{(i,id_{A})}[r] \ar[d]  & A \times A  \ar^{m}[d] &  A \ar[d] \ar_-{(id_{A},i)}[l] \\ 1 \ar_{z}[r] & A  & 1 \ar^{z}[l] }\]
(commutativity)
\[\xymatrix{A \times A  \ar^-{(p_{2},p_{1})}[rr] \ar_{m}[dr] && A \times A \ar^{m}[dl]\\ & A & }\]
These diagrams say that the arrows satisfy the equations of an abelian group.

Let $A,i,m,z$ and $A',m',i',z'$ be abelian group objects of $\mathfrak{C}$, a \textit{morphism of abelian group objects} is an arrow $f:A \rightarrow A'$ such that the following diagram commutes.
\[\xymatrix{A \times A \ar^-{m}[rr] \ar_-{f \times f}[dd] && A \ar^-{f}[dd] \\ && \\ A' \times A' \ar_-{m'}[rr] && A' }\]

We denote the category of abelian group objects of $\mathfrak{C}$ by $Ab(\mathfrak{C})$.

Let $A$ be any object of the category $\mathfrak{C}$. The
\textit{slice category} of objects of $\mathfrak{C}$ over $A$,
denoted by $\mathfrak{C}/A$, has as objects the arrows of
$\mathfrak{C}$ with target $A$. Given two objects $f:B \rightarrow
A$ and $g:C \rightarrow A$ of $\mathfrak{C}/A$, an arrow of
$\mathfrak{C}/A$ from $f$ to $g$ is an arrow $h:B \rightarrow C$
which makes the following diagram commute.
\[\xymatrix{B \ar^{h}[rr] \ar_{f}[dr] && C \ar^{g}[dl] \\ & A &}\]

\begin{definition}
Let $A$ be an object in a category $\mathfrak{C}$. An
$A$-\textit{module} is defined to be an abelian group object in the
category $\mathfrak{C} / A$,
\[A-\mathfrak{mod} := Ab(\mathfrak{C} / A).\]
\end{definition}
The category $A-\mathfrak{mod}$ is usually an abelian category.

\begin{definition}\label{pderivation}
Let $p:Y \rightarrow A$ be an object and $q:Z \rightarrow A$ be an
abelian group object of $\mathfrak{C} / A$, then we define the
abelian group of \textit{$p$-derivations}, denoted $\Der(Y,Z)$, to
be
\[\Der(Y,Z) := \Hom_{\mathfrak{C} / A}(p, q).\]
\end{definition}

\section{Cohomology}
The concepts of complexes and (co)homology began in algebraic topology
with simplicial and singular (co)homology. The methods of algebraic
topology have been applied extensively throughout pure algebra, and
have initiated many developments. Complexes are the basic tools of
homological algebra and provide us with a way of computing (co)homology.
The following definitions can be found in \cite{Mac} and \cite{CE}.

A \textit{cochain complex} $(C,\delta)$ of objects of an abelian category $\mathfrak{A}$ is a family $\{C^{n},\delta^{n}\}_{n \in \mathbb{Z}}$ of objects $C^{n} \in obj(\mathfrak{A})$ and morphisms (called the \textit{coboundary maps} or \textit{differential maps}) $\delta^{n}: C^{n} \rightarrow C^{n+1}$ such that $\delta^{n+1}\circ \delta^{n}=0$ for all $n \in \mathbb{Z}$.
\[\xymatrix{ \cdots \ar[r] & C^{n-2} \ar^{\delta^{n-2}}[r] & C^{n-1} \ar^{\delta^{n-1}}[r] & C^{n} \ar^{\delta^{n}}[r] & C^{n+1} \ar^{\delta^{n+1}}[r]& C^{n+2} \ar[r] &
\cdots}\]

The last condition is equivalent to saying that $\Image(\delta^{n})
\subseteq \Ker(\delta^{n+1})  $ for all $n \in \mathbb{Z}$. Hence,
one can define the \textit{cohomology} of $C$ denoted by $H^{*}(C)$,
\[H^{*}(C) = \{H^{n}(C)\}_{n\in \mathbb{Z}} \qquad \textrm{ where }
H^{n}(C) = \frac{\Ker(\delta^{n})}{\Image(\delta^{n-1})}.\]

$H^{n}(C)$ is called the $n^{th}$-cohomology of $C$. An
\textit{$n$-coboundary} is an element of $\Image(\delta^{n-1})$. An
\textit{$n$-cocycle} is an element of $\Ker(\delta^{n})$.

Let $(C,\delta)$ and $(C_{\diamond},\delta_{\diamond})$ be two cochain complexes of an abelian category $\mathfrak{A}$. A \textit{cochain map} $f: (C,\delta) \rightarrow (C_{\diamond},\delta_{\diamond})$ is a family of morphisms $\{f^{n}:C^{n} \rightarrow C_{\diamond}^{n}\}_{n \in \mathbb{Z}}$ such that $\delta_{\diamond}^{n}\circ f^{n} = f^{n+1}\circ \delta^{n}$ for all $n \in \mathbb{Z}$. The last condition is equivalent to saying the following diagram commutes.
\[\xymatrix{ \cdots \ar[r] & C^{n-2} \ar^{f^{n-2}}[d] \ar^{\delta^{n-2}}[r] & C^{n-1} \ar^{f^{n-1}}[d] \ar^{\delta^{n-1}}[r] & C^{n} \ar^{f^{n}}[d] \ar^{\delta^{n}}[r] & C^{n+1} \ar^{f^{n+1}}[d] \ar^{\delta^{n+1}}[r]& C^{n+2} \ar^{f^{n+2}}[d] \ar[r] & \cdots \\ \cdots \ar[r] & C_{\diamond}^{n-2} \ar_{\delta_{\diamond}^{n-2}}[r] & C_{\diamond}^{n-1} \ar_{\delta_{\diamond}^{n-1}}[r] & C_{\diamond}^{n} \ar_{\delta_{\diamond}^{n}}[r] & C_{\diamond}^{n+1} \ar_{\delta_{\diamond}^{n+1}}[r]& C_{\diamond}^{n+2} \ar[r] &
\cdots}\]

A cochain map $f: (C,\delta) \rightarrow
(C_{\diamond},\delta_{\diamond})$ induces homomorphisms $H^{n}(f):
H^{n}(C) \rightarrow H^{n}(C_{\diamond})$. This makes each $H^{n}$
into a functor.

A \textit{cochain bicomplex} of objects of an abelian category
$\mathfrak{A}$ is a family \\$\{C^{p,q},\delta^{p,q},\partial^{p,q}\}_{p,q \in \mathbb{Z}}$ of
objects $C^{p,q} \in obj(\mathfrak{A})$ and morphisms $\delta^{p,q}:
C^{p,q} \rightarrow C^{p+1,q}$ and $\partial^{p,q}: C^{p,q}
\rightarrow C^{p,q+1}$ such that $\delta^{p+1,q} \circ \delta^{p,q}
= 0$ and $\partial^{p,q+1} \circ \partial^{p,q} = 0$ and also
$\partial^{p+1,q}\delta^{p,q} + \delta^{p,q+1}\partial^{p,q} = 0$
for all $p,q \in \mathbb{Z}$.

It is useful to visualise a cochain bicomplex as a lattice
\[\xymatrix{& \vdots  & \vdots  & \vdots  &
\\ \ldots \ar[r] & C^{p-1,q+1}  \ar[u] \ar_{\delta^{p-1,q+1}}[r] & C^{p,q+1} \ar[u]   \ar_{\delta^{p,q+1}}[r] & C^{p+1,q+1} \ar[u]  \ar[r] & \ldots
\\ \ldots \ar[r]& C^{p-1,q} \ar_{\partial^{p-1,q}}[u] \ar_{\delta^{p-1,q}}[r] & C^{p,q} \ar_{\partial^{p,q}}[u] \ar_{\delta^{p,q}}[r] & C^{p+1,q} \ar_{\partial^{p+1,q}}[u] \ar[r] &  \ldots
\\  \ldots \ar[r] & C^{p-1,q-1} \ar_{\partial^{p-1,q-1}}[u] \ar_{\delta^{p-1,q-1}}[r] & C^{p,q-1} \ar_{\partial^{p,q-1}}[u] \ar_{\delta^{p,q-1}}[r] & C^{p+1,q-1} \ar_{\partial^{p+1,q-1}}[u] \ar[r] &  \ldots
\\ & \vdots \ar[u] & \vdots \ar[u] & \vdots \ar[u] &}\]
where each row $(C^{*,q},\delta^{*,q})$ and each column
$(C^{p,*},\partial^{p,*})$ is a cochain complex and each square
anticommutes.

The \textit{total complexes} $Tot(C)=Tot^{\prod}(C)$ and
$Tot^{\bigoplus}(C)$ of a cochain bicomplex $C$ are given by
\[ Tot^{\prod}(C)^{n} = \prod_{p+q=n}C^{p,q} \qquad \textrm{ and }
\qquad Tot^{\bigoplus}(C)^{n}=\bigoplus_{p+q=n}C^{p,q}.\]

The coboundary maps are given by $d = \delta + \partial$.
We note that $Tot^{\prod}(C) = Tot^{\bigoplus}(C)$ if $C$ is
bounded, especially if $C$ is a first quadrant bicomplex.

\begin{proposition}\label{genss}
If $C$ is a first quadrant bicomplex then we have the following convergent spectral sequence
\[
E_{2}^{p,q} = H_{h}^{p}H_{v}^{q}(C) \Rightarrow H^{p+q}(Tot(C)),
\]
where $H^{*}_{h}$ denotes the horizontal cohomology, and $H^{*}_{v}$ denotes the vertical cohomology.
\end{proposition}

\section{Classical derived functors}
A standard method of computing classical derived functors between abelian categories is to take a resolution, apply the functor, then take the (co)homology of the resulting complex. The following material can be found in \cite{CE}, \cite{Weibel} and \cite{Mac}.

\subsection{Projective and injective objects}
An object $P$ of an abelian category $\mathfrak{A}$ is \textit{projective} if for any epimorphism $e:A \twoheadrightarrow B$ and any morphism $f:P \rightarrow B$ there exists a morphism $g:P \rightarrow A$ such that $f = e \circ g$, in other words, if the following diagram commutes.
\[\xymatrix{ & P \ar^{f}[d] \ar@{.>}_{g}[dl] &\\ A \ar@{->>}_{e}[r] & B \ar[r] &
0}\]

An object $Q$ of an abelian category $\mathfrak{A}$ is \textit{injective} if for any monomorphism $m:A \hookrightarrow B$ and any morphism $f:A \rightarrow Q$ there exists a morphism $g:B \rightarrow Q$ such that $f = g \circ m$, in other words, if the following diagram commutes.
\[\xymatrix{ 0 \ar[r] & A \ar_{f}[d] \ar@{^{(}->}^{m}[r] & B \ar@{.>}^{g}[dl] \\   & Q
&}\]

An object $P$ is projective if and only if
$\Hom_{\mathfrak{A}}(P,-): \mathfrak{A} \rightarrow \mathfrak{Ab}$
is an exact functor. In other words, if and only if for any exact
sequence $0 \rightarrow A \rightarrow B \rightarrow C \rightarrow 0$
in $\mathfrak{A}$ it follows that the following sequence of groups
\[\xymatrix{0 \ar[r] & \Hom_{\mathfrak{A}}(P,A) \ar[r] &
\Hom_{\mathfrak{A}}(P,B) \ar[r] & \Hom_{\mathfrak{A}}(P,C) \ar[r] &
0}\] is also exact.

An object $Q$ is injective if and only if
$\Hom_{\mathfrak{A}}(-,Q):\mathfrak{A} \rightarrow \mathfrak{Ab}$ is
an exact functor. In other words, if and only if for any exact
sequence $0 \rightarrow A \rightarrow B \rightarrow C \rightarrow 0$
in $\mathfrak{A}$ it follows that the following sequence of groups
\[\xymatrix{0 \ar[r] & \Hom_{\mathfrak{A}}(C,Q) \ar[r] &
\Hom_{\mathfrak{A}}(B,Q) \ar[r] & \Hom_{\mathfrak{A}}(A,Q) \ar[r] &
0}\] is also exact.

\subsection{Projective and injective resolutions}
Let $A$ be an object of an abelian category $\mathfrak{A}$. A \textit{projective resolution} of $A$ is a complex $P$, where $P_{i} = 0$ for all $i<0$ and $P_{j}$ is projective for all $j \geq 0$, together with a morphism $\epsilon : P_{0} \rightarrow A$ called the \textit{augmentation} such that the augmented complex
\[\xymatrix{\ldots \ar[r] & P_{2}\ar^{\partial}[r] & P_{1}\ar^{\partial}[r]& P_{0} \ar^{\epsilon}[r] & A \ar[r]&
0}\] is exact.

Let $A$ be an object of an abelian category $\mathfrak{A}$. An \textit{injective resolution} of $A$ is a complex $Q$, where $Q_{i} = 0$ for all $i<0$ and $Q_{j}$ is injective for all $j \geq 0$, together with a morphism $\epsilon : A \rightarrow Q_{0}$ called the \textit{augmentation} such that the augmented complex
\[\xymatrix{ 0 \ar[r] & A\ar^{\epsilon}[r] & Q_{0}\ar^{\delta}[r]& Q_{1} \ar^{\delta}[r] & Q_{2} \ar[r]&
\ldots}\] is exact.

An abelian category $\mathfrak{A}$ is said to have \textit{enough projectives} if for every object $A$ of $\mathfrak{A}$, there exists a projective object $P$ of $\mathfrak{A}$ and an epimorphism $e:P \rightarrow A$.

An abelian category $\mathfrak{A}$ is said to have \textit{enough injectives} if for every object $A$ of $\mathfrak{A}$, there exists an injective object $Q$ of $\mathfrak{A}$ and a monomorphism $m:A \rightarrow Q$.

\subsection{Right derived functors}
Let $\mathfrak{A}, \mathfrak{B}$ be abelian categories, where
$\mathfrak{A}$ has enough injectives. If  $F: \mathfrak{A} \rightarrow
\mathfrak{B}$ is a covariant left exact functor, then we can
construct the \textit{right derived functors} of $F$, denoted by
$R^{n}F: \mathfrak{A} \rightarrow
\mathfrak{B}$ for $n \geq 0$. If $A$ is an object of $\mathfrak{A}$, and
$Q$ is an injective resolution of $A$, we define \[R^{n}F(A) :=
H^{n}(F(Q)).\]

Let $\mathfrak{A}, \mathfrak{B}$ be abelian categories, where
$\mathfrak{A}$ has enough projectives. If  $G: \mathfrak{A}
\rightarrow \mathfrak{B}$ is a contravariant left exact functor, then
we can construct the \textit{right derived functors} of $G$, denoted
by $R^{n}G: \mathfrak{A} \rightarrow
\mathfrak{B}$ for $n \geq 0$. If $A$ is an object of $\mathfrak{A}$,
and $P$ is a projective resolution of $A$, we define \[R^{n}G(A) :=
H^{n}(G(P)).\]

It is known that the functors $R^{n}F(A)$ and $R^{n}G(A)$ are
independent of the choice of projective/injective resolution chosen,
hence it only depends on $A$. We always get $R^{0}F(A) \cong F(A)$
and $R^{0}G(A) \cong G(A)$. Furthermore, if $A$ is injective then
$R^{n}F(A) = 0$ for $n>0$, and if $A$ is projective then $R^{n}G(A) =0 $
for $n>0$.

Given a covariant left exact functor $F: \mathfrak{A} \rightarrow
\mathfrak{B}$ between the abelian categories $\mathfrak{A}$ and $\mathfrak{B}$ and given a short exact sequence \[0 \rightarrow A_{1}
\rightarrow A_{2} \rightarrow A_{3} \rightarrow 0\] in $\mathfrak{A}$, then
there exists the following long exact sequence.
\[\xymatrix{0 \ar[r]& R^{0}F(A_{1}) \ar[r]&
R^{0}F(A_{2}) \ar[r]& R^{0}F(A_{3}) \ar[r]& R^{1}F(A_{1}) \ar[r] & \ldots
\\ \ldots \ar[r]&  R^{n}F(A_{1}) \ar[r]& R^{n}F(A_{2}) \ar[r]&
R^{n}F(A_{3}) \ar[r]& R^{n+1}F(A_{1}) \ar[r]& \ldots}\]

\subsection{Ext}
The main example of right derived functors are the functors $\Ext^{n}$.

Let $R$ be a ring, and let $M,N$ be left $R$-modules. The functor
$F(-) = \Hom_{R}(M,-): R-\mathfrak{mod} \rightarrow \mathfrak{Ab}$ is a covariant additive left exact functor,
so we can define its right derived functors \[\Ext^{n}_{R}(M,-) =
R^{n}\Hom_{R}(M,-): R-\mathfrak{mod} \rightarrow \mathfrak{Ab}.\]

Given a left $R$-module $M$ and a short exact sequence of left
$R$-modules $0 \rightarrow N' \rightarrow N \rightarrow N''
\rightarrow 0$ we obtain the following long exact sequence.
\[0 \rightarrow  \Hom_{R}(M,N') \rightarrow
\Hom_{R}(M,N) \rightarrow  \Hom_{R}(M,N'') \rightarrow
\Ext^{1}_{R}(M,N') \rightarrow   \ldots
\] \[ \ldots \rightarrow  \Ext_{R}^{n}(M,N') \rightarrow  \Ext_{R}^{n}(M,N) \rightarrow
\Ext_{R}^{n}(M,N'') \rightarrow  \Ext_{R}^{n+1}(M,N') \rightarrow
\ldots\]

Similarly $\Hom_{R}(-,N) : R-\mathfrak{mod} \rightarrow \mathfrak{Ab}$ is also a contravariant additive left
exact functor, so we can define its right derived functors
$\Ext^{n}_{R}(-,N)=R^{n}\Hom_{R}(-,N): R-\mathfrak{mod} \rightarrow \mathfrak{Ab}.$

Given a short exact sequence of left $R$-modules $0 \rightarrow M'
\rightarrow M \rightarrow M'' \rightarrow 0$ and a left $R$-module
$N$ we obtain the following long exact sequence.
\[0 \rightarrow  \Hom_{R}(M'',N) \rightarrow
\Hom_{R}(M,N) \rightarrow  \Hom_{R}(M',N) \rightarrow
\Ext^{1}_{R}(M'',N) \rightarrow   \ldots
\] \[ \ldots \rightarrow  \Ext_{R}^{n}(M'',N) \rightarrow  \Ext_{R}^{n}(M,N) \rightarrow
\Ext_{R}^{n}(M',N) \rightarrow  \Ext_{R}^{n+1}(M'',N) \rightarrow
\ldots\]

\section{Comonad cohomology}
Cartan and Eilenberg unified the cohomology theories of groups, Lie algebras and associative algebras by describing them as $\Ext$ groups in the appropriate abelian categories. Unfortunately, this approach does not work in all categories, for example in the category of commutative algebras. The right approach is the comonad cohomology using simplicial methods. This material can be found in \cite{Barr} and \cite{Weibel}.

\subsection{Monads and comonads}
A \textit{monad}  $\mathbb{T}=(T,\eta,\mu)$ in any category
$\mathfrak{C}$ consists of an endofunctor $T:\mathfrak{C}
\rightarrow \mathfrak{C}$ together with two natural transformations:
$\eta: Id_{\mathfrak{C}} \rightarrow T$, $\mu: T \circ T = T^{2}
\rightarrow T$ such that the following diagrams commute.

\[
\xymatrix{
T^{3} \ar[rr]^{T\mu} \ar[dd]_{\mu T} && T^{2} \ar[dd]^{\mu} \\
\\
T^{2} \ar[rr]^{\mu} && T }
\]

\[
\xymatrix{
Id_{\mathfrak{C}} T \ar[rr]^{\eta T} \ar@{=}[ddrr] && T^{2} \ar[dd]^{\mu} && TId_{\mathfrak{C}} \ar[ll]_{T\eta} \ar@{=}[ddll]\\
\\
 && T }
\]
The natural transformation $\eta$ is called the unit, and the
natural transformation $\mu$ is called the multiplication. The
diagrams are called the associativity, left unit and right unit
laws.

A \textit{comonad} $\mathbb{G}=(G,\varepsilon,\delta)$ in any category
$\mathfrak{C}$ consists of an endofunctor $G:\mathfrak{C}
\rightarrow \mathfrak{C}$ together with two natural transformations:
$\varepsilon: G \rightarrow Id_{\mathfrak{C}}$, $\delta: G
\rightarrow G^{2}$ such that the following diagrams commute.

\[
\xymatrix{
G \ar[rr]^{\delta} \ar[dd]_{\delta} && G^{2} \ar[dd]^{G \delta} \\
\\
G^{2} \ar[rr]^{\delta G} && G^{3}}
\]

\[
\xymatrix{ && G \ar[dd]^{\delta}\\
\\
Id_{\mathfrak{C}}G  \ar@{=}[uurr] && G^{2} \ar[ll]^{\varepsilon G}
\ar[rr]_{G\varepsilon} && GId_{\mathfrak{C}} \ar@{=}[uull] }
\]

A pair of functors $L: \mathfrak{C} \rightarrow \mathfrak{B}$ and
$R: \mathfrak{B} \rightarrow \mathfrak{C}$ are \textit{adjoint} if for all objects $A$ in
$\mathfrak{C}$ and $B$ in $\mathfrak{B}$ there exists a natural bijection
\[\Hom_{\mathfrak{B}}(L(A),B) \cong
\Hom_{\mathfrak{C}}(A,R(B)).\] Natural means that for all $f:A
\rightarrow A'$ in $\mathfrak{C}$ and $g:B \rightarrow B'$ in
$\mathfrak{B}$ the following diagram commutes.
\[\xymatrix{\Hom_{\mathfrak{B}}(L(A'),B) \ar^{\cong}[d] \ar^{Lf^{*}}[r] & \Hom_{\mathfrak{B}}(L(A),B) \ar^{\cong}[d] \ar^{g_{*}}[r] & \Hom_{\mathfrak{B}}(L(A),B') \ar^{\cong}[d]
\\ \Hom_{\mathfrak{C}}(A',R(B)) \ar^{f^{*}}[r] & \Hom_{\mathfrak{C}}(A,R(B)) \ar^{Rg_{*}}[r] & \Hom_{\mathfrak{C}}(A,R(B')) }\]

We say that $L$ is the \textit{left adjoint} of $R$, and $R$ is the
\textit{right adjoint} of $L$.

Let $ \xymatrix{ \mathfrak{C} \ar@<1ex>[r]^{L} & \mathfrak{B}
\ar@<1ex>[l]^{R}} $ be an adjoint pair of functors with adjunction
morphisms $\eta:Id \rightarrow RL$ and $\mu: LR \rightarrow Id$.
Then $\mathbb{T}=(RL,\eta,R\mu L)$ is a monad on $\mathfrak{C}$ and
$\mathbb{G}=(LR,\mu,L\eta R)$ is a comonad on $\mathfrak{B}$.

\begin{example}
Let $U:\mathfrak{Grp} \rightarrow \mathfrak{Sets}$ take a group
to the set of its elements forgetting the group structure, and take
group morphisms to functions between sets. The left adjoint functor to $U$, is the functor
$L:\mathfrak{Sets} \rightarrow \mathfrak{Grp}$ taking a set to the
free group generated by the set. The functor $T= UL: \mathfrak{Sets} \rightarrow \mathfrak{Sets}$ gives rise to a
monad and the functor $G = LU: \mathfrak{Grp} \rightarrow \mathfrak{Grp}$ gives
rise to a comonad.
\end{example}

Let $\mathbb{G}$ be a comonad on $\mathfrak{C}$. A morphism $f: X \rightarrow Y$ in $\mathfrak{C}$ is called a
\emph{$\mathbb{G}$-epimorphism} if the map $\Hom_{\mathfrak{C}}(G(Z),X) \rightarrow
\Hom_{\mathfrak{C}}(G(Z),Y)$ is surjective for all $Z$. We require
the following useful lemma.

\begin{lemma} For all objects $X$ in $\mathfrak{C}$ the morphism
$ \xymatrix{ GX \ar[r]^{\varepsilon_{X}} & X}$\\ is a
$\mathbb{G}$-epimorphism.
\end{lemma}

\begin{proof} For any map $h: GZ \rightarrow X$, we wish to find a map $f: GZ \rightarrow GX$ such that $f\varepsilon_{X} = h$.
We define $f$ via the following commuting diagram.

\[ \xymatrix{ G(GZ) \ar[rr]^{G(h)} && G(X) \ar[rr]^{\varepsilon_{X}} && X
\\
\\
GZ \ar[uu]^{\delta(Z)} \ar[uurr]^{f}}\]

Now we need to check that $\varepsilon_{X} \circ f= h$. By the naturality
of $\varepsilon$, the following diagram commutes.
\[\xymatrix{ GX \ar[rr]^{\varepsilon_{X}} && X
\\ \\
G(GZ) \ar[uu]_{G(h)} \ar[rr]_{\varepsilon_{GZ}} && GZ \ar[uu]_{h}
\\
\\
GZ \ar[uu]_{\delta(Z)} \ar[uurr]_{id_{GZ}}
\ar@<2ex>@/^1pc/[uuuu]^{f}}\] So $\varepsilon_{X}$ is a
$\mathbb{G}$-epimorphism.
\end{proof}

An object $P$ of $\mathfrak{C}$ is called
\emph{$\mathbb{G}$-projective} if for any $\mathbb{G}$-epimorphism
$f:X \rightarrow Y$, the map $\Hom_{\mathfrak{C}}(P,X) \rightarrow
\Hom_{\mathfrak{C}}(P,Y)$ is surjective.

\begin{example}
For any object $X$ in $\mathfrak{C}$ the object $G(X)$ is $\mathbb{G}$-projective.
\end{example}

\begin{lemma}
The coproduct of $\mathbb{G}$-projective objects is
$\mathbb{G}$-projective.
\end{lemma}

\begin{proof}
Let $P = \coprod_{i} P_{i}$ where $P_{i}$ is $\mathbb{G}$-projective
for all $i$. For a map \\$f: X \rightarrow Y$, one applies the
functors $\Hom_{\mathfrak{C}}(P,-)$ and $\Hom_{\mathfrak{C}}(P_{i},-)$
to get the maps $f_{*}: \Hom_{\mathfrak{C}}(P,X) \rightarrow
\Hom_{\mathfrak{C}}(P,Y)$ and $f_{i*}: \Hom_{\mathfrak{C}}(P_{i},X)
\rightarrow \Hom_{\mathfrak{C}}(P_{i},Y)$. If $f$ is a
$\mathbb{G}$-epimorphism then $f_{i*}$ is surjective for all $i$.
Using the well-known lemma $\Hom_{\mathfrak{C}}(\coprod_{i}P_{i},Z)
\cong \prod_{i}\Hom_{\mathfrak{C}}(P_{i},Z)$ one gets that if $f$ is
a $\mathbb{G}$-epimorphism then $f_{*} \cong \prod_{i} f_{i*}$ is
surjective. Hence $P$ is $\mathbb{G}$-projective if $P_{i}$ is
$\mathbb{G}$-projective for all $i$.\end{proof}

\begin{lemma}\label{lemma28}
An object $P$ is $\mathbb{G}$-projective if and only if $P$ is a
retract of an object of the form $G(Z)$.
\end{lemma}

\begin{proof}
A retract of a surjective map is surjective, so it is sufficient to
consider the case $P=G(Z)$, which is clear from the definition of a
$\mathbb{G}$-epimorphism.
\end{proof}

\subsection{Simplicial methods}

\begin{definition}
A \textit{simplicial object} in a category $\mathfrak{C}$ is a sequence of objects $X_{0}, X_{1}, \ldots, X_{n}, \ldots$
together with two double-indexed families of arrows in $\mathfrak{C}$. The \textit{face operators} are arrows $d^{i}_{n}:X_{n} \rightarrow X_{n-1}$ for $0 \leq i \leq n$ and $1 \leq n < \infty$. The \textit{degeneracy operators} are arrows $s^{i}_{n}:X_{n} \rightarrow X_{n+1}$ for $0 \leq i \leq n$ and $0 \leq n < \infty$. The face operators and degeneracy operators satisfy the following conditions:
\begin{align*}
  d^{i}_{n} \circ d^{j}_{n+1} =& d^{j-1}_{n}\circ d^{i}_{n+1} \qquad \hbox{if $0 \leq i < j \leq n+1$}
  \\ s^{j}_{n} \circ s^{i}_{n-1} =& s^{i}_{n}\circ s^{j-1}_{n-1} \qquad \hbox{if $0 \leq i < j \leq n$}
  \\ d^{i}_{n+1} \circ s^{j}_{n} =& \left\{
                                      \begin{array}{ll}
                                        s^{j-1}_{n-1} \circ d^{i}_{n}, & \hbox{if $0 \leq i < j \leq n$;} \\
                                        1, & \hbox{if $0 \leq i = j \leq n$ or $0 \leq i-1 = j < n$;} \\
                                        s^{j}_{n-1}\circ d^{i-1}_{n}, & \hbox{if $0 < j < i-1 \leq n$.}
                                      \end{array}
                                    \right.
\end{align*}
\end{definition}

An \textit{augmented simplicial object} in the category $\mathfrak{C}$ is a simplicial object $X_{*}$ together with another object
$X_{-1}$ and an arrow $\epsilon:X_{0} \rightarrow X_{-1}$ such that $\epsilon \circ d^{0}_{1} = \epsilon \circ d^{1}_{1}.$

An augmented simplicial object $X_{*} \rightarrow X_{-1}$ is called \textit{contractible} if for each $n \geq -1$ there exists a map
$s_{n}: X_{n} \rightarrow X_{n+1}$ such that $d^{0} \circ s = 1$ and $d^{i} \circ s = s \circ d^{i-1}$ for $0 < i \leq n$ and $s^{0} \circ s = s \circ s$ and $s^{i} \circ s = s \circ s^{i-1}$ for $0 < i \leq n+1.$

Let $X_{*}$ be a simplicial object in an additive category $\mathfrak{B}$. The associated chain complex to $X_{*}$, denoted by $C(X_{*})$, is the
complex \[
\xymatrix{ \ldots \ar[r] & X_{n+1} \ar^{d}[r] & X_{n} \ar^{d}[r] &  X_{n-1} \ar^{d}[r] & \ldots \ar^{d}[r] & X_{0} \ar[r] & 0}
\]
where the boundary maps $d= \sum_{i=0}^{n}(-1)^{i}d^{i}:X_{n} \rightarrow X_{n-1}.$

\begin{proposition}
If $X_{*} \rightarrow X_{-1}$ is a contractible augmented simplicial object in an abelian category $\mathfrak{A}$, then the associated chain complex $C(X_{*})$ is contractible.
\end{proposition}

\subsection{Comonad cohomology}
Let $\mathbb{G}$ be a comonad on a category $\mathfrak{C}$. For any object $A$ in $\mathfrak{C}$, we get a functorial augmented simplicial object which
we denote by $\mathbb{G}(A)_{*} \rightarrow A$. The object of $\mathbb{G}_{*}(A)$
in degree $n$ is $G^{n+1}(A)$. We define the face and degeneracy operators by \begin{align*}\varphi_{i} &= G^{i}\varepsilon G^{n-i}: G^{n+1}(A) \rightarrow G^{n}(A),
\\ \sigma_{i} &= G^{i}\delta G^{n-i}: G^{n+1}(A) \rightarrow
G^{n+2}(A), \end{align*}  for $0 \leq i \leq n$. The augmenting map is given by $\varepsilon$.
\[
\xymatrix{ \ldots \ar@<1.5ex>[r] \ar@<-1.5ex>[r]^-{\vdots} & G^{n}A
\ar@<1.5ex>[r] \ar@<-1.5ex>[r]^-{\vdots} & G^{n-1}A \ar@<1.5ex>[r]
\ar@<-1.5ex>[r]^-{\vdots} & \ldots \ar@<.5ex>[r]^{G \varepsilon}
\ar@<-.5ex>[r]_{\varepsilon G} & GA \ar[r]^{\varepsilon} & A}
\]
We call $\mathbb{G}(A)_{*}$ the $\mathbb{G}$ \textit{comonad resolution of} $A$.

Let $E: \mathfrak{C} \rightarrow \mathfrak{M}$ be a contravariant functor where $\mathfrak{M}$ is an abelian category. The comonad cohomology of an object $A$ with coefficients in $E$ is $H^{*}_{\mathbb{G}}(A, E)$ where \[H^{n}_{\mathbb{G}}(A, E) := H^{n}(C(E(\mathbb{G}_{*}(A)))).\]
By definition, $H^{*}_{\mathbb{G}}(A, E)$ is the cohomology of the associated cochain complex
\[
\xymatrix{ 0 \ar[r] & E(G(A)) \ar[r] & E(G^{2}(A))
 \ar[r] & \ldots}
\]

If $M \in A$-$\mathfrak{mod}$, then we define the cohomology of $A$
with coefficients in $M$ to be the comonad cohomology of $A$ with
coefficients in $\Der(-,M): \mathfrak{C} \rightarrow \mathfrak{Ab}$.
\[H^{n}_{\mathbb{G}}(A, M) := H^{n}_{\mathbb{G}}(A, \Der(-,M)).\]

\begin{lemma}\label{Gderiv}
$H^{0}_{\mathbb{G}}(A,M) \cong \Der(A,M)$ for all $A$.
\end{lemma}

\begin{lemma}\label{exacttt}
If $A$ is $\mathbb{G}$-projective then $H^{n}_{\mathbb{G}}(A,M)=0$ for $n>0$.
\end{lemma}

\begin{proof}
From lemma \ref{lemma28}, it is sufficient to check the case
where $A = G(Z)$. There exists a
contracting homotopy $s_{n}: G^{n+2} \rightarrow
G^{n+3}$ for $n \geq -1$ given by \[s_{n} = G^{n+1}\delta.\]
We get that $\epsilon s_{-1} = id$,
$\varphi_{n+1}s_{n} = id$, $\varphi_{0}s_{0} = s_{-1}\epsilon$, and
$\varphi_{i}s_{n} = s_{n-1}\varphi_{i}$ for all $0 \leq i \leq n$.
It follows that $H^{n}_{\mathbb{G}}(G(Z),M)=0$, for $n>0$.
\end{proof}

\subsection{Andr\'{e}-Quillen cohomology}
In 1967, M. Andr\'{e} and D. Quillen \cite{Quillen} independently introduced a (co)homology theory for commutative algebras. This theory now goes by the name of Andr\'{e}-Quillen cohomology.

Fix a commutative ring $k$ and consider the category $\mathfrak{Commalg}$ of commutative $k$-algebras.

The forgetful functor $U: \mathfrak{Commalg} \rightarrow \mathfrak{Sets}$ has a left adjoint which takes a set $X$ to the polynomial algebra $k[X]$ on $X$. This adjoint pair gives us a comonad $\mathbb{G}$ on $\mathfrak{Commalg}$.

Let $R$ be a commutative $k$-algebras, and $M \in R-\mathfrak{mod}$. We define the Andr\'{e}-Quillen cohomology of $R$ with coefficients in $M$ to be comonad cohomology of $R$ with coefficients in $\Der_{k}(-,M)$,
\[H^{n}_{AQ}(R/k,M) := H^{n}_{\mathbb{G}}(R,\Der_{k}(-,M)).\]
Note that $\Der_{k}(-,M))$ is a functor from the category of commutative $k$-algebras $\mathfrak{Commalg}$ to the category of abelian groups $\mathfrak{Ab}$.

An \textit{extension} of $R$ by $M$ is an exact sequence
\[ \xymatrix{0 \ar[r] & M \ar^{\alpha}[r] & X \ar^{\beta}[r] & R \ar[r] & 0} \]
where $X$ is a commutative $k$-algebra, the map $\beta$ is a commutative $k$-algebra homomorphism, the map $\alpha$ is a $k$-module homomorphism and
\[x \alpha(m) = \alpha(\beta(x)m)\]
for all $x \in X$ and all $m \in M$. The map $\alpha$ identifies $M$ with an ideal of square-zero in $X$.

Two extensions $X,X'$ with $R$ and $M$ fixed are \textit{equivalent} if there exists a $k$-algebra homomorphism $\phi: X \rightarrow \overline{X}$ such that the following diagram commutes.
\[\xymatrix{0 \ar[r] & M \ar@{=}[d] \ar[r] & X \ar[r] \ar[d]^{\phi} & R \ar[r] \ar@{=}[d] & 0 \\ 0 \ar[r] & M \ar[r] & \overline{X} \ar[r] & R \ar[r] & 0}\]
We denote the set of equivalence classes of extensions of $R$ by $M$ by $Extalg_{k}(R,M)$.

\begin{proposition}\label{AQprop1}
\begin{enumerate}
  \item $H^{0}_{AQ}(R/k,M) \cong \Der_{k}(R,M).$
  \item If $R$ is a free commutative algebra then $H^{n}_{AQ}(R/k,M) =0$ for $n>0$.
  \item $H^{1}_{AQ}(R/k,M) \cong Extalg_{k}(R,M).$
\end{enumerate}
\end{proposition}

\section{Harrison cohomology of commutative algebras}
In 1962, D.K. Harrison introduced a cohomology of commutative algebras. The Harrison complex is a subcomplex of the Hochschild complex in the case of commutative algebras. The Harrison complex consists of the linear functions which vanish on the shuffles. The Harrison cohomology is isomorphic to the comonad cohomology for a commutative algebra over a field of characteristic 0, however, there is a shift of one dimension. The following material can be found in \cite{Loday}.

\subsection{Hochschild cohomology}
Let $k$ be a ring, $R$ be an associative $k$-algebra and  $M$ be an $R-R$-bimodule. All the tensor products in this section are over the ground ring $k$. The Hochschild cochain complex of $R$ with coefficients in $M$ is given by
\[C^{n}_{HH}(R,M) =   \Hom_{R^{e}}(R^{\otimes n},M), \]
for $n \geq 0$ and $R^{e}= R \otimes R^{op}$. The coboundary maps $\delta^{n}: C^{n}_{HH}(R,M) \rightarrow C_{HH}^{n+1}(R,M)$ are given by
\begin{align*}
  \delta^{n}(f)(r_{0},\ldots,r_{n}) =& r_{0}f(r_{1},\ldots,r_{n}) \\ & + \sum_{i=0}^{n-1}(-1)^{i+1}f(r_{0},\ldots,r_{i}r_{i+1},\ldots,r_{n}) \\& + (-1)^{n+1}f(r_{0},\ldots,r_{n-1})r_{n}.
\end{align*}

We can then take the cohomology of the resulting complex to get the Hochschild cohomology which we denote by $HH^{n}(R,M)$. We get that
\[ HH^{n}(R,M) \cong  R^{n  }\Hom_{R^{e}}(R,M) \cong \Ext_{R^{e}}^{n}(R,M).\]

\subsection{Harrison Cohomology}
Let $S_{n}$ be the symmetric group which acts on the set
$\{1,\ldots,n\}$. A \textit{(p,q)-shuffle} is a permutation $\sigma$ in $S_{p+q}$
such that:
\[\sigma (1) < \sigma (2) < \ldots < \sigma (p) \textrm{ and } \sigma (p+1) < \sigma (2) < \ldots < \sigma (p+q).\]
For any $k$-algebra $A$ and $M \in A-\mathfrak{mod}$, we let $S_{n}$
act on the left on $C_{n}^{HH}(A,M) = M \otimes A^{\otimes n}$ by:
\[\sigma \cdot (m,a_{1},\ldots, a_{n}) = (m,a_{\sigma^{-1}(1)}\ldots, a_{\sigma^{-1}(n)}). \]

Let $A'$ be another $k$-algebra, $M' \in A'-\mathfrak{mod}$. The
\textit{shuffle product}:
\[- \times - = sh_{pq}: C^{HH}_{p}(A,M) \otimes C^{HH}_{q}(A',M') \rightarrow C^{HH}_{p+q}(A \otimes A',M \otimes M'),\]
is defined by the following formula:
\begin{align*}(m,a_{1},\ldots, a_{p}) \times &(m',a'_{1},\ldots, a'_{q}) \\ =& \sum_{\sigma}sgn(\sigma)\sigma\cdot (m\otimes m',a_{1}\otimes 1,\ldots,a_{p}\otimes 1,1\otimes a'_{1},\ldots,1\otimes a'_{q}).\end{align*}

\begin{proposition}
The Hochschild boundary $b$ is a graded derivation for the shuffle
product
\[ b(x \times y) = b(x) \times y  + (-1)^{|x|}x \times b(y), \qquad x \in C^{HH}_{p}(A,M), y \in C^{HH}_{q}(A',M').\]
where the Hochschild boundary $b: C^{HH}_{n}(A,M) \rightarrow C^{HH}_{n-1}(A,M)$ is given by:
\begin{align*}b(m,a_{1},\ldots,a_{n}) = & (ma_{1},a_{2},\ldots,a_{n}) + \sum_{i=1}^{n-1}(-1)^{i}(m,a_{1},\ldots,a_{i}a_{i+1},\ldots, a_{n}) \\ &+ (-1)^{n}(a_{n}m,a_{1},\ldots,a_{n-1}).
\end{align*}
\end{proposition}

Assume that $A$ is commutative and $M$ is symmetric (symmetric means that $am = ma$ for all $a \in A$ and $m \in M$). The product map $\mu: A \otimes A
\rightarrow A$ is a $k$-algebra homomorphism, and $\mu': A \otimes M
\rightarrow M$ is a homomorphism of bimodules. Composition of the
shuffle map with $\mu \otimes \mu'$ allows us to define the
\textit{inner shuffle map}
\[- \times - = sh_{pq}: C^{HH}_{p}(A,A) \otimes C^{HH}_{q}(A,M) \rightarrow C^{HH}_{p+q}(A,M),\]
given by the formula
\[(a_{0},a_{1},\ldots, a_{p}) \times (m,a_{p+1},\ldots, a_{p+q}) = \sum_{\sigma=(p,q)-shuffle}sgn(\sigma)\sigma\cdot (a_{0}m,a_{1},\ldots,a_{p+q}).\]

We let $J$ denote $\bigoplus_{n>0}C_{n}^{HH}(A,A)$. Note that $J
\subset C^{HH}_{*}(A,A)$. We define the Harrison chain complex to be
the quotient $C^{Harr}_{*}(A,M) = C^{HH}_{*}(A,M) /
J.C^{HH}_{*}(A,M)$.

Note that
\begin{align*}C^{n}_{HH}(A,M) =& \Hom_{A^{e}}(A^{\otimes n}, M) \cong \Hom_{A \otimes A^{e}}(A \otimes A^{\otimes n}, M) \\=& \Hom_{A \otimes A^{e}}(C^{HH}_{n}(A,A),M).\end{align*}

We define the Harrison cochain complex by taking
\[C^{*}_{Harr}(A,M) := \Hom_{A \otimes A^{e}}(C_{*}^{Harr}(A,A),M).\]

For example
\begin{align*}C^{0}_{Harr}(A,M) =& M, \\ C^{1}_{Harr}(A,M) =& C^{1}_{HH}(A,M), \\ C^{2}_{Harr}(A,M) =& \{f \in C^{2}_{HH}(A,M) | f(x,y)=f(y,x) \},
 \\ C^{3}_{Harr}(A,M) =& \{f \in C^{3}_{HH}(A,M) | f(x,y,z)-f(y,x,z)+f(y,z,x)=0. \} \end{align*}

We define the Harrison cohomology of $A$ with coefficients in $M$ to be the cohomology of the Harrison cochain complex.
\[Harr^{n}(A,M) := H^{n}(C^{*}_{Harr}(A,M)).\]

\begin{lemma}
$Harr^{1}(A,M) \cong \Der(A,M).$
\end{lemma}

An \textit{additively split extension} of $A$ by $M$ is an extension of $A$ by $M$
\[\xymatrix{0 \ar[r] & M \ar^{q}[r] & X \ar^{p}[r] & A \ar[r] & 0}\]
where there exists $s:A \rightarrow X$ which is an additive
section of $p$.

Two additively split extensions $(X),(\overline{X})$ with $A,M$
fixed are said to be \textit{equivalent} if there exists a
homomorphism of commutative algebras
 $\phi: X \rightarrow \overline{X}$ such that the
following diagram commutes.
\[\xymatrix{0 \ar[r] & M \ar@{=}[d] \ar[r] & X \ar[r] \ar[d]^{\phi} & A \ar[r] \ar@{=}[d] & 0 \\ 0 \ar[r] & M \ar[r] & \overline{X} \ar[r] & A \ar[r] & 0}\]

We denote the set of equivalence classes of additively split
extensions of $A$ by $M$ by $\AExt(A,M)$.

\begin{lemma}
$Harr^{2}(A,M) \cong \AExt(A,M).$
\end{lemma}

\begin{proof}
Given an additively split extension of $A$ by $M$
\[\xymatrix{0 \ar[r] & M \ar^{q}[r] & X \ar^{p}[r] &  A \ar[r] & 0}\]
there is an additive homomorphism $s:A \rightarrow X$ which is a
section of $p$. The section induces an additive isomorphism $X \approx A \oplus M$ where multiplication in $X$ is given by $(a,m)(a',m') = (aa', ma' + am'
+ f(a,a'))$ where the bilinear map $f:A\times A \rightarrow M$ is given by
\[f(a,a') = s(a)s(a')-s(aa').\] The map $f$ is a 2-cocycle. Given two additively split
extensions which are equivalent, then the two 2-cocycles we get
differ by a 2-coboundary.

A 2-cocycle is a map $f:A \times A \rightarrow M$. We get an
additively split extension of $A$ by $M$ given by taking the exact
sequence
\[ \xymatrix{0 \ar[r] & M \ar[r] & M\oplus A \ar[r] &  A \ar[r] & 0} \]
where addition in $M\oplus A$ is given by $(m,a)+(m',a') = (m+m',
a+a')$ and multiplication is given by  \[(m,a)(m',a') = (a'm + am' +
f(a,a'), aa').\] Given two 2-cocycles which differ by a 2-coboundary,
then the two additively split extensions we get are equivalent.
\end{proof}

A \textit{crossed module} consists of a commutative algebra $C_{0}$,
a $C_{0}$-module  $C_{1}$ and a  module homomorphism
\[\xymatrix{C_{1} \ar^{\rho}[r] & C_{0}},\]
which satisfies the property
\[\rho(c)c' = c \rho(c'),\]
for $c,c' \in C_{1}$. In other words, a crossed module is a chain
algebra which is non-trivial only in dimensions 0 and 1. Since
$C_{2}=0$ the condition $\rho(c)c' = c \rho(c')$ is
equivalent to the Leibnitz relation \[0  = \rho (cc') =
\rho(c)c' - c \rho(c'). \] We can define a product by \[ c *
c' := \rho(c) c', \] for $c,c' \in C_{1}$. This gives us a
commutative algebra structure on $C_{1}$ and $\rho: C_{1}
\rightarrow C_{0}$ is an algebra homomorphism.

Let $\rho : C_{1} \rightarrow C_{0}$ be a crossed module. We let
$M= \Ker (\rho)$ and $A= \Coker (\rho)$. Then the image
$\Image (\rho)$ is an ideal of $C_{0}$, $M C_{1} = C_{1}M = 0$
and $M$ has a well-defined $A$-module structure. We say such a crossed module is a crossed module over $A$ with kernel $M$.

A \textit{crossed extension} of $A$ by $M$ is an exact
sequence
\[\xymatrix{0 \ar[r] & M \ar[r]^{\alpha} & C_{1} \ar[r]^{\rho} & C_{0} \ar[r]^{\gamma} & A \ar[r] & 0}\]
where $\rho : C_{1} \rightarrow C_{0}$ is a crossed module,
$\gamma$ is an algebra homomorphism, and the module structure on $M$
coincides with the one induced from the crossed module.

A morphism between two crossed extensions consists of commutative algebra homomorphisms $h_{0}: C_{0} \rightarrow C_{0}$ and $h_{1}: C_{1} \rightarrow C'_{1}$ such that the following diagram commutes:
\[\xymatrix{0 \ar[r] & M \ar@{=}[d] \ar[r]^{\alpha} & C_{1} \ar^{h_{1}}[d] \ar[r]^{\rho} & C_{0} \ar^{h_{0}}[d] \ar[r]^{\gamma} & A \ar@{=}[d] \ar[r] & 0  \\ 0 \ar[r] & M \ar[r]^{\alpha'} & C'_{1} \ar[r]^{\rho'} & C'_{0} \ar[r]^{\gamma'} & A \ar[r] & 0 \\}\]

Let $Cross(A,M)$ denote the category of crossed modules over $A$ with kernel $M$, and let $\pi_{0} Cross(A,M)$ denote the connected components of $Cross(A,M)$.

\begin{definition}
An \textit{additively split crossed extension} of $A$ by $M$ is a
crossed extension of $A$ by \begin{equation}\label{cesasc} \xymatrix{0
\ar[r] & M \ar^{\alpha}[r] & C_{1} \ar^{\rho}[r] & C_{0}
\ar^{\gamma}[r] &  A \ar[r] & 0}\end{equation} such that all the arrows
in the exact sequence $\ref{cesasc}$ are additively split.
\end{definition}

We denote the connected components of the category of additively split crossed
extensions over $A$ with kernel $M$ by $\pi_{0} ACross(A,M)$.

\begin{lemma} If $\gamma: C_{0} \rightarrow A$ is $k$-algebra homomorphism  then
\[Harr^{2}(\gamma: C_{0} \rightarrow A,M) \cong \pi_{0} ACross(\gamma: C_{0} \rightarrow A ,M),\]
where $Harr^{*}(\gamma: C_{0} \rightarrow A,M)$ and $\pi_{0} ACross(\gamma: C_{0} \rightarrow A ,M)$ are defined as follows. Consider the following short exact sequence of cochain complexes:
\[\xymatrix{0 \ar[r] & C^{*}_{Harr}(A,M) \ar[r]^{\gamma^{*}}   \ar[r] & C^{*}_{Harr}(C_{0},M) \ar[r] \ar[r]^{\kappa^{*}} & \Coker(\gamma^{*}) \ar[r] & 0.}\] We define the cochain complex $C^{*}_{Harr}(\gamma: C_{0} \rightarrow A,M) := \Coker(\gamma^{*}).$ This allows us to define the relative Harrison cohomology \[Harr^{*}(\gamma: C_{0} \rightarrow A, M) := H^{*}(C^{*}_{Harr}(\gamma: C_{0} \rightarrow A, M)).\]

We let $ACross(\gamma: C_{0} \rightarrow A ,M)$ denote the category whose objects are the additively split crossed extensions of $A$ by $M$ \[\xymatrix{0 \ar[r] & M \ar^{\alpha}[r] & C_{1} \ar^{\rho}[r] & C_{0} \ar^{\gamma}[r] &  A \ar[r] & 0}\] with $\gamma: C_{0} \rightarrow A$ fixed. A morphisms between two of these crossed extensions consists of a morphism of crossed extensions with the map $h_{0}:C_{0} \rightarrow C_{0}$ being the identity.
\[\xymatrix{0 \ar[r] & M \ar@{=}[d] \ar[r]^{\alpha} & C_{1} \ar^{h_{1}}[d] \ar[r]^{\rho} & C_{0} \ar@{=}[d] \ar[r]^{\gamma} & A \ar@{=}[d] \ar[r] & 0  \\ 0 \ar[r] & M \ar[r]^{\alpha'} & C'_{1} \ar[r]^{\rho'} & C_{0} \ar[r]^{\gamma} & A \ar[r] & 0 \\}\]

Note that $ACross(\gamma: C_{0} \rightarrow A ,M)$ is a groupoid.
\end{lemma}

\begin{proof}
This proof is very similiar to a proof given in \cite{KL} for the crossed modules of Lie algebras.
Given any additively split crossed module of $A$ by $M$,
\[\xymatrix{0 \ar[r] & M \ar^{\alpha}[r] & C_{1} \ar^{\rho}[r] & C_{0} \ar^{\gamma}[r] &  A \ar[r] & 0,}\]
we let $V= \Ker \gamma = \Image \rho$. There are $k$-linear sections
$s: A \rightarrow C_{0}$ of $\gamma$ and $\sigma :V \rightarrow
C_{1}$ of $\rho : C_{1} \rightarrow V$. We define the map $g: A
\otimes A \rightarrow C_{1}$ by:
\[g(a,b)= \sigma(s(a)s(b)-s(ab)).\]
We also define the map $\omega : C_{0} \rightarrow C_{1}$ by:
\[\omega (c) = \sigma(c- s\gamma (c)).\]
By identifying $M$ with $\Ker \delta$, we define the map $f: C_{0} \otimes C_{0} \rightarrow M$ by:
\[f(c,c') = g(\gamma (c), \gamma (c')) + c' \omega(c) + c \omega(c') - \omega(c) * \omega(c') - \omega(cc').\]
Since $g(c,c') = g(c',c)$, it follows that $f(c,c') = f(c',c)$ and
so $f \in C^{2}_{Harr}(C_{0},M)$. We define the map $\varpi \in
C^{3}_{Harr}(A,M)$ by:
\[\varpi(x,y,z) = s(x)g(y,z) - g(xy,z) + g(x,yz) - g(y,x)s(z).\]
Note that $\varpi$ vanishes on the shuffles since $g(x,y)=g(y,x)$.

Consider the following commuting diagram.
\[\xymatrix{0 \ar[r] & C^{2}_{Harr}(A,M) \ar^{\gamma^{*}}[r] \ar[d] \ar & C^{2}_{Harr}(C_{0},M) \ar^-{\kappa^{*}}[r] \ar^{\delta}[d] & \ar[r] \ar^{\delta}[d] C^{2}_{Harr}(\gamma: C_{0} \rightarrow A, M)  & 0 \\ 0 \ar[r] & C^{3}_{Harr}(A,M) \ar^{\gamma^{*}}[r]  & C^{3}_{Harr}(C_{0},M) \ar^-{\kappa^{*}}[r]  & \ar[r]  C^{3}_{Harr}(\gamma: C_{0} \rightarrow A, M)  & 0 }\]

A direct calculation shows that $\delta f = \gamma^{*}\varpi \in
C^{3}(C_{0},M)$. We also have that $\delta \kappa^{*}f = \kappa^{*}
\delta f = \kappa^{*} \gamma^{*} \varpi =0$, this tells us that
$\kappa^{*} f$ is a cocycle. If we have two equivalent additively
split crossed modules then we can choose sections in such a way that
the associated cocycles are the same. Therefore we have a
well-defined map:
\[ \xymatrix{ACross(\gamma: C_{0} \rightarrow A ,M) \ar[r] & H^{3}_{Harr}(\gamma: C_{0} \rightarrow A ,M).}\]

Inversely, assume we have a cocycle in $C^{2}_{Harr}(\gamma: C_{0} \rightarrow A, M)$ which we lift to a cochain $f \in C^{2}_{Harr}(C_{0},M)$. Let $V = \Ker \gamma $. We define $C_{1} = M \times V$ as a module over $k$ with the following action of $C_{0}$ on $C_{1}$:
\[c(m,v) := (cm + f(c,v), cv).\]
It is easy to check using the properties of $f$ that this action is well defined and together with the map $\rho: C_{0} \rightarrow C_{1}$ given by $\rho (m,v) = v$, we have an additively split crossed module of $A$ by $M$.
\end{proof}

\begin{lemma}
If $k$ is a field of characteristic $0$ then \[Harr^{3}(A,M) \cong
\pi_{0} ACross(A,M).\]
\end{lemma}

\begin{proof}
From the definition of $C^{*}_{Harr}(\gamma: C_{0} \rightarrow A,M)$ we get the long exact sequence:
\begin{equation}
\xymatrix{ \ldots \ar[r] & Harr^{2}(A,M) \ar[r] & Harr^{2}(C_{0},M)
\ar[r] &  \\ & & Harr^{2}(\gamma: C_{0} \rightarrow A, M) \ar[r] &
Harr^{3}(A,M) \ar[r] & \ldots}\end{equation} Given any additively
split crossed module in $\pi_{0} ACross(A,M)$,
\[\xymatrix{0 \ar[r] & M \ar[r]^{\alpha} & C_{1} \ar[r]^{\rho} & C_{0} \ar[r]^{\gamma} & A \ar[r] & 0}\]
we can lift $\gamma$ to get a map $P_{0} \rightarrow A$ where
$P_{0}$ is a polynomial algebra. We can then use a pullback to
construct $P_{1}$ to get a crossed module where the following
diagram commutes:
\[\xymatrix{0 \ar[r] & M \ar@{=}[d] \ar[r]^{\alpha} & C_{1}  \ar[r]^{\rho} & C_{0}  \ar[r]^{\gamma} & A \ar@{=}[d] \ar[r] & 0 \\ 0 \ar[r] & M \ar[r] & P_{1} \ar[u] \ar[r] & P_{0} \ar[u] \ar[r] & A \ar[r] & 0. }\]
Note that these two crossed modules are in the same connected
component of $\pi_{0} ACross(A,M)$. By considering the second
crossed module in the long exact sequence, we replace $C_{0}$ by
$P_{0}$ to get the new exact sequence:
\begin{equation}\label{lesrh}
\xymatrix{ 0 \ar[r] &   Harr^{2}(\gamma: P_{0} \rightarrow A, M)
\ar[r] &  Harr^{3}(A,M) \ar[r] & 0  }\end{equation} since
$Harr^{2}(P_{0},M)=0$ and $Harr^{3}(P_{0},M) =0$.

The exact sequence \ref{lesrh} tells us that every element in
$Harr^{3}(A,M)$ comes from an element in $Harr^{2}(\gamma: P_{0}
\rightarrow A, M)$ and the previous lemma tells us that this comes
from a crossed module in $\pi_{0} ACross(A,M)$. Therefore the map
$\pi_{0} ACross(A,M) \rightarrow Harr^{3}(A,M)$ is surjective.

Assuming we have two crossed modules which go to the same element  in $Harr^{3}(A,M)$,
\begin{equation}\label{Hxm1} \xymatrix{0 \ar[r] & M \ar[r]^{\alpha} & C_{1} \ar[r]^{\rho} & C_{0} \ar[r]^{\gamma} & A \ar[r] & 0,}\end{equation}
\begin{equation}\label{Hxm2} \xymatrix{0 \ar[r] & M \ar[r]^{\alpha '} & C'_{1} \ar[r]^{\rho '} & C'_{0} \ar[r]^{\gamma '} & A \ar[r] & 0.}\end{equation}
There exist morphisms
\[\xymatrix{0 \ar[r] & M \ar@{=}[d] \ar[r]^{\alpha} & C_{1}  \ar[r]^{\rho} & C_{0}  \ar[r]^{\gamma} & A \ar@{=}[d] \ar[r] & 0
\\ 0 \ar[r] & M  \ar@{=}[d] \ar[r] & P_{1} \ar[u] \ar[r] & P_{0} \ar@{=}[d] \ar[u] \ar[r] & A \ar@{=}[d] \ar[r] & 0
\\ 0 \ar[r] & M \ar[r] & P_{2} \ar[d] \ar[r] & P_{0} \ar[d] \ar[r] & A
\ar[r] & 0 \\ 0 \ar[r] & M \ar@{=}[u] \ar[r]^{\alpha '} & C'_{1}
\ar[r]^{\rho '} & C'_{0}  \ar[r]^{\gamma '} & A \ar@{=}[u] \ar[r] &
0
 }\] where $P_{0}$ is a polynomial algebra and $P_{1},P_{2}$ are constructed via pullbacks. These give us
two elements in $Harr^{2}(\gamma: P_{0}\rightarrow A,M)$ which go to
the same element in $Harr^{3}(A,M)$. However the exact sequence
\ref{lesrh} tells us that the two crossed modules \ref{Hxm1} and
\ref{Hxm2} have to go to the same element in $Harr^{2}(\gamma:
P_{0}\rightarrow A,M)$. The previous lemma tells us that the two
crossed modules \ref{Hxm1} and \ref{Hxm2} go to the same element in
$ACross(\gamma: C_{0} \rightarrow A ,M)$ which is a groupoid, so
there is a map $P_{2} \rightarrow P_{1}$ which makes the following
diagram commute:
\[\xymatrix{0 \ar[r] & M \ar@{=}[d] \ar[r]^{\alpha} & C_{1}  \ar[r]^{\rho} & C_{0}  \ar[r]^{\gamma} & A \ar@{=}[d] \ar[r] & 0
\\ 0 \ar[r] & M  \ar@{=}[d] \ar[r] & P_{1} \ar[u] \ar[r] & P_{0} \ar@{=}[d] \ar[u] \ar[r] & A \ar@{=}[d] \ar[r] & 0
\\ 0 \ar[r] & M \ar[r] & P_{2} \ar[u] \ar[d] \ar[r] & P_{0} \ar[d] \ar[r] & A
\ar[r] & 0 \\ 0 \ar[r] & M \ar@{=}[u] \ar[r]^{\alpha '} & C'_{1}
\ar[r]^{\rho '} & C'_{0}  \ar[r]^{\gamma '} & A \ar@{=}[u] \ar[r] &
0
 }\]
Therefore the two crossed modules \ref{Hxm1} and \ref{Hxm2} are in
the same connected component of $\pi_{0} ACross(A,M)$ and the map
$\pi_{0} ACross(A,M) \rightarrow Harr^{3}(A,M)$ is injective.
\end{proof}

\section{Baues-Wirsching cohomology}
The following material can be found in \cite{BW}. A category $\mathcal{I}$ is said to be \textit{small} if the collection of morphisms is a set. Consider a small category
$\mathcal{I}$. The \emph{category of factorizations} in
$\mathcal{I}$, denoted by $\mathcal{FI}$, is the category whose
objects are the morphisms $f,g,...$ in $\mathcal{I}$, and morphisms
$f\rightarrow g$ are pairs $(\alpha, \beta)$ of morphisms in
$\mathcal{I}$ such that the following diagram commutes.
\[
\xymatrix{ B \ar[r]^{\alpha}  & B'
\\
A \ar[u]^{f} & A' \ar[l]^{\beta} \ar[u]_{g} }
\]
Composition in $\mathcal{FI}$ is given by
$(\alpha',\beta')(\alpha,\beta) = (\alpha' \alpha, \beta \beta')$. A
\emph{natural system} of abelian groups on $\mathcal{I}$ is a
functor \[D: \mathcal{FI} \rightarrow \mathfrak{Ab}.\]
There exists a canonical functor $\mathcal{FI} \rightarrow \mathcal{I}^{op} \times \mathcal{I}$ which takes $f:A \rightarrow B$ to the pair $(A, B)$. This functor allows us to consider any bifunctor $D: \mathcal{I}^{op}\times \mathcal{I} \rightarrow \mathfrak{Ab}$ as a natural system. Similarly, the projection $\mathcal{I}^{op}\times \mathcal{I} \rightarrow \mathcal{I}$ gives us the functor $\mathcal{FI} \rightarrow \mathcal{I}$ which takes $f:A \rightarrow B$ to $B$. This allows us to consider any functor $D: \mathcal{I} \rightarrow \mathfrak{Ab}$ as a natural system.

Following Baues-Wirsching \cite{BW}, we
define the cohomology $H_{BW}^{*}(\mathcal{I},D)$ of $\mathcal{I}$
with coefficients in the natural system $D$ as the cohomology of the
cochain complex $C^{*}_{BW}(\mathcal{I},D)$ given by
\[
C_{BW}^{n}(\mathcal{I},D) = \prod_{ \alpha_{1}\ldots\alpha_{n}:i_{n}
\rightarrow \ldots \rightarrow i_{0}} D(\alpha_{1}\ldots\alpha_{n}),
\]
where the product is indexed over $n$-tuples of composable morphisms and the coboundary map \[d:C_{BW}^{n}(\mathcal{I},D) \rightarrow
C_{BW}^{n+1}(\mathcal{I},D),\] is given by
\begin{align*}
(df)(\alpha_{1}\ldots\alpha_{n+1}) =&
(\alpha_{1})_{*}f(\alpha_{2},\ldots,\alpha_{n+1}) \\  &+
\sum_{j=1}^{n}(-1)^{j}f(\alpha_{1},\ldots,\alpha_{j}\alpha_{j+1},\ldots,\alpha_{n+1})
\\  &+ (-1)^{n+1}(\alpha_{n+1})^{*}f(\alpha_{1},\ldots,\alpha_{n}).
\end{align*}

\begin{lemma}\label{natsyslma}
Let $i_{0} \in \mathcal{I}$ be an initial object and $F: \mathcal{I} \rightarrow \mathfrak{Ab}$ a functor then
\[H_{BW}^{n}(\mathcal{I},F) = \left\{ \begin{array}{cc}
                     F(i_{0}) & \textrm{for } n=0 \\
                       0 & \textrm{for } n>0. \\
                      \end{array} \right.\]
\end{lemma}

\chapter{$\Psi$-rings}
\label{Chapter3}

\section{Introduction}
In this chapter, only the material in this section is already known and everything from section \ref{psimodsect} onwards is new and original material. Note that in all of the cited material, including \cite{AA}, \cite{Knut} and \cite{Yau}, what the authors call a $\Psi$-ring is what we call a special $\Psi$-ring. Also note that in our notation $\mathbb{N}$ does not include $0$.

$\lambda$-rings are complicated, and given a $\lambda$-ring it is
often difficult to prove it satisfies the $\lambda$-ring axioms. We
start by introducing another kind of ring, the $\Psi$-rings, which
are closely related to the $\lambda$-rings by the Adams operations.
The axioms for the $\Psi$-rings are a lot simpler than those for the
$\lambda$-rings.

\begin{definition}\label{psi1}
A \textit{$\Psi$-ring} is a commutative ring with identity, $R$, together
with a sequence of ring homomorphisms $\Psi^{i}: R \rightarrow R$,
for $i \in \mathbb{N}$, satisfying
\begin{enumerate}
    \item $\Psi^{1}(r) = r$,
    \item $\Psi^{i}( \Psi^{j}(r)) = \Psi^{ij}(r)$,
\end{enumerate}
for all $r \in R$, and $i,j \in \mathbb{N}$.
\end{definition}
We say that a $\Psi$-ring $R$ is \textit{special} if it also satisfies the property
\[\Psi^{p}(r) \equiv r^{p} \textrm{\qquad mod $pR$}, \]
for all primes $p$ and $r \in R$.

\begin{example}
Any commutative ring with identity, $R$, can be given a $\Psi$-ring
structure by setting $\Psi^{i}:R \rightarrow R$ to be
$\Psi^{i}(r)=r$ for all $r \in R$ and $i \in \mathbb{N}$.
\end{example}

Let $R_{1}, R_{2}$ be $\Psi$-rings. A \textit{map of $\Psi$-rings}
is a ring homomorphism $f: R_{1} \rightarrow R_{2}$, such that
$\Psi^{i}(f(r)) = f(\Psi^{i}(r))$ for all $r \in R_{1}$ and $i \in
\mathbb{N}$. The class of all $\Psi$-rings and maps of $\Psi$-rings
form the category of $\Psi$-rings, which we denote by
$\Psi\mathfrak{-rings}$.

\section{$\Psi$-modules}\label{psimodsect}
For usual rings, the modules provide us with the
coefficients for the cohomology. In this section we define the $\Psi$-modules
for $\Psi$-rings which provide us with the coefficients for the
$\Psi$-ring cohomology. We then use this to create the $\Psi$-analogue of some of the results for rings.

\begin{definition}
We say that $M$ is a \textit{$\Psi$-module} over the $\Psi$-ring $R$ if $M$ is an
$R$-module together with a sequence of abelian group homomorphisms $\psi^{i}:M
\rightarrow M$, for $i \in \mathbb{N}$, satisfying
\begin{enumerate}
    \item $\psi^{1}(m) = m $,
    \item $\psi^{i}(rm) = \Psi^{i}(r)\psi^{i}(m)$,
    \item $\psi^{i}(\psi^{j}(m))  = \psi^{ij}(m)$,
\end{enumerate}
for all $m \in M$, $r \in R$, and $i,j \in \mathbb{N}$.

Let $M,N$ be two $\Psi$-modules over $R$. A map of $\Psi$-modules is a module homomorphism $f:M \rightarrow N$ such that $\psi^{i} f(m) =f \psi^{i}(m)$ for all $m \in M$ and $i \in \mathbb{N}$. We let $R\mathfrak{-mod}_{\Psi}$ denote the category of all $\Psi$-modules over $R$.
\end{definition}

We say that $M$ is \textit{special} if $R$ is special and
\[\psi^{p}(m) \equiv 0 \textrm{\qquad mod $pM$}, \]
for all primes $p$ and $m \in M$.

Note that any $\Psi$-ring can be considered as a $\Psi$-module over itself.
Also note that if $M$ is special, then
$\psi^{i}(m) \equiv 0$ mod $iM$ for all $i \in \mathbb{N}$ and $m \in M$.

For the rest of this chapter, we let $R$ denote a $\Psi$-ring and $M
\in R\mathfrak{-mod}_{\Psi}$. We let $\underline{R}$ denote the
underlying commutative ring of $R$, and we let $\underline{M}$
denote the underlying $\underline{R}$-module of $M$.

\begin{lemma}\label{lmmapsimd1} The set $R \rtimes M$ with
\begin{align*} (r,m) + (s,n) &= (r+s, m+n),
\\ (r,m)(s,n) &= (rs, rn + ms),
\end{align*}
together with maps $\Psi^{i}: R \rtimes M \rightarrow R \rtimes M$
for $i \in \mathbb{N}$ given by \[ \Psi^{i} (r,m) = (\Psi^{i}(r),
\psi^{i}(m) + \varepsilon^{i}(r)),\] for a sequence of maps
$\varepsilon^{i}: R \rightarrow M$ for $i \in \mathbb{N}$, is a
$\Psi$-ring if and only if
\begin{enumerate}
\item $\varepsilon^{1}(r) = 0$,
\item $\varepsilon^{i}(r + s) = \varepsilon^{i}(r) + \varepsilon^{i}(s)$,
\item $\varepsilon^{i}(r s) = \Psi^{i}(r)\varepsilon^{i}(s) +
\varepsilon^{i}(r)\Psi^{i}(s)$,
\item $\varepsilon^{ij}(r) = \psi^{i}\varepsilon^{j}(r) + \varepsilon^{i}\Psi^{j}(r)$,
\end{enumerate}
for all $r,s \in R$, and $i,j \in \mathbb{N}$.
\end{lemma}

\begin{proof}[Proof of lemma]
It is known that $\underline{R} \rtimes \underline{M}$ is a
commutative ring with identity. Hence it is sufficient
 to check that $\Psi^{i}: R \rtimes M \rightarrow R \rtimes M$ satisfies the $\Psi$-ring axioms.

\begin{enumerate}
    \item $\Psi^{1}(r,m) = (\Psi^{1}(r),\psi^{1}(m) +
    \varepsilon^{1}(r)) = (r,m + \varepsilon^{1}(r))$ \\Hence $\Psi^{1}(r,m)= (r,m)$ if and only if
    $\varepsilon^{1}(r)$ = 0.

    \item $\Psi^{i}((r,m)+(s,n)) = (\Psi^{i}(r)+\Psi^{i}(s),\psi^{i}(m)+\psi^{i}(n) +  \varepsilon^{i}(r+s))$
    \\ $\Psi^{i}(r,m)+ \Psi^{i}(s,n) = (\Psi^{i}(r)+\Psi^{i}(s),\psi^{i}(m)+\psi^{i}(n) +  \varepsilon^{i}(r) +  \varepsilon^{i}(s))$
    \\Hence $\Psi^{i}((r,m)+(s,n)) = \Psi^{i}(r,m)+ \Psi^{i}(s,n)$
    if and only if \\$\varepsilon^{i}(r + s) = \varepsilon^{i}(r) +
    \varepsilon^{i}(s)$.

        \item $\Psi^{i}((r,m)(s,n)) = (\Psi^{i}(rs),\psi^{i}(rn+ms) + \varepsilon^{i}(rs))
    \\ \Psi^{i}(r,m)\Psi^{i}(s,n) = (\Psi^{i}(rs),\psi^{i}(rn + ms) + \Psi^{i}(r)\varepsilon^{i}(s) +
    \varepsilon^{i}(r)\Psi^{i}(s))$
    \\Hence $\Psi^{i}((r,m)(s,n)) = \Psi^{i}(r,m)\Psi^{i}(s,n)$ if and only if \\$\varepsilon^{i}(rs) = \Psi^{i}(r)\varepsilon^{i}(s) +
    \varepsilon^{i}(r)\Psi^{i}(s)$.

        \item $\Psi^{i}\Psi^{j}(r,m) = (\Psi^{i}\Psi^{j}(r),\psi^{i}\psi^{j}(m) + \psi^{i}\varepsilon^{j}(r) + \varepsilon^{i}\Psi^{j}(r))
    \\ \Psi^{ij}(r,m) = (\Psi^{i}\Psi^{j}(r),\psi^{i}\psi^{j}(m) + \varepsilon^{ij}(r))$
    \\Hence $\Psi^{i}\Psi^{j}(r,m) = \Psi^{ij}(r,m)$ if and only if
    $\varepsilon^{ij}(r) = \psi^{i}\varepsilon^{j}(r)) +
    \varepsilon^{i}\Psi^{j}(r)$.
\end{enumerate}\end{proof}

The maps $\varepsilon^{i}: R \rightarrow M$ given by
$\varepsilon^{i}(r) = 0$, for all $r \in R$ and $i \in \mathbb{N}$, satisfy  properties
\ref{lmmapsimd1}.1-\ref{lmmapsimd1}.4 meaning that the maps $\Psi^{i}: R
\rtimes M \rightarrow R \rtimes M$ given by $\Psi^{i}(r,m) =
(\Psi^{i}(r),\psi^{i}(m))$ give us a $\Psi$-ring structure on $R
\rtimes M$. We call this the \textit{semi-direct product} of $R$ and
$M$, denoted by $R \rtimes_{\Psi} M$.

We note that if $R$ and $M$ are both special, then $R \rtimes_{\Psi}
M$ is also special.

\section{$\Psi$-derivations}
The Andr\'{e}-Quillen cohomology for commutative rings is given by
the derived functors of the derivations functor. For a
commutative ring $S$, the derivations of $S$ with values in an
$S$-module $N$ are in one-to-one correspondence with the sections of
$\xymatrix{ S \rtimes N \ar[r]^-{\pi} & S }$. We define the
$\Psi$-derivations and show that they are in one-to-one correspondence
with the sections of $\xymatrix{ R \rtimes_{\Psi} M \ar[r]^-{\pi} &
R }$.

\begin{definition}
A \textit{$\Psi$-derivation} of $R$ with values in $M$ is an additive
homomorphism $d:R \rightarrow M$ such that
\begin{enumerate}
\item $d(r s) = r d(s) + d(r)s$,
\item $\psi^{i}(d(r)) = d(\Psi^{i}(r))$,
\end{enumerate}
for all $r, s \in R,$ and $i \in \mathbb{N}$. We let
$\Der_{\Psi}(R,M)$ denote the set of all $\Psi$-derivations of $R$
with values in $M$.
\end{definition}

\begin{example}
Let $R$ and $M$ be such that $\Psi^{i}=Id = \psi^{i}$ for all $i \in \mathbb{N}$, then
\[\Der_{\Psi}(R,M)= \Der(\underline{R},\underline{M}).\]
\end{example}

\begin{theorem}
There is a one-to-one correspondence between the elements of
$\Der_{\Psi}(R,M)$ and the sections of $\xymatrix{ R \rtimes_{\Psi}
M \ar[r]^-{\pi} & R }$.
\end{theorem}

\begin{proof}[Proof of theorem]
Assume we have a section of $\pi$, then we have the following \[
\xymatrix{ R \rtimes_{\Psi} M \ar@<1ex>[r]^-{\pi} & R,
\ar@<1ex>[l]^-{\sigma} }\] where $\pi \sigma  = Id_{R}$. Hence
$\sigma(r) = (r,d(r))$ for some $d:R \rightarrow M$. The properties
\begin{align*}d(r+s) =& d(r) + d(s), \\ \qquad d(rs) =& d(r)s + r
d(s),\end{align*} follow from $\sigma$ being a ring homomorphism.
However $\sigma$ also preserves the $\Psi$-ring structure, so we get
that $\Psi^{i}\sigma(r) = \sigma\Psi^{i}(r)$. We know that
\[\Psi^{i}\sigma(r) = \Psi^{i}(r, d(r)) = (\Psi^{i}(r),\psi^{i}(d(r))), $$  $$\sigma\Psi^{i}(r) = (\Psi^{i}(r),d(\Psi^{i}(r))).
\] Hence $\Psi^{i}\sigma(r) = \sigma\Psi^{i}(r)$ if and only if
$\psi^{i}d(r) = d\Psi^{i}(r)$. This tells us that if $\sigma$ is a
section of $\pi$, then we have a $\Psi$-derivation $d$.

Conversely, if we have a $\Psi$-derivation $d:R \rightarrow M$, then
$\sigma(r) = (r,d(r))$ is a section of $\pi$.
\end{proof}

\section{$\Psi$-ring extensions}
We have seen in proposition \ref{AQprop1} that the Andr\'{e}-Quillen cohomology
$H^{1}_{AQ}(\underline{R},\underline{M})$ classifies the extensions
of \underline{R} by \underline{M}. In this section, we develop the
$\Psi$-analogue of extensions.

\begin{definition}
A \textit{$\Psi$-ring extension} of $R$ by $M$ is an extension of $\underline{R}$ by $\underline{M}$
\[\xymatrix{0 \ar[r] & M \ar^{\alpha}[r] & X \ar^{\beta}[r] & R \ar[r] & 0}\]
where $X$ is a $\Psi$-ring, $\beta$ is a map of $\Psi$-rings and $\alpha \psi^{n} = \Psi^{n} \alpha$ for all $n \in \mathbb{N}$.
\end{definition}

Two $\Psi$-ring extensions $(X),(\overline{X})$
with $R,M$ fixed are said to be \textit{equivalent} if there exists a map of
$\Psi$-rings $\phi: X \rightarrow \overline{X}$ such that the following
diagram commutes.
\[\xymatrix{0 \ar[r] & M \ar@{=}[d] \ar[r] & X \ar[r] \ar[d]^{\phi} & R \ar[r] \ar@{=}[d] & 0 \\ 0 \ar[r] & M \ar[r] & \overline{X} \ar[r] & R \ar[r] & 0}\]

We denote the set of equivalence classes of $\Psi$-ring extensions
of $R$ by $M$ by $\Ext_{\Psi}(R,M)$.

The Harrison cohomology $Harr^{2}(\underline{R},\underline{M})$ classifies the additively split extensions
of $\underline{R}$ by $\underline{M}$. We can also define the $\Psi$-analogue of these types of extensions.

\begin{definition}
An \textit{additively split $\Psi$-ring extension} of $R$ by $M$ is a $\Psi$-ring extension of $R$ by $M$
\[\xymatrix{0 \ar[r] & M \ar^{\alpha}[r] & X \ar^{\beta}[r] & R \ar[r] & 0}\] where
$\beta$ has a section which is an additive homomorphism.
\end{definition}

Multiplication in $X=R\oplus M$ has the form $(r,m)(r',m') = (rr', mr' + rm'
+ f(r,r'))$, where $f:R \times R \rightarrow M$ is some bilinear
map. Associativity in $X$ gives us
\[0 = rf(r',r'')-f(rr',r'') + f(r,r'r'')-f(r,r')r''.\]
Commutativity in $X$ gives us
\[f(r,r') = f(r',r).\]

The $\Psi$-operations $\Psi^{i}:R \oplus M \rightarrow R
\oplus M$ for $i \in \mathbb{N}$ are given by $\Psi^{i}(r,m) =
(\Psi^{i}(r), \psi^{i}(m) + \varepsilon^{i}(r) )$ for a sequence of
operations $\varepsilon^{i}: R \rightarrow M$ which satisfy the following
properties
\begin{enumerate}
    \item $\varepsilon^{1}(r) = 0$,
    \item $\varepsilon^{i}(r+s) = \varepsilon^{i}(r)+\varepsilon^{i}(s)$,
    \item $\varepsilon^{i}(rs) = \Psi^{i}(s)\varepsilon^{i}(r) + \Psi^{i}(r)\varepsilon^{i}(s) + f(\Psi^{i}(r),\Psi^{i}(s))-\psi^{i}(f(r,s))$,
    \item $\varepsilon^{ij}(r) = \psi^{i}\varepsilon^{j}(r) + \varepsilon^{i}\Psi^{j}(r)$,
\end{enumerate}
for all $r,s \in R$ and $i,j \in \mathbb{N}$.

Assuming we have two $\Psi$-ring extensions
$(X,\varepsilon,f),(\overline{X},\overline{\varepsilon},\overline{f})$ which are equivalent,
together with a $\Psi$-ring map $\phi:X \rightarrow \overline{X}$ with
$\phi(r,m) = (r, m + g(r))$ for some $g: R \rightarrow M$. We have
that $\phi$ being a homomorphism tells us that
\[g(r+r') = g(r) + g(r'),\]
\[f(r,r')- \overline{f}(r,r') = rg(r') - g(rr') + g(r)r'.\]
We also have $\phi(\Psi^{i})= \overline{\Psi}^{i}(\phi)$ for all $i \in \mathbb{N}$, which
tells us that \[\varepsilon^{i}(r)-\overline{\varepsilon}^{i}(r) =
\psi^{i}(g(r)) - g(\Psi^{i}(r)).\]

We denote the set of equivalence classes of the additively split $\Psi$-ring extensions
of $R$ by $M$ by $\AExt_{\Psi}(R,M)$.

\begin{definition} \label{pmextdef}
An \textit{additively and multiplicatively split $\Psi$-ring extension} of $R$ by $M$ is a $\Psi$-ring extension of $R$ by $M$
\[\xymatrix{0 \ar[r] & M \ar^{\alpha}[r] & X \ar^{\beta}[r] & R \ar[r] & 0}\]
where $\beta$ has a section which is an additive and multiplicative homomorphism.
\end{definition}

As a commutative ring $X = R \rtimes M$, i.e. $f=0$ above. The $\Psi$-operations $\Psi^{i}:X \rightarrow X$ for $i \in \mathbb{N}$ are given by
$\Psi^{i}(r,m) = (\Psi^{i}(r), \psi^{i}(m) + \varepsilon^{i}(r) )$
for a sequence of operations $\varepsilon^{i}: R \rightarrow M$ such
that
\begin{enumerate}
    \item $\varepsilon^{1}(r) = 0$,
    \item $\varepsilon^{i}(r+s) =
    \varepsilon^{i}(r)+\varepsilon^{i}(s)$,
    \item $\varepsilon^{i}(rs) = \Psi^{i}(s)\varepsilon^{i}(r) + \Psi^{i}(r)\varepsilon^{i}(s)$,
    \item $\varepsilon^{ij}(r) = \psi^{i}\varepsilon^{j}(r) + \varepsilon^{i}\Psi^{j}(r)$,
\end{enumerate}
for all $r,s \in R$ and $i,j \in \mathbb{N}$. Note that conditions 2 and 3 tell us that $\varepsilon^{i} \in \Der(\underline{R}, \underline{M}^{i})$ where $M^{i}$ denotes the $\Psi$-module over $R$ with $M$ as an abelian group and the action of $R$ given by $(r,m) \mapsto \Psi^{i}(r)m$, for  $r \in R, m \in M$.

Assume we have two additively and multiplicatively split $\Psi$-ring extensions
$(X,\varepsilon),(\overline{X},\overline{\varepsilon})$ which are equivalent, together
with a $\Psi$-ring map $\phi:X \rightarrow \overline{X}$ with $\phi(r,m) = (r,
m + g(r))$ for some $g: R \rightarrow M$. Since $\phi$ is a ring
homomorphism we get that $g \in \Der(\underline{R},\underline{M})$.
Since $\phi$ is a map of $\Psi$-rings we get that
\[\varepsilon^{i}(r)- \overline{\varepsilon}^{i}(r) = \psi^{i}(g(r)) -
g(\Psi^{i}(r)),\] for all $i \in \mathbb{N}$.

We denote the set of equivalence classes of the additively and multiplicatively split $\Psi$-ring extensions
of $R$ by $M$ by $\MExt_{\Psi}(R,M)$.

\begin{example}
Let $R$ and $M$ be such that $\Psi^{i}=Id = \psi^{i}$ for all $i \in \mathbb{N}$, then
\[\MExt_{\Psi}(R,M) \cong \prod_{\textrm{p prime}}\Der(\underline{R},\underline{M}).\]
\end{example}

\begin{lemma}
There exist exact sequences \[\xymatrix@C=1.5pc { 0 \ar[r] &
\MExt_{\Psi}(R,M) \ar[r]^-{w} & \Ext_{\Psi}(R,M) \ar[r]^-{u} &
H^{1}_{AQ}(\underline{R},\underline{M}) \ar[r] & \frac{H^{1}_{AQ}(\underline{R},\underline{M})}{Im(u)}\ar[r] & 0 }\]
\[\xymatrix@C=1.5pc{ 0 \ar[r] &
\MExt_{\Psi}(R,M) \ar[r]^-{w} & \AExt_{\Psi}(R,M) \ar[r]^-{u} &
Harr^{2}(\underline{R},\underline{M})  \ar[r] & \frac{Harr^{2}(\underline{R},\underline{M})}{Im(u)}\ar[r] & 0 }\]
where $w$ is the inclusion, and $u$ maps the class of a $\Psi$-ring extension to the class of its underlying extension.
\end{lemma}

\begin{proof}
We only need to check exactness at $\Ext_{\Psi}(R,M)$ and $\AExt_{\Psi}(R,M)$. A class in $\Ext_{\Psi}(R,M)$ or $\AExt_{\Psi}(R,M)$ belongs to the kernel of $u$ if the underlying class is the trivial class. The additively and multiplicatively split extensions are precisely the $\Psi$-ring extensions whose underlying extension is trivial. Exactness follows.
\end{proof}

From the definitions, we see that $\Ext_{\Psi}(R,M) \supseteq \AExt_{\Psi}(R,M) \supseteq \MExt_{\Psi}(R,M)$.

If $R$ and $M$ are both special, then we say that a $\Psi$-ring extension
\[\xymatrix{0 \ar[r] & M \ar^{\alpha}[r] & X \ar^{\beta}[r] & R \ar[r] & 0}\]
is \textit{special} if $X$ is also special.

We denote the set of equivalence classes of the special $\Psi$-ring extensions
of $R$ by $M$ by $\Ext_{\Psi_{s}}(R,M)$. Similarly, we can define $\AExt_{\Psi_{s}}(R,M)$ and $\MExt_{\Psi_{s}}(R,M)$.

\section{Crossed $\Psi$-extensions}
A \textit{crossed $\Psi$-module} consists of a $\Psi$-ring $C_{0}$, a $\Psi$-module $C_{1}$ over $C_{0}$ and a map of $\Psi$-modules
\[\xymatrix{C_{1} \ar^{\partial}[r] & C_{0}},\]
which satisfies the property
\[\partial(c)c' = c \partial(c'),\]
for $c,c' \in C_{1}$. In other words, a crossed $\Psi$-module is a chain algebra which is non-trivial only in dimensions 0 and 1. Since $C_{2}=0$ the condition
$\partial(c)c' = c \partial(c')$ is equivalent to the Leibnitz relation \[0  = \partial (cc') = \partial(c)c' - c \partial(c'). \]
We can define a product by \[ c * c' := \partial(c) c', \]
for $c,c' \in C_{1}$. This gives us a $\Psi$-ring structure on $C_{1}$ and $\partial: C_{1} \rightarrow C_{0}$ is a map of $\Psi$-rings.

Let $\partial : C_{1} \rightarrow C_{0}$ be a crossed $\Psi$-module. We let $M= \Ker (\partial)$ and $R= \Coker (\partial)$ Then the image $\Image (\partial)$ is an ideal of $C_{0}$, $M C_{1} = C_{1}M = 0$ and $M$ has a well-defined $\Psi$-module structure over $R$.

A \textit{crossed $\Psi$-extension} of $R$ by $M$ is an exact sequence
\[\xymatrix{0 \ar[r] & M \ar[r]^{\alpha} & C_{1} \ar[r]^{\partial} & C_{0} \ar[r]^{\gamma} & R \ar[r] & 0}\]
where $\partial : C_{1} \rightarrow C_{0}$ is a crossed $\Psi$-module, $\gamma$ is a map of $\Psi$-rings, and the $\Psi$-module structure on $M$ coincides with the one induced from the crossed $\Psi$-module. We denote the category of crossed $\Psi$-extensions of $R$ by $M$ by $Cross_{\Psi}(R,M)$. We let $\pi_{0} Cross_{\Psi}(R,M)$ denote the connected components of the category $Cross_{\Psi}(R,M)$.

An \textit{additively split crossed $\Psi$-extension} of $R$ by $M$ is a
crossed $\Psi$-extension \begin{equation}\label{pcesasc} \xymatrix{0
\ar[r] & M \ar^{\omega}[r] & C_{1} \ar^{\rho}[r] & C_{0}
\ar^{\pi}[r] &  R \ar[r] & 0}\end{equation} such that all the arrows
in the exact sequence $\ref{pcesasc}$ are additively split.
We denote the connected components of the category of additively split crossed
$\Psi$-extensions of $R$ by $M$ by $\pi_{0} ACross_{\Psi}(R,M)$.

An \textit{additively and multiplicatively split crossed $\Psi$-extension} of $R$ by $M$ is a
crossed $\Psi$-extension \[ \xymatrix{0
\ar[r] & M \ar^{\omega}[r] & C_{1} \ar^{\rho}[r] & C_{0}
\ar^{\pi}[r] &  R \ar[r] & 0}\] such that $\pi$ is additively and multiplicatively split.
We denote the connected components of the category of additively and multiplicatively split crossed
$\Psi$-extensions of $R$ by $M$ by $\pi_{0} MCross_{\Psi}(R,M)$.


\section{Deformation of $\Psi$-rings}\label{Psi deformation}
In this section, we apply Gerstenhaber and Schack's definition of a deformation of a diagram of algebras \cite{Gerdiag} to the case of $\Psi$-rings.

\begin{definition}
Let
\[\alpha_{t} = \alpha_{0} + t\alpha_{1} +
t^{2}\alpha_{2} + \ldots \]
be a deformation of $\underline{R}$, i.e. be a formal power series, in which each $\alpha_{k}: R \times R \rightarrow R$ is a bilinear map, $\alpha_{0}$ is the multiplication in $R$ and $\alpha_{t}$ is associative and commutative.

For each $i \in \mathbb{N}$, let \[\Psi^{i}_{t} = \psi^{i}_{0} + t\psi^{i}_{1} +
t^{2}\psi^{i}_{2} + \ldots \] be a formal power series, in which
each $\psi^{i}_{k}$ is a function \[\psi^{i}_{k} : R
\rightarrow R,\] satisfying \begin{enumerate}
                                     \item $\psi^{i}_{0}(r) = \Psi^{i}(r)$,
                                     \item $\psi^{1}_{k}(r) = 0$,
                                     \item $\psi^{i}_{k}(r+s) = \psi^{i}_{k}(r) + \psi^{i}_{k}(s) $,
                                     \item $\sum_{h=0}^{k}\psi^{i}_{h}\alpha_{k-h}(r,s) = \sum_{h=0}^{k}\sum_{l=0}^{k-h}\alpha_{h}(\psi^{i}_{l}(r),\psi^{i}_{k-h-l}(s))$,
                                     \item $\psi^{ij}_{k}(r)=\sum_{l=0}^{k}\psi^{i}_{l}\circ\psi^{j}_{k-l}(r),$
                                   \end{enumerate}
for all $i,j,k \in \mathbb{N}$ and $r,s \in R$. We call $(\alpha_{t},\Psi^{*}_{t})$ a $\Psi$-\textit{ring deformation} of $R$.
\end{definition}

We call $(\alpha_{1},\psi^{*}_{1})$ the \textit{infinitesimal deformation} of $(\alpha_{t},\Psi^{*}_{t})$. The
infinitesimal $\Psi$-ring deformation $(\alpha_{1},\psi^{*}_{1})$ is identified with the
additively split $\Psi$-ring extensions of $R$ by $R$ by setting $f = \alpha_{1}$ and $\varepsilon^{i} =
\psi^{i}_{1}$ for all $i \in \mathbb{N}$.

\begin{definition}
We define a \textit{formal automorphism} of the $\Psi$-ring $R$ to
be a formal power series \[\Phi_{t} = \phi_{0} + t\phi_{1} +
t^{2}\phi_{2} + \ldots\] where each $\phi_{k}: R \rightarrow R$ such
that \begin{enumerate}
       \item $\phi_{0}(r)= r$,
       \item $\phi_{k}(r+s)=\phi_{k}(r)+\phi_{k}(s).$
     \end{enumerate}
Two $\Psi$-ring deformations $(\alpha_{t},\Psi^{*}_{t})$ and $(\overline{\alpha}_{t},\overline{\Psi}^{*}_{t})$ are
\textit{equivalent} if there exists a formal automorphism $\Phi_{t}$
such that $\Phi_{t} \alpha_{t}(r,s) = \overline{\alpha}_{t}(\Phi_{t}r,\Phi_{t}s) $  and $\Phi_{t}\Psi^{*}_{t} = \overline{\Psi}^{*}_{t}\Phi_{t}$.
\end{definition}

If two $\Psi$-ring deformations $(\alpha_{t},\Psi^{*}_{t})$ and $(\overline{\alpha}_{t},\overline{\Psi}^{*}_{t})$ are
equivalent, then the differences satisfy $\alpha_{1}(r,s)-\overline{\alpha}_{1}(r,s)= r\phi_{1}(s) - \phi_{1}(rs) + s\phi_{1}(r)$  and $\psi^{i}_{1} -
\overline{\psi}^{i}_{1} = \overline{\Psi^{i}}\phi_{1} -
\phi_{1}\Psi^{i}$ for all $i \in \mathbb{N}$. Hence the
equivalence classes of the infinitesimal $\Psi$-ring deformations are identified
with the equivalence classes of the additively split $\Psi$-ring extensions, $\AExt_{\Psi}(R,R)$.

Yau \cite{Yau} defined the cohomology of $\lambda$-rings in order to study
deformations with respect to the $\Psi$-operations corresponding to
the $\lambda$-ring. Here, I provide an alternative definition to Yau's definition. A deformation of the $\Psi$-operations should be a $\Psi$-ring deformation $(\alpha_{t},\Psi^{*}_{t})$ where $\alpha_{t}$ is the trivial deformation. If we let $\alpha_{k}=0$ for all $k \geq 1$ in the definition of a $\Psi$-ring deformation then we get the following definition.

\begin{definition}
For each $i \in \mathbb{N}$, let \[\Psi^{i}_{t} = \psi^{i}_{0} + t\psi^{i}_{1} +
t^{2}\psi^{i}_{2} + \ldots \] be a formal power series, in which
each $\psi^{i}_{k}$ is a function \[\psi^{i}_{k} : R
\rightarrow R,\] such that \begin{enumerate}
                                     \item $\psi^{i}_{0}(r) = \Psi^{i}(r)$,
                                     \item $\psi^{1}_{k}(r) = 0$ for
                                     $k \geq 1.$,
                                     \item $\psi^{i}_{k}(r+s) = \psi^{i}_{k}(r) + \psi^{i}_{k}(s) $,
                                     \item $\psi^{i}_{k}(rs) = \sum_{l=0}^{k}\psi^{i}_{l}(r)\psi^{i}_{k-l}(s)$,
                                     \item $\psi^{ij}_{k}(r)=\sum_{l=0}^{k}\psi^{i}_{l}\circ\psi^{j}_{k-l}(r)$,
                                   \end{enumerate}
for all $i,j,k \in \mathbb{N}$ and $r,s \in R$.                                    
We call $\Psi^{*}_{t}$ a $\Psi$-\textit{operation deformation} of $R$.
\end{definition}

The infinitesimal $\Psi$-operation deformation $\psi^{*}_{1}$ is identified with the
additively and multiplicatively split $\Psi$-ring extensions of $R$ by $R$ by setting $\varepsilon^{i} =
\psi^{i}_{1}$ for all $i \in \mathbb{N}$.

If two $\Psi$-operation deformations $\Psi^{*}_{t}$ and $\overline{\Psi}^{*}_{t}$ are
equivalent, then the difference  satisfies $\psi^{i}_{1} - \overline{\psi}^{i}_{1} = \overline{\Psi^{i}}\phi_{1} -
\phi_{1}\Psi^{i}$ for all $i \in \mathbb{N}$. Note that now $\Phi_{t}(rs)= \Phi_{t}(r)\Phi_{t}(s)$ so we get that $\phi_{1} \in \Der(\underline{R},\underline{R})$.  Hence the
equivalence classes of the infinitesimal $\Psi$-operation deformations are identified
with the equivalence classes of the additively and multiplicatively split $\Psi$-ring extensions, $\MExt_{\Psi}(R,R)$.

\chapter{$\lambda$-rings}
\label{Chapter4} 
\section{Introduction}
In this chapter, only the material in this section and section \ref{yaucosect} is already known (see \cite{AA}, \cite{Knut} and \cite{Yau}) and everything else is new and original material. Note that in our notation $\mathbb{N}_{0} = \mathbb{N} \cup \{0\}$.

In this chapter, we start by introducing the concept of a
pre-$\lambda$-ring. After giving the definition, we will look at
some examples of pre-$\lambda$-rings. Later, we introduce the
definition of $\lambda$-rings, which are pre-$\lambda$-rings which
satisfy some additional axioms. Then we will look at which of the
pre-$\lambda$-ring structures also give us $\lambda$-rings.

\begin{definition}\label{defprelam} A \textit{pre-$\lambda$-ring} is a commutative ring $R$ with identity $1$,
together with a sequence of operations $\lambda^{i}: R \rightarrow
R$, for $i \in \mathbb{N}_{0}$, satisfying
\begin{enumerate}
    \item $\lambda^{0}(r) = 1$,
    \item $\lambda^{1}(r) = r$,
    \item $\lambda^{i}(r+s)= \Sigma_{k=0}^{i} \lambda^{k}(r)
    \lambda^{i-k}(s)$,
\end{enumerate}
for all $r,s \in R$ and $i \in \mathbb{N}_{0}$.
\end{definition}

To be able to describe examples of pre-$\lambda$-rings or
$\lambda$-rings it is often useful to consider, for $r \in R$, the
formal power series in the variable $t$
\begin{align*}
\lambda_{t}(r) &= \sum_{i=0}^{\infty}\lambda^{i}(r) t^{i}\\ &=
\lambda^{0}(r) + \lambda^{1}(r)t + \lambda^{2}(r)t^{2}+ \ldots
\end{align*}

Note that \begin{align*}\lambda_{t}(r+s) &= \lambda^{0}(r+s) +
\lambda^{1}(r+s)t + \lambda^{2}(r+s)t^{2}+ \lambda^{3}(r+s)t^{3} \ldots
\\&= 1 + (r+s)t + \Sigma_{k=0}^{2} \lambda^{k}(r)
    \lambda^{2-k}(s) t^{2} + \Sigma_{k=0}^{3} \lambda^{k}(r)
    \lambda^{3-k}(s) t^{3}+\ldots \\&=(1 + rt +
\lambda^{2}(r)t^{2}+ \ldots)(1 + st + \lambda^{2}(s)t^{2}+ \ldots)\\ &=
\lambda_{t}(r)\lambda_{t}(s).\end{align*}

This gives us an equivalent definition of a pre-$\lambda$-ring.

\begin{definition}
A \textit{pre-$\lambda$-ring} is a commutative ring $R$ with
identity $1$, together with a sequence of operations $\lambda^{i}: R
\rightarrow R$, for $i \in \mathbb{N}_{0}$, satisfying
\begin{enumerate}
    \item $\lambda^{0}(r) = 1$,
    \item $\lambda^{1}(r) = r$,
    \item $\lambda_{t}(r+s)= \lambda_{t}(r) \lambda_{t}(s)$, where $\lambda_{t}(r) = \sum_{i \geq 0} \lambda^{i}(r)t^{i}$,
\end{enumerate}
for all $r,s \in R$ and $i \in \mathbb{N}_{0}$.
\end{definition}

\begin{example}\label{prelamex}
We can get a pre-$\lambda$-ring structure on $\mathbb{Z}$ by taking \[\lambda_{t}(r) = (1+t + n_{2}t^{2} + n_{3}t^{3}+ \ldots)^{r},\] where $1+ t +
n_{2}t^{2} + n_{3}t^{3}+\ldots$ is a power series with integer coefficients.

We can get a pre-$\lambda$-ring structure on $\mathbb{R}$ by taking either
\begin{enumerate}
  \item $\lambda_{t}(r) = (1+t + n_{2}t^{2} + n_{3}t^{3}+ \ldots)^{r}$, where $1+ t +
n_{2}t^{2} + n_{3}t^{3}+\ldots$ is a power series with integer coefficients, or
  \item $\lambda_{t}(r)= e^{tr}$.
\end{enumerate}
\end{example}

The $\lambda$-ring axioms involve some universal polynomials. We are now going to introduce the elementary symmetric functions in order to define these universal polynomials.

\begin{definition}
 Let $\xi_{1}, \xi_{2}, \ldots ,\xi_{q}; \eta_{1}, \eta_{2},\ldots,\eta_{r}$ be
indeterminates. Define $s_{i}$ and $\sigma_{j}$ to be the elementary
symmetric functions of the $\xi_{i}'s, \eta_{j}'s$, i.e.
\begin{displaymath} (1 + s_{1}t + s_{2}t^{2} + \ldots + ) = \Pi_{i}
(1 + \xi_{i}t),
\end{displaymath}
\begin{displaymath}
(1 + \sigma_{1}t + \sigma_{2}t^{2} + \ldots + ) = \Pi_{j} (1 +
\eta_{j}t).
\end{displaymath}
Let $P_{k}(s_{1},
s_{2},\ldots,s_{k};\sigma_{1},\sigma_{2},\ldots,\sigma_{k})$ be the
coefficient of $t^{k}$ in $\Pi_{i,j}(1 + \xi_{i}\eta_{j}t)$.
\\ Let $P_{k,l}(s_{1}, s_{2},\ldots,s_{kl})$ be the coefficient of $t^{k}$
in $\Pi_{1 \leq i_{1} < \ldots < i_{l} \leq q}(1 +
\xi_{i_{1}}\xi_{i_{2}}\ldots\xi_{i_{l}}t)$.
\end{definition}

\begin{example}See also appendix \ref{appb}.
\begin{itemize}
  \item $P_{1}(s_{1}; \sigma_{1}) = s_{1} \sigma_{1}$,
  \item $P_{2}(s_{1}, s_{2}; \sigma_{1},\sigma_{2}) = s_{1}^{2}\sigma_{2} - 2 s_{2} \sigma_{2} + s_{2} \sigma_{1}^{2}$,
  \item $P_{1,1}(s_{1}) = s_{1}$,
  \item $P_{1,2}(s_{1},s_{2}) = P_{2,1}(s_{1},s_{2}) = s_{2}$,
  \item $P_{2,2}(s_{1},s_{2},s_{3},s_{4}) = s_{1}s_{3} - s_{4}$.
\end{itemize}
\end{example}

\begin{definition}\label{deflam}
A \textit{$\lambda$-ring} is a commutative ring $R$ with identity 1, together with a sequence of operations $\lambda^{i}:R \rightarrow R$, for $i \in \mathbb{N}_{0}$, satisfying
\begin{enumerate}
    \item $R$ is a pre-$\lambda$-ring,
    \item $\lambda_{t}(1) = 1 + t$,
    \item $\lambda^{i}(rs) = P_{i}(\lambda^{1}(r),\lambda^{2}(r),\ldots,\lambda^{i}(r),\lambda^{1}(s),\ldots,\lambda^{i}(s))$,
    \item $\lambda^{i}(\lambda^{j}(r)) =
    P_{i,j}(\lambda^{1}(r),\ldots,\lambda^{ij}(r))$,
\end{enumerate}
for all $r,s \in R$ and $i,j \in \mathbb{N}_{0}$.
\end{definition}

Since $\lambda^{1}$ is the identity, it follows that $P_{k,1}(s_{1},\ldots,s_{k}) = P_{1,k}(s_{1},\ldots,s_{k}) = s_{k}$.
In general, $P_{k,j} \neq P_{j,k}$, so the $\lambda$-operations do not commute.

\begin{example}
The simplest example of a $\lambda$-ring is $\mathbb{Z}$, together
with binomial coefficients $\lambda^{i}(r) = \binom{r}{i}$. The
additional axioms for $\lambda$-rings eliminate the more exotic
pre-$\lambda$-ring structures. From \ref{prelamex}, the only
$\lambda$-rings are taking $\lambda_{t}(r)=(1+t)^{r}$, which gives
us a $\lambda$-ring structure on $\mathbb{Z}$ or $\mathbb{R}$.
\end{example}

\begin{corollary}[Some properties of $\lambda$-rings]
\begin{enumerate}
\item The characteristic of $R$ is zero.

\item $\lambda^{i}(1) = 0$ for $i \geq 2$.
\end{enumerate}
\end{corollary}

\begin{proof}[Proof of corollary]
\begin{enumerate}
\item Let $j$ be any integer. \\ $\lambda_{t}(j) = \lambda_{t}(\underbrace{1+1+
\ldots+1}_{\textrm{j times}}) = \lambda_{t}(1)^{j} = (1+ t)^{j} \neq
0$.

\item This follows from \ref{deflam}.1.
\end{enumerate}
\end{proof}

\textit{A map of $\lambda$-rings} $R_{1} \rightarrow R_{2}$, is a
ring homomorphism $f: R_{1} \rightarrow R_{2}$, such that
\\$\lambda^{i}(f(r)) = f(\lambda^{i}(r))$ for all $r \in R_{1}$ and
$i \in \mathbb{N}_{0}$. The class of all $\lambda$-rings and maps of
$\lambda$-rings form the category of $\lambda$-rings, which we
denote by $\lambda\mathfrak{-rings}$.

The $\lambda$-operations are often difficult to work with as they
are neither additive nor multiplicative. We can get ring maps from
the $\lambda$-operations, which are the Adams operations $\Psi^{i}:
R \rightarrow R$ for $i \in \mathbb{N}$, defined by the Newton
formula
\begin{displaymath}
\Psi^{i}(r) -
\lambda^{1}(r)\Psi^{i-1}(r)+\ldots+(-1)^{i-1}\lambda^{i-1}(r)\Psi^{1}(r)
+ (-1)^{i}i\lambda^{i}(r) = 0.
\end{displaymath}

\begin{example} See also appendix \ref{appa}.
\begin{align*}\Psi^{1}(r) &= \lambda^{1}(r),\\
\Psi^{2}(r) &= \Psi^{1}(r)\lambda^{1}(r) - 2\lambda^{2}(r),\\
\Psi^{3}(r) &= \Psi^{2}(r)\lambda^{1}(r) - \Psi^{1}(r)\lambda^{2}(r)
+ 3\lambda^{3}(r),\\
\Psi^{4}(r) &= \Psi^{3}(3)\lambda^{1}(r) - \Psi^{2}(r)\lambda^{2}(r)
+ \Psi^{1}(r)\lambda^{3}(r) - 4\lambda^{4}(r).
\\
& \vdots \end{align*} By rearranging and making substitutions we get
the following
\begin{align*}\Psi^{1}(r) &= \lambda^{1}(r) = r,
\\\Psi^{2}(r) &= r^{2} - 2\lambda^{2}(r), \\\Psi^{3}(r) &= r^{3} -
3r\lambda^{2}(r) + 3\lambda^{3}(r),\\
\Psi^{4}(r) &= r^{4} - 4r^{2}\lambda^{2}(r) + 4r\lambda^{3}(r) -
4\lambda^{4}(r) + 2(\lambda^{2}(r))^{2},\\
&\vdots\end{align*} It is known that in general
\[\Psi^{i}(r) = det\left(
                           \begin{array}{cccccc}
                             r & 1 & 0 & 0 & \ldots & 0 \\
                             2\lambda^{2}(r) & r & 1 & 0 &\ldots & 0 \\
                             3\lambda^{3}(r) & \lambda^{2}(r) & r & 1 & 0 & \vdots \\
                             \vdots & \vdots & \ddots & \ddots & \ddots & 0\\
                              \vdots & \vdots &  & \lambda^{2}(r) & r  & 1\\
                             i\lambda^{i}(r) & \lambda^{i-1}(r) & \ldots & \ldots & \lambda^{2}(r) & r \\
                           \end{array}
                         \right).
\]
\end{example}

\begin{theorem}
  If $R$ is a $\lambda$-ring then the Adams operations give us a
  special $\Psi$-ring structure on $R$, which we denote by $R_{\Psi}$.
\end{theorem}

We will require the following useful theorem from \cite{Knut} (p.49).

\begin{theorem} \label{Psilamring}
  Let $R$ be a torsion-free pre-$\lambda$-ring. Let $\Psi^{i}:R \rightarrow R$ be the corresponding Adams operations. If
  $R$ together with the $\Psi$-operations form a $\Psi$-ring, then $R$ is a $\lambda$-ring.
\end{theorem}
The proof of this theorem can also be found in \cite{Knut}.

\begin{example}
Consider the simplest example of a $\lambda$-ring, $\mathbb{Z}$,
together with binomial coefficients $\lambda^{i}(r) =
\binom{r}{i}$. The Adams operations give us $\Psi^{i}(r) = r$ for all $r \in \mathbb{Z}$ and $i
\in \mathbb{N}$, which we have already seen gives us a $\Psi$-ring
structure on $\mathbb{Z}$.
\end{example}

\section{$\lambda$-modules}
For usual rings, we have modules which provide us with the
coefficients for the cohomology. We now define the $\lambda$-modules
for $\lambda$-rings which provide us with the coefficients
for the $\lambda$-ring cohomology.

\begin{definition}
$M$ is a \textit{$\lambda$-module} over the $\lambda$-ring $R$ if $M$ is an $R$-module together
with a sequence of abelian group homomorphisms $\Lambda^{i}: M \rightarrow M$, for $i \in \mathbb{N}$, satisfying
\begin{enumerate}
\item $\Lambda^{1}(m) = m$,
\item $\Lambda^{i}(rm) = \Psi^{i}(r) \Lambda^{i}(m)$,
\item $\Lambda^{ij}(m) = (-1)^{(i+1)(j+1)}\Lambda^{i}\Lambda^{j}(m)$,
\end{enumerate}
for all $m \in M, r \in R$ and $i,j \in \mathbb{N}$.

Let $M,N$ be two $\lambda$-modules over $R$. A map of $\lambda$-modules is a module homomorphism $f:M \rightarrow N$ such that $\Lambda^{i} f(m) =f \Lambda^{i}(m)$ for all $m \in M$ and $i \in \mathbb{N}_{0}$. We let $R \mathfrak{-mod}_{\lambda}$ denote the category of all $\lambda$-modules over $R$.
\end{definition}

The main motivation for our definition of a $\lambda$-module is as follows. First we let $R$ and $X$ be two $\lambda$-rings and $\beta: X \rightarrow R$ be a map of $\lambda$-rings. Assume $M= \Ker \beta$ is a square-zero ideal. Since $\lambda^{i}(0) = 0$, for $i>0$, there are maps $\Lambda^{i}: M \rightarrow M$, for $i>0$, which make the following diagram commutes:
\[\xymatrix{0 \ar[r] & M \ar^{\Lambda^{i}}[d] \ar^{\alpha}[r] & X \ar^{\beta}[r] \ar^{\lambda^{i}}[d] & R \ar[r]  \ar^{\lambda^{i}}[d] & 0 \\ 0 \ar[r] & M \ar^{\alpha}[r] & X \ar^{\beta}[r] & R \ar[r] & 0. }\]
The properties of the $\Lambda$-operations follow from the properties of the $\lambda$-operations. For example,
\begin{align*}\alpha \lambda^{i}(rm) = & \lambda^{i}\alpha(rm) \\ = & \lambda^{i}(x \alpha(m)), \end{align*}
for some $x \in X$ with $\beta(x)=r$. Therefore,
\[\alpha \lambda^{i}(rm) = P_{i}(\lambda^{1}(x),\ldots,\lambda^{i}(x),\lambda^{1}(\alpha(m)),\ldots\lambda^{i}(\alpha(m))). \]
However $\alpha(m)\alpha(n)=0$ for all $m,n \in M$ so most of the
terms vanish leaving
\[\alpha \lambda^{i}(rm) = \alpha\Psi^{i}(r)\Lambda^{i}(m).\]

For the rest of this chapter, we let $R$ denote a $\lambda$-ring and $M \in R-\mathfrak{mod}_{\lambda}$. We let $\underline{R}$ denote the underlying commutative ring of $R$, and $\underline{M}$ denote the underlying $\underline{R}$-module of $M$.

\begin{example}
In general, $R$ is not a $\lambda$-module over
itself unless the multiplication in $R$ is trivial. However we can consider the sequence of operations
$\Lambda^{i}:R \rightarrow R$ given by
$\Lambda^{i}(r)=(-1)^{(i+1)}\Psi^{i}(r)$. With these
$\Lambda$-operations $R$ is a $\lambda$-module over $R$.
\end{example}

\begin{theorem}
The Adams operation $\psi^{n}:M \rightarrow M$ given by \[\psi^{n}(m) =
(-1)^{(n+1)}n\Lambda^{n}(m),\] give us a special $\Psi$-module structure on $M$ over $R_{\Psi}$, which we denote by $M_{\Psi}$.
\end{theorem}

\begin{proof}
\begin{enumerate}
  \item $\psi^{1}(m) = \Lambda^{1}(m) = m$,
  \item $\psi^{i}(m_{1} + m_{2}) = (-1)^{i+1}i\Lambda^{i}(m_{1} + m_{2}) = (-1)^{i+1}i\Lambda^{i}(m_{1}) + (-1)^{i+1}i\Lambda^{i}(m_{2}) \\=\psi^{i}(m_{1}) +
  \psi^{i}(m_{2})$,
  \item $\psi^{i}(rm) = (-1)^{i+1}i\Lambda^{i}(rm) = (-1)^{i+1}i\Psi^{i}(r)\Lambda^{i}(m) = \Psi^{i}(r)\psi^{i}(m)$,
  \item $\psi^{i}(\psi^{j}(m)) = \psi^{i}((-1)^{(j+1)}j \Lambda^{j}(m)) = (-1)^{(i+ j)}ij \Lambda^{i}(\Lambda^{j}(m)) \\= (-1)^{(ij+1)}ij\Lambda^{ij}(m) = \psi^{(ij)}(m).$
          \end{enumerate}
\end{proof}

We will require the following useful lemma.

\begin{lemma}\label{chilemma}
  \[\sum_{i=1}^{\nu -1}[(-1)^{i+1}\chi_{i}(r,m)\Psi^{\nu-i}(r) + (-1)^{\nu+1}i\lambda^{i}(r)\Lambda^{\nu-i}(m)] = 0,\]
  for all $r \in R, m \in M$ and $\nu \geq 2$, where $\chi_{i}(r,m) = \sum_{j=1}^{i}\Lambda^{j}(m)\lambda^{i-j}(r)$.
\end{lemma}

\begin{proof}
  We are going to use proof by induction on $\nu$. Consider the case when $\nu=2$.
\begin{align*}  LHS =& (-1)^{2}\chi_{1}(r,m)\Psi^{1}(r) +
(-1)^{3}\lambda^{1}(r)\Lambda^{1}(m) \\=& mr - rm \\=& 0.\end{align*}
We are also going to consider the case when $\nu=3$.

\begin{align*}  LHS =& \chi_{1}(r,m)\Psi^{2}(r) + \lambda^{1}(r)\Lambda^{2}(m)
- \chi_{2}(r,m)\Psi^{1}(r) + 2\lambda^{2}(r)\Lambda^{1}(m) \\=&
m[r^{2}-2\lambda^{2}(r)] + r\Lambda^{2}(m) - [mr + \Lambda^{2}(m)]r
+ 2\lambda^{2}(r)m \\=& 0. \end{align*}

Now assume that
\[\sum_{i=1}^{\nu-k-1}[(-1)^{i+1}\chi_{i}(r,m)\Psi^{\nu-k-i}(r) +
(-1)^{\nu-k+1}i\lambda^{i}(r)\Lambda^{\nu-k-i}(m)] = 0,\] for $1 \leq k
\leq \nu-2$.

It follows that
\begin{align*}
&\sum_{i=1}^{\nu-1}[(-1)^{i+1}\chi_{i}(r,m)\Psi^{\nu-i}(r) +
(-1)^{\nu+1}i\lambda^{i}(r)\Lambda^{\nu-i}(m)] \\=&
\sum_{i=1}^{\nu-1}(-1)^{\nu}(\nu-i)\lambda^{\nu-i}(r)\chi_{i}(r,m) \\&+
\sum_{i=1}^{\nu-2}(-1)^{i+1}\chi_{i}(r,m)[\sum_{j=1}^{\nu-i-1}(-1)^{j+1}\lambda^{j}(r)\Psi^{\nu-i-j}(r)]
+ \sum_{i=1}^{\nu-1}(-1)^{\nu+1}i\lambda^{i}(r)\Lambda^{\nu-i}(m) \\=&
\sum_{i=1}^{\nu-2}(-1)^{\nu}i\lambda^{i}(r)[\sum_{j=1}^{\nu-i-1}\Lambda^{j}(m)\lambda^{\nu-i-j}(r)]
+
\sum_{i=1}^{\nu-2}\chi_{i}(r,m)[\sum_{j=1}^{\nu-i-1}(-1)^{j+i}\lambda^{j}(r)\Psi^{\nu-i-j}(r)]
\\=&
\sum_{k=1}^{\nu-2}\lambda^{k}(r)[\sum_{i=1}^{\nu-k-1}(-1)^{\nu}i\lambda^{i}(r)\lambda^{\nu-k-i}(r)]
+
\sum_{k=1}^{\nu-2}\lambda^{k}(r)[\sum_{i=1}^{\nu-k-1}(-1)^{i+k}\chi_{i}(r,m)\Psi^{\nu-k-i}(r)]
\\=&\sum_{k=1}^{\nu-2}(-1)^{k+1}\lambda^{k}(r)[\sum_{i=1}^{\nu-k-1}[(-1)^{i+1}\chi_{i}(r,m)\Psi^{\nu-k-i}(r)
+ (-1)^{\nu-k+1}i\lambda^{i}(r)\Lambda^{\nu-k-i}(m)] \\=& 0.\end{align*}
as required.
\end{proof}

\begin{lemma}\label{semidirectlma}
The set $R \rtimes M$ with \begin{align*} (r,m) + (s,n) &= (r+s, m+n),
\\ (r,m)(s,n) &= (rs, rn + ms),
\end{align*} together with maps $\lambda^{i}: R \rtimes M
\rightarrow R \rtimes M$ for $i \in \mathbb{N}_{0}$ given by
\[ \lambda^{i} (r,m) = (\lambda^{i}(r), f_{i}(r,m)),\] for a sequence
of maps $f_{i}: R \rtimes M \rightarrow M$, for $i \in \mathbb{N}_{0}$, is a pre-$\lambda$-ring if and only if
\begin{enumerate}
\item $f_{0}(r,m) = 0$,
\item $f_{1}(r,m) = m$,
\item $f_{i}( (r,m)+(s,n) ) = \sum_{j=0}^{i}(f_{j}(r,m)\lambda^{i-j}(s) + \lambda^{j}(r)f_{i-j}(s,n)).$
\end{enumerate}
\end{lemma}

\begin{proof}[Proof of lemma]
$\underline{R}$ is a commutative ring with identity, and $\underline{M}$ is an $\underline{R}$-module.
Then we know that $\underline{R} \rtimes \underline{M}$ is a commutative ring with identity.
So we only have to check the properties of $\lambda^{i}: R \rtimes M
\rightarrow R \rtimes M$.
\begin{enumerate}
    \item $\lambda^{0}(r,m) = (\lambda^{0}(r),f_{0}(r,m))$. \\Hence $\lambda^{0}(r,m)= (1,0)$ if and only if
    $f_{0}(r,m) = 0$,

    \item $\lambda^{1}(r,m) = (\lambda^{1}(r),f_{1}(r,m))$. \\Hence $\lambda^{1}(r,m)= (r,m)$ if and only if
    $f_{1}(r,m) = m$,

    \item $\lambda^{i}((r,m)+(s,n)) = \lambda^{i}(r+s,m+n) = (\lambda^{i}(r+s),f_{i}(r+s,m+n))
    \\ \sum_{j=0}^{i}\lambda^{j}(r,m)\lambda^{i-j}(s,n) = \sum_{j=0}^{i}(\lambda^{j}(r),f_{j}(r,m))(\lambda^{i-j}(s),f_{i-j}(s,n))
    \\ = \sum_{j=0}^{i}(\lambda^{j}(r)\lambda^{i-j}(s) ,f_{j}(r,m)\lambda^{i-j}(s) +
    \lambda^{j}(r)f_{i-j}(s,n))$ \\Hence \\$\lambda^{i}((r,m)+(s,n)) = \sum_{j=0}^{i}\lambda^{j}(r,m)\lambda^{i-j}(s,n)$
     if and only if \\$f_{i}( (r,m)+(s,n) ) = \sum_{j=0}^{i}(f_{j}(r,m)\lambda^{i-j}(s) + \lambda^{j}(r)f_{i-j}(s,n)).$
\end{enumerate}
\end{proof}

\begin{lemma}
The set $R \rtimes M$ together with maps $\lambda^{i}: R \rtimes M
\rightarrow R \rtimes M$, for $i \in \mathbb{N}_{0}$, given by
\[ \lambda^{i} (r,m) = \left\{
                         \begin{array}{ll}
                           (1,0) & \textrm{for }
i =0, \\
                           (\lambda^{i}(r),\sum_{j=1}^{i} \Lambda^{j}(m)\lambda^{i-j}(r)) & \textrm{for }
i \in \mathbb{N},
                         \end{array}
                       \right.
 \] gives us a $\lambda$-ring.
\end{lemma}
We call this $\lambda$-ring the \textit{semi-direct product} of $R$ and $M$, denoted by $R \rtimes_{\lambda} M$.

\begin{proof}
We start by showing this is a pre-$\lambda$ ring by using lemma \ref{semidirectlma} with \[f_{i}(r,m) = \left\{
\begin{array}{cc}
0 & \textrm{for }i =0, \\
\sum_{j=1}^{i}\Lambda^{j}(m)\lambda^{i-j}(r) & \textrm{for } i \geq
1.
\end{array}
\right.\] Clearly properties 1 and 2 hold, so we only have to check
3. Let $i \geq 2$ then
\\\begin{align*}f_{i}((r,m)+(s,n)) =&
f_{i}(r+s,m+n) = \sum_{j=1}^{i}\Lambda^{j}(m+n)\lambda^{i-j}(r+s)
\\=&
\sum_{j=1}^{i}((\Lambda^{j}(m)+\Lambda^{j}(n))\sum_{k=0}^{i-j}\lambda^{k}(r)\lambda^{i-j-k}(s)
\\=&
\sum_{j=1}^{i}\sum_{k=0}^{i-j}(\Lambda^{i}(m)\lambda^{k}(r)\lambda^{i-j-k}(s)
+ \Lambda^{i}(n)\lambda^{k}(r)\lambda^{i-j-k}(s)) \\ =&
\sum_{j=1}^{i} \sum_{k=1}^{j}
\Lambda^{k}(m)\lambda^{j-k}(r)\lambda^{i-j}(s) +
\sum_{j=1}^{i-1}\sum_{k=1}^{i-j}\lambda^{j}(r)\Lambda^{k}(n)\lambda^{i-j-k}(s)
\\&+ \sum_{k=1}^{i}\Lambda^{k}(n)\lambda^{i-k}(s)\lambda^{0}(r) \\ =&
\sum_{j=1}^{i-1}
\sum_{k=1}^{j}\Lambda^{k}(m)\lambda^{j-k}(r)\lambda^{i-j}(s) +
\sum_{k=1}^{i}\Lambda^{k}(m)\lambda^{i-k}(r)\lambda^{0}(s)
\\&+\sum_{j=1}^{i-1}
\sum_{k=1}^{i-j}\lambda^{j}(r)\lambda^{k}(n)\lambda^{i-j-k}(s) +
\sum_{k=1}^{i}\Lambda^{k}(n)\lambda^{i-k}(s)\lambda^{0}(r)
\\ =& \sum_{j=1}^{i-1}(f_{j}(r,m)\lambda^{i-j}(s)
\lambda^{j}(r)f_{i-j}(s,n)) \\ & + f_{i}(r,m)\lambda^{0}(s) +
f_{i}(s,n)\lambda^{0}(r)  + \lambda^{i}(s)f_{0}(r,m) +
\lambda^{i}(r)f_{0}(s,n) \\=&
\sum_{j=0}^{i}(f_{j}(r,m)\lambda^{i-j}(s) +
\lambda^{j}(r)f_{i-j}(s,n)).\end{align*}

So we have proved that $R \rtimes_{\lambda} M$ is a
pre-$\lambda$-ring. Checking the last two axioms is reduced to
checking the following the following universal polynomial identities hold.
\begin{itemize}
    \item $P_{i}(\lambda^{1}(r,m),\ldots,
    \lambda^{i}(r,m), \lambda^{1}(s,n), \ldots, \lambda^{i}(s,n)) \\= (P_{i}(\lambda^{1}(r),\ldots,
    \lambda^{i}(r), \lambda^{1}(s), \ldots, \lambda^{i}(s)),\\\sum_{k=1}^{i-1}P_{i-k}(\lambda^{1}(r),\ldots,
    \lambda^{i-k}(r), \lambda^{1}(s), \ldots, \lambda^{i-k}(s))[\Psi^{k}(s)\Lambda^{k}(m)+\Psi^{k}(r)\Lambda^{k}(n)])$,
    \item $P_{i,j}(\lambda^{1}(r,m),\ldots,\lambda^{ij}(r,m)) = (P_{i,j}(\lambda^{1}(r),\ldots,\lambda^{ij}(r)),\\\sum_{k=1}^{i}\sum_{l=1}^{j}(-1)^{(k+1)(l+1)}\Lambda^{kl}(m)\Psi^{k}(\lambda^{j-l}(r))P_{(i-k),j}(\lambda^{1}(r),\ldots,\lambda^{(i-k)j}(r))) $.
\end{itemize}

We are going to start by considering the case where $R$ is a free
$\lambda$-ring and $M$ is free as a $\lambda$-module over $R$.

Our aim is to show that the Adams operations give us the $\Psi$-ring
structure on $R \rtimes M$ with $\Psi^{\nu}(r,m) =
(\Psi^{\nu}(r),\psi^{\nu}(m))$ by using induction on $\nu$. Then
theorem \ref{Psilamring} tells us that $R \rtimes_{\lambda} M$ is a
$\lambda$-ring and the universal polynomial identities hold.

Consider the case when $\nu=1$ \[\Psi^{1}(r,m) = (r,m) =
(\Psi^{1}(r),\psi^{1}(m)).\]

Consider the case when $\nu=2$ \[\Psi^{2}(r,m) = (r^{2},2rm) -
2\lambda^{2}(r,m) = (r^{2} - 2\lambda^{2}(r), -2\Lambda^{2}(m)) =
(\Psi^{2}(r),\psi^{2}(m)).\]

Assume that $\Psi^{\nu-k}(r,m) = (\Psi^{\nu-k}(r),\psi^{\nu-k}(m))$ for $1
\leq k \leq \nu-1$. It follows that

\begin{align*}\Psi^{\nu}(r,m) =&
\sum_{j=1}^{\nu-1}(-1)^{\nu-j+1}(\lambda^{\nu-j}(r),\sum_{k=1}^{\nu-j}\lambda^{\nu-j-k}(r)\Lambda^{k}(m))(\Psi^{j}(r),\psi^{j}(m))
\\&+(-1)^{\nu+1}\nu(\lambda^{\nu}(r),\sum_{k=1}^{\nu}\lambda^{\nu-k}(r)\Lambda^{k}(m)) \\=& \sum_{j=1}^{\nu-1}(-1)^{\nu-j+1}
(\lambda^{\nu-j}(r)\Psi^{j}(r),\Psi^{j}(r)\sum_{k=1}^{\nu-j}\lambda^{\nu-j-k}(r)\Lambda^{k}(m)
\\& + \Psi^{j}(m)\lambda^{\nu-j}(r)) +
(-1)^{\nu+1}\nu(\lambda^{\nu}(r),\sum_{k=1}^{\nu}\lambda^{\nu-k}(r)\Lambda^{k}(m))
\\ =&  (\Psi^{\nu}(r),\sum_{j=1}^{\nu-1}(-1)^{\nu-j+1}[
\lambda^{\nu-j}(r)\Psi^{j}(m) +
\Psi^{j}(r)\sum_{k=1}^{\nu-j}\lambda^{\nu-j-k}(r)\Lambda^{k}(m)]
\\&+
(-1)^{\nu+1}\nu[\sum_{k=1}^{\nu}\lambda^{\nu-k}(r)\Lambda^{k}(m)])
\\=&
(\Psi^{\nu}(r),\sum_{j=1}^{\nu-1}[(-1)^{j+1}\Psi^{\nu-i}(r)\chi_{i}(r,m)
+ (-1)^{\nu+1}j\lambda^{j}(r)\Lambda^{\nu-j}(m)] \\&+
(-1)^{\nu+1}\nu\Lambda^{\nu}(m)) \\ =& (\Psi^{\nu}(r),
(-1)^{\nu+1}\nu\Lambda^{\nu}(m)) =
(\Psi^{\nu}(r),\psi^{\nu}(m)),\end{align*} as required.

Now consider the case where $R$ is a free $\lambda$-ring and $M$ is
an arbitrary $\lambda$-module over $R$. Choose $P$ a free
$\lambda$-module over $R$ with a surjective
homomorphism $P \twoheadrightarrow M$, this gives us a surjective
homomorphism $R \rtimes_{\lambda} P \twoheadrightarrow R
\rtimes_{\lambda}M$. Since the universal polynomial identities hold
on $R \rtimes_{\lambda} P$ they also hold on $R \rtimes_{\lambda}
M$.

Now we can consider the case when $R$ is an arbitrary $\lambda$-ring
and $M$ is a $\lambda$-module over $R$. Any $\lambda$-ring is the
quotient of a free $\lambda$-ring, therefore $R$ is the quotient of
a free $\lambda$-ring $F$. There exists a surjective homomorphism $F
\rtimes_{\lambda} M \twoheadrightarrow R \rtimes_{\lambda}M$. Since
the universal polynomial identities hold on $F \rtimes_{\lambda} M$
they also hold on $R \rtimes_{\lambda} M$. Hence $R
\rtimes_{\lambda} M$ is a $\lambda$-ring. Moreover we proved that
$(R \rtimes_{\lambda} M)_{\Psi}=R_{\Psi}\rtimes_{\Psi}M_{\Psi}$.
\end{proof}

\section{$\lambda$-derivations}
\begin{definition}
A \textit{$\lambda$-derivation} of $R$ with values in $M$ is an additive homomorphism $d:R
\rightarrow M$ such that
\begin{enumerate}
\item $d(r s) = r d(s) + d(r)s$,
\item $d(\lambda^{i}(r)) = \Lambda^{i}(d(r)) + \Lambda^{i-1}(d(r)) \lambda^{1}(r) + \ldots +
  \Lambda^{2}(d(r))\lambda^{i-2}(r)+ \Lambda^{1}(d(r))\lambda^{i-1}(r)$,
\end{enumerate}
for all $r,s \in R$, and $i \in \mathbb{N}$. We let $\Der_{\lambda}(R,M)$ denote the set of all $\lambda$-derivations
of $R$ with values in $M$.
\end{definition}

\begin{example}
Let $\mathbb{Z}_{\lambda}[x]$ be the free $\lambda$-ring on one generator $x$, and let
$M \in \mathbb{Z}_{\lambda}[x]\mathfrak{-mod}_{\lambda}$.
\[Der_{\lambda}(\mathbb{Z}_{\lambda}[x],M) \cong M.\]
$\mathbb{Z}_{\lambda}[x] = \mathbb{Z}[x_{1},x_{2},\ldots]$ together with
operations determined by $\lambda^{i}(x_{1}) = x_{i}$. For any
$\lambda$-derivation, $d: \mathbb{Z}_{\lambda}[x] \rightarrow M$, we have that
\begin{align*}
d(x_{1}) &= m,
\\d(x_{i}) &= \sum_{j=1}^{i}\Lambda^{j}(m)x_{i-j},
\end{align*}
where $m \in M$ and $x_{0}=1$.\end{example}

\begin{theorem}
There is a one-to-one correspondence between the sections of $
\xymatrix{ R \rtimes_{\lambda} M \ar[r]^-{\pi} & R }$ and the
$\lambda$-derivations $d:R \rightarrow M$.
\end{theorem}

\begin{proof}[Proof of theorem]
Assume we have a section of $\pi$, then we have the following  \[ \xymatrix{ R \rtimes_{\lambda} M
\ar@<1ex>[r]^-{\pi} & R, \ar@<1ex>[l]^-{\sigma} }\] where $\pi \sigma  =
Id_{R}$. Hence $\sigma(r) = (r,d(r))$ for some $d:R \rightarrow M$. The properties
\begin{align*}d(r+s) =& d(r) + d(s), \\ \qquad d(rs) =& d(r)s + r
d(s),\end{align*} follow from
$\sigma$ being a ring homomorphism. However $\sigma$ also preserves
the $\lambda$-ring structure, meaning that $\lambda^{i}\sigma(r) =
\sigma\lambda^{i}(r)$. We know that \begin{align*}\lambda^{i}\sigma(r) &=
\lambda^{i}(r, d(r)) = \lambda^{i}((r,0)+(0,d(r))) =
\Sigma_{j=0}^{i}\lambda^{j}(r,0)\lambda^{i-j}(0,d(r))
\\ &= \Sigma_{j=0}^{i-1}(0,\lambda^{j}(r)\Lambda^{i-j}(d(r))) +
(\lambda^{i}(r),0) = (\lambda^{i}(r),
\Sigma_{j=0}^{i-1}\lambda^{j}(r)\Lambda^{i-j}(d(r)))
\\\sigma\lambda^{i}(r) &= (\lambda^{i}(r),d(\lambda^{i}(r))).
\end{align*} Hence $\lambda^{i}\sigma(r) = \sigma\lambda^{i}(r)$ if and only if
$d\lambda^{i}(r) =
\Sigma_{j=0}^{i-1}\lambda^{j}(r)\Lambda^{i-j}(d(r))$. This tells us
that if $\sigma$ is a section of $\pi$, then we have a $\lambda$-derivation
$d$.

Conversely, if we have a $\lambda$-derivation $d:R \rightarrow M$,
then $\sigma(r) = (r,d(r))$ is a section of $\pi$.
\end{proof}

\begin{theorem}
The $\lambda$-derivations of $R$ with values in $M$ are also $\Psi$-derivations of $R_{\Psi}$ with values in $M_{\Psi}$.
\end{theorem}

\begin{proof}
  Let $d:R \rightarrow M$ be a $\lambda$-derivation, we are going to
use induction on $\nu$ to show $\psi^{\nu}(d(r)) = d(\Psi^{\nu}(r))$ for all
$\nu \geq 1$.

Consider the case when $\nu=1$
\[\psi^{1}(d(r)) = d(r) = d(\Psi^{1}(r)).\]
Consider the case when $\nu=2$.
\[d(\Psi^{2}(r)) = d(r^{2} - 2\lambda^{2}(r)) = 2rd(r) - 2[\Lambda^{2}(d(r)) + d(r)r] = -2\Lambda^{2}(d(r)) = \psi^{2}(d(r)).\]
Also consider the case $\nu=3$.
\[d(\Psi^{3}(r)) = d(r^{3} - 3r\lambda^{2}(r)+3\lambda^{3}(r)) =  3\Lambda^{3}(d(r)) = \psi^{3}(d(r)).\]

Assume that $\psi^{\nu-k}(d(r)) = d(\Psi^{\nu-k}(r))$ for $1 \leq k \leq
\nu-1$.
\begin{align*}d(\Psi^{\nu}(r))=&
\sum_{i=1}^{\nu-1}(-1)^{i+1}d(\lambda^{i}(r))\Psi^{\nu-i}(r) +
\sum_{i=1}^{\nu-1}(-1)^{i+1}\lambda^{i}(r)d(\Psi^{\nu-i}(r)) \\
&+ (-1)^{\nu+1}\nu d(\lambda^{\nu}(r)) \\ d(\Psi^{\nu}(r)) - \psi^{\nu}(d(r)) =&
\sum_{i=1}^{\nu-1}(-1)^{i+1}[\sum_{j=1}^{i}\Lambda^{j}(d(r))\lambda^{i-j}(r)]\Psi^{\nu-i}(r)
\\
&+ \sum_{i=1}^{\nu-1}(-1)^{i+1}\lambda^{i}(r)[(-1)^{\nu-i+1}(\nu-i)\Lambda^{\nu-i}(d(r))]
\\ &+ (-1)^{\nu+1}\nu[\sum_{j=1}^{\nu-1}\Lambda^{j}(d(r))\lambda^{\nu-j}(r)]
\\=& \sum_{i=1}^{\nu-1}[(-1)^{i+1}\chi_{i}(r,d(r))\Psi^{\nu-i}(r)+(-1)^{\nu+1}i\lambda^{i}(r)\Lambda^{\nu-i}(d(r))]
\\=&0.
\end{align*}

Hence $d(\Psi^{\nu}(r)) = \psi^{\nu}(d(r))$.
\end{proof}

\begin{theorem}
If $M$ is $\mathbb{Z}$-torsion-free  then the
$\Psi$-derivations of $R_{\Psi}$ with values in $M_{\Psi}$ are also
$\lambda$-derivations of $R$ with values in $M$
\[\Der_{\lambda}(R,M) = \Der_{\Psi}(R_{\Psi},M_{\Psi}).\]
\end{theorem}

\begin{proof}
Let $M$ be $\mathbb{Z}$-torsion-free and $d:R_{\Psi} \rightarrow M_{\Psi}$ be a $\Psi$-derivation. We are going to
use induction on $\nu$ to show $d(\lambda^{\nu}(r)) =
\sum_{i=1}^{\nu}\Lambda^{i}(d(r))\lambda^{\nu -i}(r)$ for $\nu \in \mathbb{N}$.

Consider the case when $\nu=1$.
\[ \Lambda^{1}(d(r)) = d(r) = d(\lambda^{1}(r)). \]
Consider the case when $\nu=2$.
\begin{align*}
d(\Psi^{2}(r))  &=
\psi^{2}(d(r))
\\ d(r^{2} - 2\lambda^{2}(r)) &=  -2\Lambda^{2}(d(r))
\\ 2[d(\lambda^{2}(r)) - rd(r) - \Lambda^{2}(d(r))]&=0
\\ 2[d(\lambda^{2}(r)) - \sum_{i=1}^{2}\Lambda^{i}(d(r))\lambda^{2-i}(r) ] &= 0
\\ d(\lambda^{2}(r)) - \sum_{i=1}^{2}\Lambda^{i}(d(r))\lambda^{2-i}(r)
&=0. \end{align*}

Assume that $d(\lambda^{\nu-k}(r)) =
\sum_{i=1}^{\nu-k}\Lambda^{i}(d(r))\lambda^{\nu-i-k}(r)$ for $1 \leq k
\leq \nu-1$, we want to show that $\nu d(\lambda^{\nu}(r)) = \nu \sum_{i=1}^{\nu}\Lambda^{i}(d(r))\lambda^{\nu-i}(r)$.
From $\psi^{v}(d(r))=d(\Psi^{v}(r))$ we get\\
$\nu(\Lambda^{\nu}(d(r))-d(\lambda^{\nu}(r)))=\sum_{i=1}^{\nu-1}(-1)^{i+v}[d(\lambda^{i}(r))\Psi^{\nu -i}(r) + \lambda^{i}(r)d(\Psi^{\nu-i}(r))].$
 \\Therefore we have to show that
\begin{align*}
&(-1)^{\nu}\nu\sum_{i=1}^{\nu-1}\Lambda^{i}(d(r))\lambda^{\nu-i}(r) \\=&
\sum_{i=1}^{\nu-1}(-1)^{i+1}d(\lambda^{i}(r))\Psi^{\nu-i}(r)  +
\sum_{i=1}^{\nu-1}(-1)^{i+1}\lambda^{i}(r)d(\Psi^{\nu-i}(r)) \\=&
\sum_{i=1}^{\nu-1}(-1)^{i+1}[\sum_{j=1}^{i}\Lambda^{j}(d(r))\lambda^{i-j}(r)]\cdot[\sum_{k=1}^{\nu-i-1}(-1)^{k+1}\lambda^{k}(r)\Psi^{\nu-i-k}(r)
 \\&+ (-1)^{\nu-i-1}(\nu-i)\lambda^{\nu-i}(r)] +
(-1)^{\nu}\sum_{i=1}^{\nu-1} i \Lambda^{i}(d(r))\lambda^{\nu-i}(r).
\end{align*}
\\Hence it is sufficient to show that \\$\sum_{i=1}^{\nu-2}(-1)^{i+1}\chi_{i}(r,d(r))[\sum_{k=1}^{\nu-i-1}(-1)^{k+1}\lambda^{k}(r)\Psi^{\nu-i-k}(r)]
\\ + \sum_{i=1}^{\nu-1}(-1)^{\nu+1}(i-\nu)\lambda^{\nu-i}(r)\chi_{i}(r,d(r))
+ (-1)^{\nu+1}\sum_{i=1}^{\nu-1}i\lambda^{i}(r)\Lambda^{\nu-i}(d(r))]=0,$
\\with $\chi_{i}$ as in lemma \ref{chilemma}. We get that
\\ \begin{align*} &\sum_{i=1}^{\nu-2}(-1)^{i+1}\chi_{i}(r,d(r))[\sum_{k=1}^{\nu-i-1}(-1)^{k+1}\lambda^{k}(r)\Psi^{\nu-i-k}(r)]
\\& + \sum_{i=1}^{\nu -1}(-1)^{\nu+1}(i-\nu)\lambda^{\nu -i}(r) \chi_{i}(r,d(r))
+ (-1)^{\nu+1}\sum_{i=1}^{\nu -1}i\lambda^{i}(r)\Lambda^{\nu-i}(d(r))]
\\=&
\sum_{i=1}^{\nu -2}\chi_{i}(r,d(r))[(-1)^{i+1}[\sum_{k=1}^{\nu -i-1}(-1)^{k+1}\lambda^{k}(r)\Psi^{\nu -i-k}(r)
\\ &+ (-1)^{\nu -i-1}(\nu -i)\lambda^{\nu -i}(r) - \Psi^{\nu -i}(r)]]
\\ =& 0,
\end{align*}
as required. \end{proof}

\section{$\lambda$-ring extensions}
We have seen in proposition \ref{AQprop1} that the Andr\'{e}-Quillen cohomology
$H^{1}_{AQ}(\underline{R},\underline{M})$ classifies the extensions
of \underline{R} by \underline{M}. In this section, we develop the
$\lambda$-analogue of extensions.

\begin{definition}
A \textit{$\lambda$-ring extension} of $R$ by $M$ is an extension of $\underline{R}$ by $\underline{M}$
\[\xymatrix{0 \ar[r] & M \ar^{\alpha}[r] & X \ar^{\beta}[r] & R \ar[r] & 0}\] where
$X$ is a $\lambda$-ring, $\beta$ is a map of $\lambda$-rings and $\alpha \Lambda^{n} = \lambda^{n} \alpha$ for all $n \in \mathbb{N}$.
\end{definition}

Two $\lambda$-ring extensions $(X),(X')$
with $R,M$ fixed are said to be \textit{equivalent} if there exists a map of
$\lambda$-rings $\phi: X \rightarrow X'$ such that the following
diagram commutes.
\[\xymatrix{0 \ar[r] & M \ar@{=}[d] \ar[r] & X \ar[r] \ar[d]^{\phi} & R \ar[r] \ar@{=}[d] & 0 \\ 0 \ar[r] & M \ar[r] & X' \ar[r] & R \ar[r] & 0}\]

We denote the set of equivalence classes of $\lambda$-ring extensions
of $R$ by $M$ by $Ext_{\lambda}(R,M)$.

The Harrison cohomology $Harr^{1}(\underline{R},\underline{M})$ classifies the additively split extensions
of \underline{R} by \underline{M}. We can also define the $\lambda$-analogue of these types of extensions.

\begin{definition}
Let $R$ be a $\lambda$-ring and $M \in R$-mod$_{\lambda}$ then an
\textit{additively split $\lambda$-ring extension} of $R$ by $M$ is a $\lambda$-ring extension of $R$ by $M$
\[\xymatrix{0 \ar[r] & M \ar^{\alpha}[r] & X \ar^{\beta}[r] & R \ar[r] & 0}\]
where $\beta$ has a section that is an additive homomorphism.
\end{definition}

Multiplication in $X = R \oplus
M$ has the form $(r,m)(r',m') = (rr', mr' + rm'
+ f(r,r'))$, where $f:R \times R \rightarrow M$ is some bilinear
map. Associativity in $X$ gives us
\[0 = rf(r',r'')-f(rr',r'') + f(r,r'r'')-f(r,r')r''.\]
Commutativity in $X$ gives us
\[f(r,r') = f(r',r).\]

The $\lambda$-operations $\lambda^{\nu}:R \rtimes M
\rightarrow R \rtimes M$ for $\nu \in \mathbb{N}_{0}$ are given by $\lambda^{\nu}(r,m)
= (\lambda^{\nu}(r), \sum_{i=1}^{\nu}\Lambda^{i}(m)\lambda^{\nu-i}(r) +
\epsilon^{\nu}(r) )$ for a sequence of operations $\epsilon^{\nu}: R
\rightarrow M$ which satisfy the following properties
\begin{enumerate}
    \item $\epsilon^{0}(r) = \epsilon^{1}(r) = 0,$
    \item $\epsilon^{\nu}(r+s) =
    \sum_{i=0}^{\nu}[\epsilon^{i}(r)\lambda^{\nu-i}(s)+\epsilon^{\nu-i}(s)\lambda^{i}(r)],$
    \item $\epsilon^{\nu}(1) = 0,$
    \item $P_{i}(\lambda^{1}(r,m),\ldots,\lambda^{i}(s,n))\\=(\lambda^{i}(rs),\sum_{j=1}^{i}(\Psi^{j}(s)\Lambda^{j}(m)+\Psi^{j}(r)\Lambda^{j}(n) +\Lambda^{j}(f(r,r')))\lambda^{i-j}(rs)+\epsilon^{j}(rs)),$
    \item $P_{i,j}(\lambda^{1}(r,m),\ldots,\lambda^{ij}(r,m))\\=(\lambda^{i}(\lambda^{j}(r)),\sum_{k=1}^{i}\Lambda^{k}(\sum_{a=1}^{j}(\Lambda^{a}(m)\lambda^{j-a}(r)+\epsilon^{j}(r))\lambda^{i-k}(\lambda^{j}(r)))+\epsilon^{i}(\lambda^{j}(r))).$
\end{enumerate}

Assuming we have two additively split $\lambda$-ring extensions
$(X,\varepsilon,f)$,$(X',\varepsilon',f')$ which are equivalent,
together with a $\lambda$-ring map $\phi:X \rightarrow X'$ with
$\phi(r,m) = (r, m + g(r))$ for some $g: R \rightarrow M$. We have
that $\phi$ being a homomorphism tells us that
\[g(r+r') = g(r) + g(r'),\]
\[f(r,r')- f'(r,r') = rg(r') - g(rr') + g(r)r'.\]
We also have $\phi(\lambda^{\nu})= \lambda^{\nu}(\phi)$ for all $\nu \in \mathbb{N}_{0}$, which
tells us that \[\varepsilon^{\nu}(r)- \varepsilon'^{\nu}(r) =
\sum_{i=1}^{\nu}\Lambda^{i}(g(r))\lambda^{\nu - i}(r) - g(\lambda^{\nu}(r)).\]

We denote the set of equivalence classes of additively split $\lambda$-ring extensions
of $R$ by $M$ by $AExt_{\lambda}(R,M)$.

In order to describe the properties of $\lambda$-ring extensions we need to define the partial derivatives of the universal polynomials, see appendix \ref{appc} for examples.

We can use the universal polynomials to define continuous functions
\[P_{i}:\mathbb{R}^{2i} \rightarrow \mathbb{R},\]
\[P_{i,j}: \mathbb{R}^{ij} \rightarrow \mathbb{R}.\]

For example $P_{2}: \mathbb{R}^{4} \rightarrow \mathbb{R}$ is given by \[P_{2}(x_{1},x_{2},x_{3},x_{4}) = x_{1}^{2}x_{4} -2 x_{2}x_{4} + x_{2}x_{3}^{2}.\]

We can take the partial derivatives of these functions which are again polynomials. We call these new polynomials the \textit{partial derivatives} of the universal polynomials. For example
\begin{align*}
  \frac{\partial P_{2}(x_{1},x_{2},x_{3},x_{4})}{\partial x_{1}} =& 2x_{1}x_{4}, \\
  \frac{\partial P_{2}(x_{1},x_{2},x_{3},x_{4})}{\partial x_{2}} =& x_{3}^{2} - 2 x_{4},
\end{align*}

For $1 \leq j \leq i$, we let \[ \frac{\partial P_{i}(r,s)}{\partial \lambda^{j}(r)} := \frac{\partial P_{i}(\lambda^{1}(r),\ldots,\lambda^{i}(r),\lambda^{1}(s),\ldots,\lambda^{i}(s))}{\partial \lambda^{j}(r)}. \]

Since the polynomials $P_{i}$ are symmetric, we can let
\[ \frac{\partial P_{i}(r,s)}{\partial \lambda^{j}(s)} :=  \frac{\partial P_{i}(s,r)}{\partial \lambda^{j}(s)}. \]

In our examples
\begin{align*}
 \frac{\partial P_{2}(r,s)}{\partial \lambda^{1}(r)} =& \frac{\partial P_{2}(\lambda^{1}(r),\lambda^{2}(r),\lambda^{1}(s),\lambda^{2}(s))}{\partial \lambda^{1}(r)} =& 2r\lambda^{2}(s), \\
 \frac{\partial P_{2}(r,s)}{\partial \lambda^{2}(r)} =& \frac{\partial P_{2}(\lambda^{1}(r),\lambda^{2}(r),\lambda^{1}(s),\lambda^{2}(s))}{\partial \lambda^{2}(r)} =& s^{2} - 2\lambda^{2}(s).
\end{align*}

Similarly, for $1 \leq k \leq ij$, we let
 \[ \frac{\partial P_{i,j}(r)}{\partial \lambda^{k}(r)} := \frac{\partial P_{i}(\lambda^{1}(r),\ldots,\lambda^{ij}(r))}{\partial \lambda^{k}(r)}. \]

For example,
\[  \frac{\partial P_{2,2}(x_{1},x_{2},x_{3},x_{4})}{\partial x_{1}} = x_{3}. \]
So it follows that
\[ \frac{\partial P_{2,2}(r)}{\partial \lambda^{1}(r)} = \lambda^{3}(r).\]

These partial derivatives appear because of the multiplication in $R \rtimes M$. Consider the following
\[(r,m)^{2} = (r^{2}, 2rm),\]
\[(r,m)^{3} = (r^{3}, 3r^{2}m).\]
\begin{definition}
An \textit{additively and multiplicatively split $\lambda$-ring extension} of $R$ by $M$ is a $\lambda$-ring extension of $R$ by $M$
\[\xymatrix{0 \ar[r] & M \ar[r] & X \ar[r] & R \ar[r] & 0}\]
where $\beta$ has a section that is an additive and multiplicative homomorphism.

As a commutative ring $X = R \rtimes M$, the sequence of operations $\lambda^{\nu}:R \rtimes M
\rightarrow R \rtimes M$ for $\nu \in \mathbb{N}_{0}$ are given by $\lambda^{\nu}(r,m)
= (\lambda^{\nu}(r), \sum_{i=1}^{\nu}\Lambda^{i}(m)\lambda^{\nu-i}(r) +
\epsilon^{\nu}(r))$ for a sequence of operations $\epsilon^{\nu}: R
\rightarrow M$ such that
\begin{enumerate}
    \item $\epsilon^{0}(r) = \epsilon^{1}(r) = 0,$
    \item $\epsilon^{\nu}(r+s) =
    \sum_{i=0}^{\nu}[\epsilon^{i}(r)\lambda^{\nu-i}(s)+\epsilon^{\nu-i}(s)\lambda^{i}(r)],$
    \item $\epsilon^{\nu}(1) = 0,$
    \item $\epsilon^{\nu}(rs) = \sum_{i=1}^{\nu}[\epsilon^{i}(r) \frac{\partial P_{\nu}(r,s)}{\partial \lambda^{i}(r)} + \epsilon^{i}(s)\frac{\partial P_{\nu}(r,s)}{\partial \lambda^{i}(s)}],$
    \item $\epsilon^{k}(\lambda^{\nu}(r)) = \sum_{i=1}^{\nu k}\epsilon^{i}(r) \frac{\partial P_{\nu ,k}(r)}{\partial \lambda^{i}(r)} - \sum_{j=1}^{k}\Lambda^{j}(\epsilon^{\nu}(r))\lambda^{k-j}(\lambda^{\nu}(r)).$
\end{enumerate}
\end{definition}

Two additively and multiplicatively split $\lambda$-ring extensions $(X,\epsilon)$,$(X',\epsilon')$ with
$R,M$ fixed are said to be \textit{equivalent} if there exists a map
of $\lambda$-rings $\phi: X \rightarrow X'$ such that the following
diagram commutes.
\[\xymatrix{0 \ar[r] & M \ar@{=}[d] \ar[r] & X \ar[r] \ar[d]^{\phi} & R \ar[r] \ar@{=}[d] & 0 \\ 0 \ar[r] & M \ar[r] & X' \ar[r] & R \ar[r] & 0}\]

Assuming we have two additively and multiplicatively split $\lambda$-ring extensions
$(X,\epsilon)$,$(X',\epsilon')$ which are equivalent, together with a
$\lambda$-ring map $\phi:X \rightarrow X'$ with $\phi(r,m) = (r, m +
g(r))$ for some $g: R \rightarrow M$. We also have $\phi$ being a
homomorphism which tells us that $g \in Der(R,M)$. We also have
$\phi(\lambda^{\nu})= \lambda^{\nu}(\phi)$ for all $\nu$, which tells us
that
\[\varepsilon^{\nu}(r)- \varepsilon'^{\nu}(r) = \sum_{i=1}^{\nu}\Lambda^{i}(g(r))\lambda^{\nu -i}(r) -
g(\lambda^{\nu}(r)).\]

We denote the set of equivalence classes of additively and multiplicatively split $\lambda$-ring extensions
of $R$ by $M$ by $MExt_{\lambda}(R,M)$.

\begin{theorem}
If $\epsilon^{\nu}:R \rightarrow M$ gives us an additively and multiplicatively split $\lambda$-ring
extension of $R$ by $M$, then $\varepsilon^{\nu}: R \rightarrow M$ with
\[\varepsilon^{\nu}(r) = \sum_{i=1}^{\nu-1}(-1)^{i+1}[\epsilon^{i}(r)\Psi^{\nu-i}(r)+\lambda^{i}(r)\varepsilon^{\nu-i}(r)]+(-1)^{\nu+1}\nu\epsilon^{\nu}(r),\]
give us an additively and multiplicatively split $\Psi$-ring extension of $R_{\Psi}$ by $M_{\Psi}$.
\end{theorem}

\begin{proof}
If $\epsilon^{\nu}:R \rightarrow M$ gives an additively and multiplicatively split $\lambda$-ring extension of
$R$ by $M$, then $\lambda^{\nu}:R\rtimes M \rightarrow R \rtimes M$
given by $\lambda^{\nu}(r,m) =
(\lambda^{\nu}(r),\sum_{i=1}^{\nu}\Lambda^{i}(m)\lambda^{\nu-i}(r)+
\epsilon^{\nu}(r))$ is a $\lambda$-ring and hence the Adams operations
give the $\Psi$-ring with operations $\Psi^{\nu}:R\rtimes M
\rightarrow R \rtimes M$ given by $\Psi^{\nu}(r,m) =
(\Psi^{\nu}(r),\psi^{\nu}(m) + \varepsilon^{\nu}(r))$ which is an additively and multiplicatively split
$\Psi$-ring extension of $R_{\Psi}$ by $M_{\Psi}$.
\end{proof}

\section{Crossed $\lambda$-extensions}
A \textit{crossed $\lambda$-module} consists of a $\lambda$-ring $C_{0}$, a $\lambda$-module $C_{1}$ over $C_{0}$ and a map of $\lambda$-modules
\[\xymatrix{C_{1} \ar^{\partial}[r] & C_{0}},\]
which satisfies the property
\[\partial(c)c' = c \partial(c'),\]
for $c,c' \in C_{1}$. In other words, a crossed $\lambda$-module is a chain algebra which is non-trivial only in dimensions 0 and 1. Since $C_{2}=0$ the condition
$\partial(c)c' = c \partial(c')$ is equivalent to the Leibnitz relation \[0  = \partial (cc') = \partial(c)c' - c \partial(c'). \]
We can define a product by \[ c * c' := \partial(c) c', \]
for $c,c' \in C_{1}$. This gives us a $\lambda$-ring structure on $C_{1}$ and $\partial: C_{1} \rightarrow C_{0}$ is a map of $\lambda$-rings.

Let $\partial : C_{1} \rightarrow C_{0}$ be a crossed $\lambda$-module. We let $M= \Ker (\partial)$ and $R= \Coker (\partial)$ Then the image $\Image (\partial)$ is an ideal of $C_{0}$, $M C_{1} = C_{1}M = 0$ and $M$ has a well-defined $\lambda$-module structure over $R$.

A \textit{crossed $\lambda$-extension} of $R$ by $M$ is an exact sequence
\[\xymatrix{0 \ar[r] & M \ar[r]^{\alpha} & C_{1} \ar[r]^{\partial} & C_{0} \ar[r]^{\gamma} & R \ar[r] & 0}\]
where $\partial : C_{1} \rightarrow C_{0}$ is a crossed $\lambda$-module, $\gamma$ is a map of $\lambda$-rings, and the $\lambda$-module structure on $M$ coincides with the one induced from the crossed $\lambda$-module. We denote the category of crossed $\lambda$-extensions of $R$ by $M$ by $Cross_{\lambda}(R,M)$. We let $\pi_{0} Cross_{\lambda}(R,M)$ denote the connected components of the category $Cross_{\lambda}(R,M)$.

An \textit{additively split crossed $\lambda$-extension} of $R$ by $M$ is a
crossed $\lambda$-extension \begin{equation}\label{lcesasc} \xymatrix{0
\ar[r] & M \ar^{\omega}[r] & C_{1} \ar^{\rho}[r] & C_{0}
\ar^{\pi}[r] &  R \ar[r] & 0}\end{equation} such that all the arrows
in the exact sequence $\ref{lcesasc}$ are additively split.
We denote the connected components of the category of additively split crossed
$\lambda$-extensions of $R$ by $M$ by $\pi_{0} ACross_{\lambda}(R,M)$.

An \textit{additively and multiplicatively split crossed $\lambda$-extension} of $R$ by $M$ is an additively split
crossed $\lambda$-extension \[ \xymatrix{0
\ar[r] & M \ar^{\omega}[r] & C_{1} \ar^{\rho}[r] & C_{0}
\ar^{\pi}[r] &  R \ar[r] & 0}\] such that $\pi$ is additively and multiplicatively split.
We denote the connected components of the category of additively and multiplicatively split crossed
$\lambda$-extensions of $R$ by $M$ by $\pi_{0} MCross_{\lambda}(R,M)$.

\section{Yau cohomology for $\lambda$-rings}\label{yaucosect}
In 2005, Donald Yau published a paper entitled, ``Cohomology of
$\lambda$-rings'' \cite{Yau}, in which he developed a cohomology
theory for $\lambda$-rings. In this section we describe Yau's
cochain complex and what it computes.

Let $R$ be a $\lambda$-ring. We let $End(R)$ denote the algebra of $\mathbb{Z}$-linear endomorphisms of $R$, where the product is given by composition. We let $\overline{End}(R)$ denote the subalgebra of $End(R)$ which consists of the linear endomorphisms $f$ of $R$ which satisfy the condition,
\[f(r)^{p} \equiv f(r^{p}) \qquad \textrm{mod } pR,\] for each prime $p$ and every $r \in R$.

Yau defined $C_{Yau}^{0}(R)$ be the underlying group of $\overline{End}(R)$. He defined $C_{Yau}^{1}(R)$ be the set of functions $f: \mathbb{N} \rightarrow End(R)$ satisfying the condition $f(p)(R) \subset pR$ for each prime $p$. Then for $\nu \geq 2$ he set $C_{Yau}^{\nu}(R)$ to be the set of functions $f: \mathbb{N}^{\nu} \rightarrow End(R)$. For $\nu \in \mathbb{N}_{0}$, the coboundary map, $\delta^{\nu}: C^{\nu}_{Yau} \rightarrow C^{\nu+1}_{Yau}$, is given by the following
\begin{align*}
\delta^{\nu}(f)(m_{0},\ldots,m_{\nu}) =& \Psi^{m_{0}}\circ f(m_{1},\ldots,m_{\nu}) + \sum_{i=1}^{\nu}(-1)^{i}f(m_{0},\ldots,m_{i-1}m_{i},\ldots,m_{\nu}) \\ &+ (-1)^{\nu+1}f(m_{0},\ldots,m_{\nu-1})\circ\Psi^{m_{\nu}}.
\end{align*}
We say that the $\nu^{th}$ cohomology of the cochain complex $(C_{Yau},\delta)$ is the $\nu^{th}$ Yau cohomology of $R$, denoted by $H^{\nu}_{Yau}(R).$

From the cochain complex it is clear that \[H^{0}_{Yau}(R) = \{f \in
\overline{End}(R) : f\Psi^{\nu} = \Psi^{\nu}f \textrm{ for all } \nu \in
\mathbb{N}\}.\]

We define the group of \textit{Yau derivations} of $R$, denoted by
$YDer_{\lambda}(R)$, to consist of the functions $f \in C^{1}_{Yau}(R)$
such that
\[f(ij) = \Psi^{j}\circ f(i) + f(j) \circ \Psi^{i},\] for all $i,j \in \mathbb{N}$. We define the group of \textit{Yau inner-derivations} of $R$,
denoted by $YIDer_{\lambda}(R)$, to consist of the functions
$f:\mathbb{N} \rightarrow End(R)$ which are of the form \[f(i) =
\Psi^{i}\circ g - g \circ \Psi^{i},\] for some $g \in
\overline{End}(R)$.

The first Yau cohomology is given by the quotient,
\[H^{1}_{Yau}(R) = \frac{YDer_{\lambda}(R)}{YIDer_{\lambda}(R)}.\]

Yau tells us that there exists a canonical surjection,
\[H^{2}_{Yau}(R) \twoheadrightarrow HH^{2}(\mathbb{Z}[\mathbb{N}]
,End(R)),\] and for $\nu \geq 3$, there exists a canonical
isomorphism,
\[H^{\nu}_{Yau}(R) \cong HH^{\nu}(\mathbb{Z}[\mathbb{N}] ,End(R)),\]
where $HH^{\nu}(\mathbb{Z}[\mathbb{N}] ,End(R))$ denotes the $\nu^{th}$
Hochschild cohomology of $\mathbb{Z}[\mathbb{N}]$ with coefficients
in $End(R)$.

Yau defined his cohomology in order to study deformations of
$\lambda$-rings.

We let \[\Psi^{*}_{t} = \psi^{*}_{0} + t\psi^{*}_{1} +
t^{2}\psi^{*}_{2} + \ldots \] be a formal power series, in which
each $\psi^{*}_{i}$ is a function \[\psi^{*}_{i} : \mathbb{N}
\rightarrow End(R),\]  satisfying the following properties. We let $\psi^{j}_{i}$ denote $\psi^{*}_{i}(j).$ \begin{enumerate}
                                     \item $\psi^{j}_{0}(r) = \Psi^{j}(r)$,
                                     \item $\psi^{1}_{i}=0$ \qquad
                                     for $i \geq 1$,
                                     \item $\psi^{kl}_{i}(r)=\sum_{j=0}^{i}\psi^{k}_{j}\circ\psi^{l}_{i-j}(r)$
                                      \qquad for $k,l \geq 1$ and $i \geq 0$,
                                     \item $\psi^{p}_{i}(r) \subset
                                     pR$ \qquad for $i \geq 1$ and
                                     $p$ prime.
                                   \end{enumerate}
Yau calls $\Psi^{*}_{t}$ a \textit{deformation} of $R$.

Note that the Gerstenhaber and Schack's definition we provided in \ref{Psi deformation} is very
similar to Yau's definition but gives a different result. We would like to compare the results in the case when $\alpha_{i}=0$ for $i \geq 1$. We omitted the condition
$\psi^{p}_{i}(r) \subset pR$ for $p$ prime, but introduced the condition
$\psi^{j}_{i}(rs) = \sum_{k=0}^{i}\psi^{j}_{k}(r)\psi^{j}_{i-k}(s)$.
This last condition makes things more complicated and may seem
strange, but it is necessary to ensure that
\[\Psi^{*}_{t}(rs)=\Psi^{*}_{t}(r)\Psi^{*}_{t}(s).\]
Yau's condition gives us $\psi^{i}_{1} \in End(R)$. Gerstenhaber and Schack's condition gives us $\psi^{i}_{1} \in \Der(R,R^{i})$ where $R^{i}$ is the $R$-module with $R$ as an abelian group and the following action of $R$
\[(r,a) \mapsto \Psi^{i}(r)a, \qquad \textrm{ for } r \in R, a \in R^{i} .\]

\chapter{Harrison cohomology of diagrams of commutative algebras}
\label{Chapter5}

\section{Introduction}
For this chapter we let $I$ denote a small category. A category $I$ is said to be \textit{small} if
the collection of morphisms is a set. We let $i,j,k$
denote objects in $I$ and we let $\alpha:i \rightarrow j$ and
$\beta:j \rightarrow k$ denote morphisms in $I$.

\begin{definition}
A \textit{diagram of commutative algebras} is a covariant functor
\[A: I \rightarrow \mathfrak{Com.alg},\] where $I$ is a small
category, and $\mathfrak{Com.alg}$ is some category of commutative
algebras. We call $I$ the \textit{shape} of the diagram.

If $A,B$ are two covariant functors from $I$ to
$\mathfrak{Com.alg}$, then a map of diagrams is a natural
transformation $\mu :A \rightarrow B$. We denote the category of diagrams of commutative algebras with
shape $I$ by $\mathfrak{Com.alg}^{I}$.
\end{definition}

\begin{definition}
An $A$-module is a functor $M: I \rightarrow \mathfrak{Ab}$ such
that for all $i \in I$ we have that $M(i) \in A(i)$-$\mathfrak{mod}$
and for all $\alpha \in I$ we have \[M(\alpha)(a \cdot m) =
A(\alpha)(a)\cdot M(\alpha)(m),\] for all $a \in A(i), m \in M(i)$.
We let $A-\mathfrak{mod}^{I}$ denote the category of all $A$-modules.
\end{definition}

\section{Natural System}
Let $A: I \rightarrow \mathfrak{Com.alg}$ be a diagram of a commutative algebra, and $M$ be an $A$-module. For any $n \geq 0$ there exists a natural system on $I$ as follows
\[D_{\alpha} := C^{n}_{Harr}(A(i), \alpha^{*} M(j)),\]
where $(\alpha: i \rightarrow j) \in I$ and $M(j)$ is considered an $A(i)$-module via $\alpha$. For any $(\beta:j \rightarrow j') \in I$, we have $\beta_{*}:D_{\alpha} \rightarrow D_{\beta\alpha}$ which is induced by $M(\beta): M(j) \rightarrow M(j')$.  For any $(\gamma:i' \rightarrow i) \in I$, we have $\gamma^{*}:D_{\alpha} \rightarrow D_{\alpha\gamma}$ which is induced by $A(\gamma): A(i') \rightarrow A(i)$.

\section{Bicomplex}\label{dhbi}
Let $A: I \rightarrow \mathfrak{Com.alg}$ be a diagram of a
commutative algebra, and $M$ be an $A$-module. For each $i \in I$ we
can consider the Harrison cochain complex of the commutative algebra
$A(i)$ with coefficients in $M(i)$.
\[
\xymatrix{ C^{0}_{Harr}(A(i),M(i)) \ar[r] & C^{1}_{Harr}(A(i),M(i)) \ar[r] & C^{2}_{Harr}(A(i),M(i)) \ar[r]  & \ldots}
\]

We can use this to construct the following bicomplex denoted by $C^{*,*}_{Harr}(I,A,M)$:

\[
C^{p,q}_{Harr}(I,A,M) = \prod_{\alpha: i_{0}\rightarrow ... \rightarrow
i_{p}}C^{q+1}_{Harr}(A(i_{0})),\alpha^{*} M(i_{p})),
\]
for $p,q \geq 0$. The map $C^{p,q}_{Harr}(I,A,M) \rightarrow C^{p+1,q}_{Harr}(I,A,M)$ is the map in
the Baues-Wirsching cochain complex, and the map $C^{p,q}_{Harr}(I,A,M)
\rightarrow C^{p,q+1}_{Harr}(I,A,M)$ is the product of the Harrison coboundary maps.

\[
\xymatrix@W=1pc{ \vdots & \vdots
\\\prod_{i}C^{3}_{Harr}(A(i),M(i)) \ar[u]\ar^-{\delta}[r] &
\prod_{\alpha: i\rightarrow j}C^{3}_{Harr}(A(i),\alpha^{*}M(j)) \ar[r]\ar[u] &\ldots
\\\prod_{i}C^{2}_{Harr}(A(i),M(i)) \ar^-{\delta}[r] \ar^{\partial}[u]& \prod_{\alpha:i\rightarrow
j}C^{2}_{Harr}(A(i),\alpha^{*}M(j)) \ar[r]\ar^{-\partial}[u] & \ldots
\\\prod_{i}C^{1}_{Harr}(A(i),M(i)) \ar^-{\delta}[r] \ar^{\partial}[u] & \prod_{\alpha:i\rightarrow
j}C^{1}_{Harr}(A(i),\alpha^{*}M(j)) \ar[r]  \ar^{-\partial}[u] & \ldots}
\]

Let $(\alpha_{n}:i_{n} \rightarrow i_{n+1}) \in I$, and $\alpha = \alpha_{p} \ldots \alpha_{0}: i_{0} \rightarrow i_{p+1}$. Then the coboundary map $\delta: C^{p,q}_{Harr}(I,A,M) \rightarrow C^{p+1,q}_{Harr}(I,A,M)$ is given by

\begin{align*}
\delta (f)_{\alpha_{p+1},\ldots,\alpha_{0}}(x_{1},\ldots,x_{q}) =&  f_{\alpha_{p+1},\ldots,\alpha_{1}}(A(\alpha_{0})(x_{1}),\ldots,A(\alpha_{0})(x_{q}))
\\& + \sum_{k=0}^{p}(-1)^{k+1}f_{\alpha_{p+1},\ldots, \alpha_{k+1}\alpha_{k},\ldots,\alpha_{0}}(x_{1},\ldots,x_{q})
\\ &+ (-1)^{p+2} M(\alpha_{p+1})(f_{\alpha_{p},\ldots,\alpha_{0}}(x_{1},\ldots,x_{q})).
\end{align*}
The coboundary map $\partial: C^{p,q}_{Harr}(I,A,M) \rightarrow C^{p,q+1}_{Harr}(I,A,M)$ is given by
\begin{align*}
\partial (f)_{\alpha_{p},\ldots,\alpha_{0}}(x_{1},\ldots,x_{q+1}) =& A(\alpha)(x_{1})\cdot f_{\alpha_{p},\ldots,\alpha_{0}}(x_{2},\ldots,x_{q+1})
\\& + \sum_{k=1}^{q}(-1)^{k}f_{\alpha_{p},\ldots,\alpha_{0}}(x_{1},\ldots,x_{k}x_{k+1},\ldots,x_{q})
 \\ &+ (-1)^{q+1}f_{\alpha_{p},\ldots,\alpha_{0}}(x_{1},\ldots,x_{q})\cdot A(\alpha)(x_{q+1}).
\end{align*}

\begin{lemma} The maps $\partial$ and $\delta$ are coboundary maps. \[\partial^{2} = 0 = \delta^{2}.\]
\end{lemma}

\begin{proof}
$\partial(f) = \sum_{k=0}^{q+1}(-1)^{k}\partial_{k}(f)$ where
\[(\partial_{k}(f))(x_{1},\ldots,x_{q+1}) = \left\{
              \begin{array}{cc}
                A(\alpha)(x_{1})\cdot f_{\alpha_{p},\ldots,\alpha_{0}}(x_{2},\ldots,x_{q+1}) & k=0,\\
                 f_{\alpha_{p},\ldots,\alpha_{0}}(x_{1},\ldots,x_{k}x_{k+1},\ldots,x_{q+1}) & 0< k < q+1,\\
                f_{\alpha_{p},\ldots,\alpha_{0}}(x_{1},\ldots,x_{q})\cdot A(\alpha)(x_{q+1}) & k=q+1.\\
              \end{array}
            \right. \]

$\delta(f) = \sum_{k=0}^{p+2}(-1)^{k}\delta_{k}(f)$ where
$$(\delta_{k}(f))(x_{1},\ldots,x_{q}) = \left\{
              \begin{array}{cc}
                f_{\alpha_{p+1},\ldots,\alpha_{1}}(A(\alpha_{0})(x_{1}),\ldots,A(\alpha_{0})(x_{q})) & k=0,\\
                 f_{\alpha_{p+1},\ldots,\alpha_{k}\alpha_{k-1},\ldots,\alpha_{0}}(x_{1},\ldots,x_{q}) & 0< k < p+2,\\
                M(\alpha_{p+1})(f_{\alpha_{p},\ldots,\alpha_{0}}(x_{1},\ldots,x_{q})) & k=p+2.\\
              \end{array}
            \right.
 $$
$\partial^{2} = 0 = \delta^{2}$ follows from:
$$\partial_{k}\partial_{l} = \partial_{l}\partial_{k-1} \qquad 0 \leq l < k \leq q+2, $$
$$\delta_{k}\delta_{l} = \delta_{l}\delta_{k-1} \qquad 0 \leq l < k \leq p+2. $$
\end{proof}

\begin{lemma}  The coboundary maps $\partial$ and $\delta$ commute. \[\delta\partial = \partial\delta.\]
\end{lemma}
The proof is given on the next page.
\begin{proof}
Let $f \in C^{p,q}_{Harr}(I,A,M)$.
\begin{align*}
\delta\partial(f) =&
A(\alpha)(x_{1})\cdot f_{\alpha_{p+1},\ldots,\alpha_{1}}(A(\alpha_{0})(x_{2}),\ldots,A(\alpha_{0})(x_{q+1}))
\\&+ \sum_{k=1}^{q}(-1)^{k}f_{\alpha_{p+1},\ldots,\alpha_{1}}(A(\alpha_{0})(x_{1}),\ldots,A(\alpha_{0})(x_{k}x_{k+1}),\ldots,A(\alpha_{0})(x_{q+1}))
\\ &+
(-1)^{q+1}f_{\alpha_{p+1},\ldots,\alpha_{1}}(A(\alpha_{0})(x_{1}),\ldots,A(\alpha_{0})(x_{q}))\cdot A(\alpha)(x_{q+1})
\\&+
\sum_{l=0}^{p}(-1)^{l+1}[A(\alpha)(x_{1})\cdot f_{\alpha_{p+1},\ldots,\alpha_{l+1}\alpha_{l},\ldots,\alpha_{0}}(x_{2},\ldots,x_{q+1})
\\ &+ \sum_{k=1}^{q} (-1)^{k}
f_{\alpha_{p+1},\ldots,\alpha_{l+1}\alpha_{l},\ldots,\alpha_{0}}(x_{1},\ldots,x_{k}x_{k+1},\ldots,x_{q+1})
\\&+ (-1)^{q+1}
f_{\alpha_{p+1},\ldots,\alpha_{l+1}\alpha_{l},\ldots,\alpha_{0}}(x_{1},\ldots,x_{q})\cdot A(\alpha)(x_{q+1})]
\\&+
(-1)^{p+2}M(\alpha_{p+1})[A(\alpha_{p}\cdots \alpha_{0})(x_{1})\cdot f_{\alpha_{p},\ldots,\alpha_{0}}(x_{2},\ldots,x_{q+1})
\\ &+
\sum_{k=1}^{q}(-1)^{k}f_{\alpha_{p},\ldots,\alpha_{0}}(x_{1},\ldots,x_{k}x_{k+1},\ldots,x_{q+1})
\\&+
(-1)^{q+1}f_{\alpha_{p},\ldots,\alpha_{0}}(x_{1},\ldots,x_{q})\cdot A(\alpha_{p}\cdots \alpha_{0})(x_{q+1})]
\\  =&
A(\alpha)(x_{1})\cdot [f_{\alpha_{p+1},\ldots,\alpha_{1}}(A(\alpha_{0})(x_{2}),\ldots,A(\alpha_{0})(x_{q+1}))
\\&+
\sum_{l=0}^{p}(-1)^{l+1}f_{\alpha_{p+1},\ldots,\alpha_{l+1}\alpha_{l},\ldots,\alpha_{0}}(x_{2},\ldots,x_{q+1})
\\&+
(-1)^{p+2}M(\alpha_{p+1}) f_{\alpha_{p},\ldots,\alpha_{0}}(x_{2},\ldots,x_{q+1})]
\\&+ \sum_{k=1}^{q}(-1)^{k}
[f_{\alpha_{p+1},\ldots,\alpha_{1}}(A(\alpha_{0})(x_{1}),\ldots,A(\alpha_{0})(x_{k}x_{k+1}),\ldots,A(\alpha_{0})(x_{q+1}))
\\ &+
\sum_{l=0}^{p}(-1)^{l+1}f_{\alpha_{p+1},\ldots,\alpha_{l+1}\alpha_{l},\ldots,\alpha_{0}}(x_{1},\ldots,x_{k}x_{k+1},\ldots,x_{q+1})
\\&+
(-1)^{p+2}M(\alpha_{p+1}) f_{\alpha_{p},\ldots,\alpha_{0}}(x_{1},\ldots,x_{k}x_{k+1},\ldots,x_{q+1})]
\\ &+ (-1)^{q+1}[f_{\alpha_{p+1},\ldots,\alpha_{1}}(A(\alpha_{0})(x_{1}),\ldots,A(\alpha_{0})(x_{q}))
\\& + \sum_{l=0}^{p}(-1)^{l+1}f_{\alpha_{p+1},\ldots, \alpha_{l+1}\alpha_{l},\ldots,\alpha_{0}}(x_{1},\ldots,x_{q})
\\ &+ (-1)^{p+2} M(\alpha_{p+1}) (f_{\alpha_{p},\ldots,\alpha_{0}}(x_{1},\ldots,x_{q}))
]\cdot A(\alpha)(x_{q+1}) \\ =& \partial\delta(f).\end{align*}
\end{proof}

\section{Harrison cohomology of diagrams of commutative algebras}
Let $A: I \rightarrow \mathfrak{Com.alg}$ be a diagram of commutative algebras, and $M$ be an $A$-module.
We define the \textit{Harrison cohomology} of $A$ with coefficients in $M$, denoted by $Harr^{*}(I,A,M)$, to be the cohomology of the total
complex of  $C^{*,*}_{Harr}(I,A,M)$.

The spectral sequence of a bicomplex yields the following spectral sequence.
$$E^{p,q}_{2} = H^{p}_{BW}(I,\mathcal{H}_{Harr}^{q+1}(A,M)) \Rightarrow Harr^{p+q}(I,A,M),$$
where $\mathcal{H}_{Harr}^{q}(A,M)$ is the natural system on $I$ whose value on $(\alpha:i \rightarrow j)$ is given by $Harr^{q}(A(i),\alpha^{*}M(j))$.

\begin{definition}
A \textit{derivation} $d: A \rightarrow M$ is of the form $d = (d_{i})_{i \in I}$ where each $d_{i}:A(i) \rightarrow M(i)$ is a
derivation of $A(i)$ with values in $M(i)$ such that for all $(\alpha:i \rightarrow j)
\in I$ we have that $M(\alpha)(d_{i}) = d_{j}( A(\alpha))$.
We denote the set of all derivations of $A$ with values in $M$ by
$\mathfrak{Der}(A,M)$.
\end{definition}

\begin{lemma}
\[Harr^{0}(I,A,M) \cong \mathfrak{Der}(A,M),\]
\[H_{BW}^{0}(I,\mathcal{H}_{Harr}^{1}(A,M)) \cong \mathfrak{Der}(A,M).\]
\end{lemma}

\begin{definition}
An \textit{additively split extension} of $A$ by $M$ is an exact sequence of functors
\[\xymatrix{0 \ar[r] & M \ar^{q}[r] & X \ar^{p}[r] &  A \ar[r] & 0}\]
where $X: I \rightarrow \mathfrak{Com.alg}$ such that for all $i \in
I$ we get an additively split extension of $A(i)$ by $M(i)$.
\[\xymatrix{0 \ar[r] & M(i) \ar^{q(i)}[r] & X(i) \ar^{p(i)}[r] &  A(i) \ar[r] & 0}\]
This means that there are additive homomorphisms $s(i):A(i)
\rightarrow X(i)$ for all $i \in I$ such that $s(i)$ is a section of
$p(i)$. The sections induce additive isomorphisms $M(i) \oplus A(i)
\approx X(i)$ where addition is given by $(m,a)+(m',a') = (m+m',
a+a')$ and multiplication is given by
\[(m,a)(m',a') = (a'm + am' + f_{i}(a,a'), aa'),\] where
$f_{i}:A(i)\times A(i) \rightarrow M(i)$ is a bilinear map given by
\[f_{i}(a,a') = s(i)(a)s(i)(a')-s(i)(aa').\] Associativity in $X(i)$
gives us \[0 = af_{i}(a',a'')-f_{i}(aa',a'') +
f_{i}(a,a'a'')-f_{i}(a,a')a''.\] Commutativity in $X(i)$ gives us
\[f_{i}(a,a')=f_{i}(a',a).\] For all $(\alpha:i \rightarrow j) \in
I$ we identify $M(j)$ with $\Ker(p(j))$ and $M(\alpha)$ with the restriction of $X(\alpha)$ to get a map $\epsilon_{\alpha}:A(i) \rightarrow M(j)$ given by
\[\epsilon_{\alpha}(a) =  X(\alpha)(s(i)(a)) - s(j)(A(\alpha)(a)),\] which
satisfies the following properties:
\begin{enumerate}
  \item $\epsilon_{id}(a)=0$,
  \item $\epsilon_{\alpha}(a+a') = \epsilon_{\alpha}(a)+\epsilon_{\alpha}(a')$,
  \item $\epsilon_{\alpha}(aa') = A(\alpha)(a)\epsilon_{\alpha}(a') + A(\alpha)(a')\epsilon_{\alpha}(a) \\+f_{j}(A(\alpha)(a),A(\alpha)(a')) - M(\alpha)(f_{i}(a,a'))$,
  \item $\epsilon_{\beta\alpha}(a) = M(\beta)(\epsilon_{\alpha}(a)) + \epsilon_{\beta}( A(\alpha)(a))$.
\end{enumerate}
\end{definition}

Two additively split extensions $(X),(X')$ with $A,M$ fixed are said
to be \textit{equivalent} if there exists a map of diagrams $\phi: X
\rightarrow X'$ such that the following diagram commutes.
\[\xymatrix{0 \ar[r] & M \ar@{=}[d] \ar[r] & X \ar[r] \ar[d]^{\phi} & A \ar[r] \ar@{=}[d] & 0 \\ 0 \ar[r] & M \ar[r] & X' \ar[r] & A \ar[r] & 0}\]
For all $i \in I$ we get that $\phi_{i}: X(i) \rightarrow X'(i)$ is a homomorphism of commutative algebras. Hence $\phi_{i}(m,a) = (m + g_{i}(a),a)$ for some $g_{i}:A \rightarrow M$ such that
\[g_{i}(a+a') = g_{i}(a) + g_{i}(a'),\]
\[f_{i}(a,a') -f'_{i}(a,a') = ag_{i}(a) - g_{i}(aa') + g_{i}(a)a'.\]
For all $\alpha \in I$ we get that
\[\epsilon_{\alpha}(a) - \epsilon'_{\alpha}(a) = M(\alpha)(g_{i}(a)) - g_{j}( A(\alpha)(a)).\]

We denote the set of equivalence classes of additively split
extensions of $A$ by $M$ by $\mathfrak{AExt}(A,M)$.

An \textit{additively and multiplicatively split extension} of $A$
by $M$ is an additively split extension of $A$ by $M$
\[\xymatrix{0 \ar[r] & M \ar^{q}[r] & X \ar^{p}[r] &  A \ar[r] & 0}\]
such that for each $i \in I$ the arrow $p(i)$ is additively and
multiplicatively split.

We denote the set of equivalence classes of additively and
multiplicatively split extensions of $A$ by $M$ by
$\mathfrak{MExt}(A,M)$.

\begin{lemma}
\[Harr^{1}(I,A,M) \cong \mathfrak{AExt}(A,M).\]
\end{lemma}

\begin{proof}
A 1-cocycle is a pair $(f_{i}:A(i) \times A(i) \rightarrow M(i))_{i
\in I}$ and $(\epsilon_{\alpha}: A(i) \rightarrow M(j))_{(\alpha:i
\rightarrow j) \in I}$. We get an additively split extension of $A$
by $M$ given by taking the exact sequence
\[ \xymatrix{0 \ar[r] & M \ar[r] & M\oplus A \ar[r] &  A \ar[r] & 0} \]
where addition in $M\oplus A$ is given by $(m,a)+(m',a') = (m+m',
a+a')$ and multiplication is given by  \[(m,a)(m',a') = (a'm + am' +
f_{i}(a,a'), aa').\] For all $(\alpha:i \rightarrow j) \in I$ set the
map $(M\oplus A)(\alpha):(M\oplus A)(i) \rightarrow (M\oplus A)(j)$
to be \[ (M\oplus A)(\alpha)(m,a) = (M(\alpha)(m) +
\epsilon_{\alpha}(a), A(\alpha)(a)).\] Given two 1-cocycles which
differ by a 1-coboundary, then the two additively split extensions
we get are equivalent.

Given an additively split extension of $A$ by $M$
\[\xymatrix{0 \ar[r] & M \ar^{q}[r] & X \ar^{p}[r] &  A \ar[r] & 0}\]
there are additive homomorphisms $s(i):A(i) \rightarrow X(i)$ for
all $i \in I$ such that $s(i)$ is a section of $p(i)$.

For all $i \in I$ we define the maps $f_{i}:A(i)\times A(i)
\rightarrow M(i)$ to be given by \[f_{i}(a,a') =
s(i)(a)s(i)(a')-s(i)(aa').\] For all $(\alpha:i \rightarrow j) \in
I$ we define the maps $\epsilon_{\alpha}:A(i) \rightarrow M(j)$ to
be given by \[\epsilon_{\alpha}(a) =  X(\alpha)(s(i)(a)) -
s(j)(A(\alpha)(a)).\] Then $(f_{i}:A(i) \times A(i) \rightarrow
M(i))_{i \in I}$ and $(\epsilon_{\alpha}: A(i) \rightarrow
M(j))_{(\alpha:i \rightarrow j) \in I}$ give us a 1-cocycle. Given
two additively split extensions which are equivalent, then the two
1-cocycles we get differ by a 1-coboundary.
\end{proof}

\begin{corollary}
\[H_{BW}^{1}(I,\mathcal{H}_{Harr}^{1}(A,M)) \cong \mathfrak{MExt}(A,M).\]
\end{corollary}

\begin{definition}
An \textit{additively split crossed extension} of $A$ by $M$ is an exact sequence of functors
\[\xymatrix{0 \ar[r] & M \ar^{\phi}[r] & C_{1} \ar^{\rho}[r] & C_{0} \ar^{\pi}[r] &  A \ar[r] & 0}\]
such that for all $i \in I$ we get an additively split crossed extension of $A(i)$ by $M(i)$.
\begin{equation}\label{esasc} \xymatrix{0 \ar[r] & M(i) \ar^{\phi(i)}[r] & C_{1}(i) \ar^{\rho(i)}[r] & C_{0}(i) \ar^{\gamma(i)}[r] &  A(i) \ar[r] & 0}
\end{equation}
This means that all the arrows in the exact sequence $\ref{esasc}$
are additively split. We let $\pi_{0} \mathfrak{ACross}(A,M)$ denote the connected components of the category of additively split crossed extensions of $A$ by $M$.

An \textit{additively and multiplicatively split crossed extension}
of $A$ by $M$ is an exact sequence of functors
\[\xymatrix{0 \ar[r] & M \ar^{\phi}[r] & C_{1} \ar^{\rho}[r] & C_{0} \ar^{\gamma}[r] &  A \ar[r] & 0}\]
such that for all $i \in I$ we get an additively and multiplicatively split crossed
extension of $A(i)$ by $M(i)$,
\begin{equation} \xymatrix{0 \ar[r] & M(i) \ar^{\phi(i)}[r] & C_{1}(i) \ar^{\rho(i)}[r] & C_{0}(i) \ar^{\gamma(i)}[r] &  A(i) \ar[r] & 0}
\end{equation}
where $\gamma(i)$ and $\rho(i)$ are additively and multiplicatively split.  We let $\pi_{0} \mathfrak{MCross}(A,M)$ denote the connected components of the category of additively and multiplicatively split crossed extensions of $A$ by $M$.
\end{definition}

\begin{lemma}\label{hul1} If $\gamma: C_{0} \rightarrow A$ is a morphism of diagrams of commutative algebras  then
\[Harr^{1}(I,\gamma: C_{0} \rightarrow A,M) \cong \pi_{0} \mathfrak{ACross}(\gamma: C_{0} \rightarrow A ,M),\]
where $Harr^{*}(I,\gamma: C_{0} \rightarrow A,M)$ and $\pi_{0} \mathfrak{ACross}(\gamma: C_{0} \rightarrow A ,M)$ are defined as follows. Consider the following short exact sequence of cochain complexes:
\[\xymatrix{0 \ar[r] & C^{*}_{Harr}(I,A,M) \ar[r]^-{\gamma^{*}}   \ar[r] & C^{*}_{Harr}(I,C_{0},M) \ar[r] \ar[r]^-{\kappa^{*}} & \Coker(\gamma^{*}) \ar[r] & 0,}\] where $C^{*}_{Harr}(I,A,M)$ denotes the total complex of the bicomplex $(C^{*,*}_{Harr}(I,A,M)$. We define the cochain complex $C^{*}_{Harr}(I,\gamma: C_{0} \rightarrow A,M) := \Coker(\gamma^{*}).$ This allows us to define the relative Harrison cohomology \[Harr^{*}(I,\gamma: C_{0} \rightarrow A, M) := H^{*}(C^{*}_{Harr}(I,\gamma: C_{0} \rightarrow A, M)).\]

We let $\mathfrak{ACross}(\gamma: C_{0} \rightarrow A ,M)$ denote the category whose objects are the additively split crossed extensions of $A$ by $M$ \[\xymatrix{0 \ar[r] & M \ar^{\phi}[r] & C_{1} \ar^{\rho}[r] & C_{0} \ar^{\gamma}[r] &  A \ar[r] & 0}\] with $\gamma: C_{0} \rightarrow A$ fixed. A morphism between two of these crossed extensions consists of a morphism of diagrams of commutative algebras $h_{1}:C_{1} \rightarrow C_{1}$ such that the following diagram commutes.
\[\xymatrix{0 \ar[r] & M \ar@{=}[d] \ar[r]^{\phi} & C_{1} \ar^{h_{1}}[d] \ar[r]^{\rho} & C_{0} \ar@{=}[d] \ar[r]^{\gamma} & A \ar@{=}[d] \ar[r] & 0  \\ 0 \ar[r] & M \ar[r]^{\phi'} & C'_{1} \ar[r]^{\rho'} & C_{0} \ar[r]^{\gamma} & A \ar[r] & 0 \\}\]

Note that $\mathfrak{ACross}(\gamma: C_{0} \rightarrow A ,M)$ is a groupoid.
\end{lemma}

\begin{proof}
We use the method used in \cite{KL} for the crossed modules of Lie
algebras. Given any additively split crossed module of $A$ by $M$,
\[\xymatrix{0 \ar[r] & M \ar^{\phi}[r] & C_{1} \ar^{\rho}[r] & C_{0} \ar^{\gamma}[r] &  A \ar[r] & 0,}\]
we let $V= \Ker \gamma = \Image \rho$. For all objects $i \in I$
there are linear sections $s_{i}: A(i) \rightarrow C_{0}(i)$ of
$\gamma$ and $\sigma_{i} :V(i) \rightarrow C_{1}(i)$ of $\rho(i) :
C_{1}(i) \rightarrow V(i)$.  We define the maps $g_{i}: A(i) \otimes
A(i) \rightarrow C_{1}(i)$ by:
\[g_{i}(a,b)= \sigma_{i}(s_{i}(a)s_{i}(b)-s_{i}(ab)).\]
We also define the maps $\omega_{i} : C_{0}(i) \rightarrow C_{1}(i)$ by:
\[\omega_{i} (c) = \sigma_{i}(c- s_{i}\gamma_{i} (c)).\]
By identifying $M$ with $\Ker \rho$, we define the maps $f_{i}: C_{0}(i) \otimes C_{0}(i) \rightarrow M(i)$ by:
\[f_{i}(c,c') = g_{i}(\gamma_{i} (c), \gamma_{i} (c')) + c' \omega_{i}(c) + c \omega_{i}(c') - \omega_{i}(c) * \omega_{i}(c') - \omega_{i}(cc').\]
Since $g_{i}(c,c') = g_{i}(c',c)$, it follows that $f_{i}(c,c') = f_{i}(c',c)$ and so $f_{i} \in C^{2}_{Harr}(C_{0}(i),M(i))$.

For all morphisms $(\alpha: i \rightarrow j) \in I$ we define the maps $q_{\alpha}:A(i) \rightarrow C_{1}(j)$ by:
\[q_{\alpha}(a) = \sigma_{j}(C_{0}(\alpha)(s_{i}(a)) - s_{j}(A(\alpha)(a))).\]

By identifying $M$ with $\Ker \rho$, we define the maps $e_{\alpha}: C_{0}(i) \rightarrow M(j)$ by:
\[e_{\alpha}(c) = \omega_{j}(C_{0}(\alpha)(c)) - C_{1}(\alpha)(\omega_{i}(c)) -  q_{\alpha}(\gamma_{i}(c)).\]
Note that $e_{\alpha} \in C^{1}_{Harr}(C_{0}(i),\alpha^{*}M(j))$.

For all objects $i \in I$ we define the maps $\theta_{i} \in C^{3}_{Harr}(A(i),M(i))$ by:
\[\theta_{i}(x,y,z) = s_{i}(x)g_{i}(y,z) - g_{i}(xy,z) + g_{i}(x,yz) - g_{i}(y,x)s_{i}(z).\]

For all morphisms $(\alpha: i \rightarrow j) \in I$ we define the maps $\vartheta_{\alpha} \in C^{2}_{Harr}(A(i),\alpha^{*} M(j))$ by:
\begin{align*}\vartheta_{\alpha} (x,y) =& g_{j}(A(\alpha)(x), A(\alpha)(y)) - C_{1}(\alpha)g_{i}(x,y)  \\ &+ C_{0}(\alpha)(s_{i}(x))q_{\alpha}(y) - q_{\alpha}(xy) + q_{\alpha}(x)s_{j}(A(\alpha)(y)).\end{align*}

For all pairs of composable morphisms $(\beta\alpha:i \rightarrow j \rightarrow k) \in I$ we define the maps $\eta_{\beta\alpha} \in C^{1}_{Harr}(A(i),(\beta\alpha)^{*}M(k))$ by:
\[\eta_{\beta\alpha}(x) =  - q_{\beta}(A(\alpha)(x)) + q_{\beta\alpha}(x) - C_{1}(\beta)(q_{\alpha}(x)) .\]

We let $f = (f_{i})_{(i \in I)}$ and $e = (e_{\alpha})_{(\alpha:i \rightarrow j \in I)}$. We also let $\theta = (\theta_{i})_{(i \in I)}$, \\$\vartheta = (\vartheta_{\alpha})_{(\alpha:i \rightarrow j \in I)}$ and $\eta = (\eta_{\beta\alpha})_{(\beta\alpha:i \rightarrow j \rightarrow k \in I)}$. Consider the following commutative diagram.
\[\xymatrix{0 \ar[r] & C^{1}_{Harr}(I,A,M) \ar^{\gamma^{*}}[r] \ar[d] & C^{1}_{Harr}(I,C_{0},M) \ar^-{\kappa^{*}}[r] \ar^{\delta}[d] & \ar[r] \ar^{\delta}[d] C^{1}_{Harr}(I,\gamma: C_{0} \rightarrow A, M)  & 0 \\ 0 \ar[r] & C^{2}_{Harr}(I,A,M) \ar^{\gamma^{*}}[r]  & C^{2}_{Harr}(I,C_{0},M) \ar^-{\kappa^{*}}[r]  & \ar[r]  C^{2}_{Harr}(I,\gamma: C_{0} \rightarrow A, M)  & 0 }\]

Note that $(f,e) \in C^{1}_{Harr}(I,C_{0},M)$ and $(\theta,\vartheta,\eta) \in C^{2}_{Harr}(I,A,M)$. A direct calculation shows that $\delta (f,e) = \gamma^{*}(\theta,\vartheta,\eta)$. We also have that $\delta \kappa^{*}(f,e) = \kappa^{*} \delta (f,e) = \kappa^{*} \gamma^{*}(\theta,\vartheta,\eta) =0$, this tells us that $\kappa^{*} (f,e)$ is a cocycle. If we have two equivalent additively split crossed modules then we can choose sections in such a way that the associated cocycles are the same. Therefore we have a well-defined map:
\[ \xymatrix{\mathfrak{ACross}(\gamma: C_{0} \rightarrow A ,M) \rightarrow H^{2}_{Harr}(I,\gamma: C_{0} \rightarrow A ,M).}\]

Inversely, assume we have a cocycle in $C^{1}_{Harr}(I,\gamma: C_{0} \rightarrow A, M)$ which we lift to a cochain $(f,e) \in C^{1}_{Harr}(I,C_{0},M)$. Let $V = \Ker \gamma $. For all objects $i \in I$ we define $C_{1}(i) = M(i) \times V(i)$ as a module over $k$ with the following action of $C_{0}(i)$ on $C_{1}(i)$:
\[c(m,v) := (cm + f_{i}(c,v), cv).\]
The maps $C_{1}(\alpha): C_{1}(i) \rightarrow C_{1}(j)$ are given by:
\[C_{1}(\alpha)(m,v) := (M(\alpha)(m) + e_{\alpha}(v),C_{0}(\alpha)(v)).\]
It is easy to check using the properties of $f_{i}$ and $e_{\alpha}$ that this action is well defined and together with the maps $\rho_{i}: C_{0}(i) \rightarrow C_{1}(i)$ given by $\rho_{i} (m,v) = v$, we have an additively split crossed module of $A$ by $M$.
\end{proof}

\begin{lemma}\label{hul2}
If $k$ is a field of characteristic $0$ then \[Harr^{2}(I,A,M) \cong
\pi_{0} \mathfrak{ACross(A,M)}.\] \end{lemma}

\begin{proof}
From the definition of $C^{*}_{Harr}(I,\gamma: C_{0} \rightarrow A,M)$ we get the long exact sequence:
\begin{equation}
\xymatrix{ \ldots \ar[r] & Harr^{1}(I,A,M) \ar[r] &
Harr^{1}(I,C_{0},M) \ar[r] &  \\  & Harr^{1}(I,\gamma: C_{0}
\rightarrow A, M) \ar[r] &  Harr^{2}(I,A,M) \ar[r] &
\ldots}\end{equation} Given any additively split crossed module in
$\pi_{0} \mathfrak{ACross}(A,M)$,
\[\xymatrix{0 \ar[r] & M \ar[r]^{\phi} & C_{1} \ar[r]^{\rho} & C_{0} \ar[r]^{\gamma} & A \ar[r] & 0}\]
we can lift $\gamma$ to get a map $P_{0} \rightarrow A$ where
$P_{0}$ is free as a diagram of commutative algebras. We can then
use a pullback to construct $P_{1}$ to get a crossed module where
the following diagram commutes:
\[\xymatrix{0 \ar[r] & M \ar@{=}[d] \ar[r]^{\phi} & C_{1}  \ar[r]^{\rho} & C_{0}  \ar[r]^{\gamma} & A \ar@{=}[d] \ar[r] & 0 \\ 0 \ar[r] & M \ar[r] & P_{1} \ar[u] \ar[r] & P_{0} \ar[u] \ar[r] & A \ar[r] & 0 }\]
These two crossed modules are in the same connected component of
$\pi_{0} \mathfrak{ACross}(A,M)$. By considering the second crossed
module in the long exact sequence, we replace $C_{0}$ by $P_{0}$ to
get the new exact sequence:
\begin{equation}\label{lesrh2}
\xymatrix{ 0 \ar[r] &   Harr^{1}(I,\gamma: P_{0} \rightarrow A, M) \ar[r] &  Harr^{2}(I,A,M) \ar[r] & 0  }\end{equation}
since $Harr^{1}(I,P_{0},M)=0$ and $Harr^{2}(I,P_{0},M) =0$.

The exact sequence \ref{lesrh2} tells us that every element in
$Harr^{2}(I,A,M)$ comes from an element in $Harr^{1}(I,\gamma: P_{0}
\rightarrow A, M)$ and the previous lemma tells us that this comes
from a crossed module in $\pi_{0} \mathfrak{ACross}(A,M)$. Therefore
the map $\pi_{0} \mathfrak{ACross}(A,M) \rightarrow Harr^{2}(I,A,M)$
is surjective.

Assume we have two crossed modules which go to the same element  in $Harr^{2}(I,A,M)$,
\begin{equation}\label{Hxm12} \xymatrix{0 \ar[r] & M \ar[r]^{\phi} & C_{1} \ar[r]^{\rho} & C_{0} \ar[r]^{\gamma} & A \ar[r] & 0,}\end{equation}
\begin{equation}\label{Hxm22}\xymatrix{0 \ar[r] & M \ar[r]^{\phi '} & C'_{1} \ar[r]^{\rho '} & C'_{0} \ar[r]^{\gamma '} & A \ar[r] & 0.}\end{equation}
There exist morphisms
\[\xymatrix{0 \ar[r] & M \ar@{=}[d] \ar[r]^{\phi} & C_{1} \ar[d] \ar[r]^{\rho} & C_{0} \ar[d] \ar[r]^{\gamma} & A \ar@{=}[d] \ar[r] & 0 \\ 0 \ar[r] & M \ar[r] & P_{1} \ar[r] & P_{0} \ar[r] & A \ar[r] & 0, }\]
\[\xymatrix{0 \ar[r] & M \ar@{=}[d] \ar[r]^{\phi '} & C'_{1} \ar[d] \ar[r]^{\rho '} & C'_{0} \ar[d] \ar[r]^{\gamma '} & A \ar@{=}[d] \ar[r] & 0 \\ 0 \ar[r] & M \ar[r] & P_{2} \ar[r] & P_{0} \ar[r] & A \ar[r] & 0, }\]
where $P_{0}$ is free as a diagram of commutative algebras and
$P_{1},P_{2}$ are constructed via pullbacks. These give us two
elements in $Harr^{1}(I,\gamma: P_{0}\rightarrow A,M)$ which go to
the same element in $Harr^{2}(I,A,M)$. However the exact sequence
\ref{lesrh2} tells us that the two crossed modules \ref{Hxm12} and
\ref{Hxm22} have to go to the same element in $Harr^{1}(I,\gamma:
P_{0}\rightarrow A,M)$. The previous lemma tells us that the two
crossed modules \ref{Hxm12} and \ref{Hxm22} go to the same element
in $\mathfrak{ACross}(\gamma: C_{0} \rightarrow A ,M)$ which is a
groupoid, so there is a map $P_{2} \rightarrow P_{1}$ which makes
the following diagram commute:
\[\xymatrix{0 \ar[r] & M \ar@{=}[d] \ar[r]^{\phi} & C_{1}  \ar[r]^{\rho} & C_{0}  \ar[r]^{\gamma} & A \ar@{=}[d] \ar[r] & 0
\\ 0 \ar[r] & M  \ar@{=}[d] \ar[r] & P_{1} \ar[u] \ar[r] & P_{0} \ar@{=}[d] \ar[u] \ar[r] & A \ar@{=}[d] \ar[r] & 0
\\ 0 \ar[r] & M \ar[r] & P_{2} \ar[d] \ar[r] & P_{0} \ar[d] \ar[r] & A
\ar[r] & 0 \\ 0 \ar[r] & M \ar@{=}[u] \ar[r]^{\phi '} & C'_{1}
\ar[r]^{\rho '} & C'_{0}  \ar[r]^{\gamma '} & A \ar@{=}[u] \ar[r] &
0
 }\]
Therefore the two crossed modules \ref{Hxm12} and \ref{Hxm22} are in
the same connected component of $\pi_{0} \mathfrak{ACross}(A,M)$ and
the map $\pi_{0} \mathfrak{ACross}(A,M) \rightarrow Harr^{2}(I,A,M)$
is injective.\end{proof}

\begin{corollary} If $k$ is a field of characteristic 0 then
\[H_{BW}^{2}(I,\mathcal{H}_{Harr}^{1}(A,M)) \cong \pi_{0} \mathfrak{MCross}(A,M).\]
\end{corollary}

\begin{proof}
Given an additively and multiplicatively split crossed extension of $A$ by $M$ we get that (with the notation of lemma \ref{hul1}) $g_{i}=0$ for all $i \in I$. Since $\rho(i)$ is additively and multiplicatively split for all $i \in I$ it follows that $f=0$, $\theta=0$ and $\vartheta=0$.
Therefore $\eta$ is a cocycle in $C^{2}_{BW}(I,\mathcal{H}_{Harr}^{1}(A,M))$.

Inversely, the construction given in lemma \ref{hul1} gives us an additively and multiplicatively split extension.
\end{proof}

\section{Harrison cohomology of $\Psi$-rings}
Let $R$ be a $\Psi$-ring, and $M \in R-\mathfrak{mod}_{\Psi}$.
Let $I$ denote the category with one object associated to the multiplicative monoid of the natural numbers $\mathbb{N}^{mult}$.
For any $j \geq 1$, there is a natural system on $I$ as follows:
\[D_{f} := C_{Harr}^{j}(R,f^{*}M),\]
where $f^{*}M$ is an $\Psi$-module over $R$ with
$M$ as an abelian group and the following action of $R$
\[(r,m) \mapsto \Psi^{f}(r)m, \textrm{ for $r \in R, m \in M$}.\]
For $u \in \mathcal{F}I$ (the category of factorisations in $I$), we have $u_{*}:D_{f} \rightarrow
D_{uf}$ which is induced by $\Psi^{u}: f^{*}M \rightarrow (uf)^{*}M$.
For $v \in \mathcal{F}I$, we have $v^{*}:D_{f} \rightarrow
D_{fv}$ which is induced by $\Psi^{v}: R \rightarrow R$.

The bicomplex in section \ref{dhbi} becomes

\[
\xymatrix{ C^{1}_{Harr}(R,M) \ar[d]_-{d} \ar[r]^-{b} & \prod_{i \in \mathbb{N}}C^{1}_{Harr}(R,i^{*}M) \ar[d]_-{-d} \ar[r]^-{b} & \prod_{i,j \in \mathbb{N}}C^{1}_{Harr}(R,(ij)^{*}M) \ar[d]_-{d} \ar[r]^-{b} & \cdots\\
C^{2}_{Harr}(R,M) \ar[d]_-{d} \ar[r]^-{b} &
\prod_{i \in \mathbb{N}} C^{2}_{Harr}(R,i^{*}M) \ar[d]_-{-d} \ar[r]^-{b} &
\prod_{i,j \in \mathbb{N}}C^{2}_{Harr}(R,(ij)^{*}M) \ar[d]_-{d} \ar[r]^-{b} & \cdots
\\ C^{3}_{Harr}(R,M) \ar[d]_-{d} \ar[r]^-{b} & \prod_{i \in \mathbb{N}}C^{3}_{Harr}(R,i^{*}M) \ar[d]_-{-d} \ar[r]^-{b} & \prod_{i,j \in \mathbb{N}}C^{3}_{Harr}(R,(ij)^{*}M) \ar[d]_-{d} \ar[r]^-{b} & \cdots
\\ \vdots & \vdots & \vdots}
\]

with
\[
d: \prod_{t=t_{1}\ldots t_{i} \in \mathbb{N}} C^{j}_{Harr}(R,t^{*}M) \rightarrow
\prod_{t=t_{1}\ldots t_{i} \in \mathbb{N}} C^{j+1}_{Harr}(R,t^{*}M),\]
with the product being over $i$-tuples $(t_{1},\ldots,t_{i})$ and $t$ is the composite,
is given by
\begin{align*}df_{t_{1},\ldots,t_{i}}(x_{1},\ldots,x_{j+1}) = & \Psi^{t_{1} t_{2} \ldots
t_{i}}(x_{1})f_{t_{1},\ldots,t_{i}}(x_{2},\ldots, x_{j+1}) \\&+
\sum_{k=1}^{j}(-1)^{k}f_{t_{1},\ldots,t_{i}}(x_{1},\ldots,x_{k}x_{k+1},\ldots,x_{j+1})
\\ &+ (-1)^{j+1}f_{t_{1},\ldots,t_{i}}(x_{1},\ldots,x_{j})\Psi^{t_{1} t_{2}
\ldots t_{i}}(x_{j+1}). \end{align*}
and
\[
b: \prod_{t=t_{1}\ldots t_{i} \in \mathbb{N}} C^{j}_{Harr}(R,t^{*}M) \rightarrow
\prod_{t=t_{1}\ldots t_{i+1} \in \mathbb{N}} C^{j}_{Harr}(R,t^{*}M),\]
being given by
\begin{align*} bf_{t_{1},\ldots,t_{i+1}}(x_{1},\ldots,x_{j}) = &
\Psi^{t_{1}}f_{t_{2},\ldots,t_{i+1}}(x_{1},\ldots,x_{j}) \\&+
\sum_{k=1}^{i}(-1)^{k}f_{t_{1},\ldots,t_{k}t_{k+1},\ldots,t_{i+1}}(x_{1},\ldots,x_{j})
\\ &+
(-1)^{i+1}f_{t_{1},\ldots,t_{i}}(\Psi^{t_{i+1}}(x_{1}),\ldots,\Psi^{t_{i+1}}(x_{j})).
\end{align*}

We let $Harr^{i}_{\Psi}(R,M)$ denote the $i^{th}$ cohomology of
the total complex of the bicomplex described above.

\begin{theorem}
There exists a spectral sequence
\[E^{p,q}_{2} = H^{p}_{BW}(I,\mathcal{H}_{Harr}^{q+1}(R,M)) \Rightarrow Harr^{p+q}_{\Psi}(R,M).\]
where $\mathcal{H}_{Harr}^{q}(R,M)$ is the natural system on $I$ whose value on $(\alpha:i \rightarrow j)$ is given by $Harr^{q}(R,\alpha^{*}M)$.
\end{theorem}

\begin{theorem}
\[Harr^{0}_{\Psi}(R,M) = \Der_{\Psi}(R,M),\]
\[Harr^{1}_{\Psi}(R,M) = \AExt_{\Psi}(R,M),\]
\[Harr^{2}_{\Psi}(R,M) = \pi_{0} ACross_{\Psi}(R,M).\]
\end{theorem}

\section{Harrison cohomology and $\lambda$-rings}
Let $R$ be a $\lambda$-ring and $M \in R-\mathfrak{mod}_{\lambda}$.
\begin{conjecture}
There exists a cochain bicomplex which starts:
\[
\xymatrix{ C^{1}_{1-Harr}(\underline{R},\underline{M}) \ar[d]_-{d_{1}} \ar[r]^-{b^{1}_{1}} & C^{1}_{2-Harr}(\underline{R},\underline{M})) \ar[d]_-{-d_{2}} \ar[r]^-{b^{1}_{2}} & C^{1}_{3-Harr}(\underline{R},\underline{M})) \ar[d] \ar[r] & \cdots\\
C^{2}_{1-Harr}(\underline{R},\underline{M}) \ar[d]_-{d_{1}} \ar[r]^-{b^{2}_{1}} &
C^{2}_{2-Harr}(\underline{R},\underline{M})) \ar[d]_-{-d_{2}} \ar[r] &
\ddots  &
\\ C^{3}_{1-Harr}(\underline{R},\underline{M}) \ar[d]_-{d_{1}} \ar[r] & \ddots &   &
\\ \vdots & & }
\]
where the first column is the Harrison cochain complex.
\[C^{i}_{1-Harr}(\underline{R},\underline{M}) := C^{i}_{Harr}(\underline{R},\underline{M}).\]

For all $i\geq 1$ and $j\geq 2$ we have that

\[C^{i}_{j-Harr}(\underline{R},\underline{M}) \subset  \prod_{n_{1},\ldots,n_{j-1} \in \mathbb{N}}Maps(\underline{R}^{\otimes i},\underline{M}). \]

For example, when $j=2$, we have

\begin{align*}C^{1}_{2-Harr}(\underline{R},\underline{M})) = \{ & f \in \prod_{n \in \mathbb{N}}Maps(\underline{R},\underline{M}) | \\& f_{n}(r+s) = \sum_{j=1}^{n}[f_{j}(r)\lambda^{n-j}(s) + f_{j}(s)\lambda^{n-j}(r)] \}.\end{align*}

\begin{align*}
C^{2}_{2-Harr}(\underline{R},\underline{M})) = \{&f \in \prod_{n \in \mathbb{N}}Maps(\underline{R}\otimes\underline{R},\underline{M}) | f_{n}(r,s) = f_{n}(s,r), \\ &f_{n}((r,s)+(t,u)) = \sum_{j=1}^{n}[f_{j}(r,s)\lambda^{n-j}(tu+ru+ts)\\ &+ f_{j}(t,u)\lambda^{n-j}(rs+ru+ts) + f_{j}(r,u)\lambda^{n-j}(rs+tu+ts) + \\ & f_{j}(t,s)\lambda^{n-j}(rs+ru+tu)]\}.
\end{align*}

The coboundary maps $d_{2}:C^{i}_{2-Harr}(\underline{R},\underline{M}) \rightarrow C^{i+1}_{2-Harr}(\underline{R},\underline{M})$ are given by
\begin{align*}(d_{2}(f))_{n}(r_{1},\ldots,r_{i+1}) =& \sum_{j=1}^{n}[\frac{\partial P_{n}(r_{1},r_{2}\ldots r_{i+1})}{\partial \lambda^{j}(r_{2}\ldots r_{i+1})} f_{j}(r_{2},\ldots,r_{i+1}) ] \\ & + \sum_{j=1}^{i}f_{n}(r_{1},\ldots, r_{j}r_{j+1}, \ldots, r_{i+1}) \\ & + \sum_{j=1}^{n}[\frac{\partial P_{n}(r_{1}\ldots r_{i},r_{i+1})}{\partial \lambda^{j}(r_{1}\ldots r_{i})} f_{j}(r_{1},\ldots,r_{i}) ]. \end{align*}

\[(b^{1}_{1}(g))_{n}(r) = g(\lambda^{n}(r))-\sum_{i=1}^{n}\Lambda^{i}(g(r))\lambda^{n-i}(r). \]

\[(b^{1}_{2}(f))_{n,m}(r) = f_{m}(\lambda^{n}(r))-\sum_{i=1}^{nm}f_{i}(r) \frac{\partial P_{n,m}(r)}{\partial \lambda^{i}(r)} + \sum_{j=1}^{m}\Lambda^{j}(f_{n}(r))\lambda^{m-j}(\lambda^{n}(r)). \]

We let $Harr^{i}_{\lambda}(R,M)$ denote the $i^{th}$ cohomology of the bicomplex above. Then we get the following
\[Harr^{0}_{\lambda}(R,M) \cong \Der_{\lambda}(R,M),\]
\[Harr^{1}_{\lambda}(R,M) \cong \AExt_{\lambda}(R,M).\]
\end{conjecture}

\section{Gerstenhaber-Schack cohomology}
In the paper \cite{Gerdiag} Gerstenhaber and Schack describe a cohomology for diagrams of associative algebras which we denote by $H^{*}_{GS}(I,A,M)$. Let $I=\{i,j,k,\ldots \}$  be a partially ordered set. We can view $I$ as the set of objects of a category in which there exists a unique morphism $i\rightarrow j$ when $i \leq j$.  They define a diagram to be a contravariant functor $A:I^{op} \rightarrow \mathfrak{Com.alg}$. They define an $A$-module to be a contravariant functor $M: I^{op} \rightarrow \mathfrak{Ab}$ such that $M(i) \in A(i)-\mathfrak{mod}$ for all $i \in I$ and for each
$i \leq j$ the map $M(i \rightarrow j)$ is an $A(j)$-module homomorphism where $A(i)$ is viewed as an $A(j)$-module via the morphism $A(i \rightarrow j)$. If we consider $A$ as a covariant functor $A:I \rightarrow \mathfrak{Com.alg}$ and $M$ as a covariant functor $M:I \rightarrow \mathfrak{Ab}$ then we can apply the theory we developed earlier.

The bicomplex described by Gerstenhaber and Schack coincides with our bicomplex $C^{p,q}_{Harr}(I,A,M)$. Therefore $H^{n}_{GS}(I,A,M) = Harr^{n}(I,A,M)$ for $n \geq 0$. Therefore we get a new spectral sequence
\[E^{p,q}_{2} = H^{p}_{BW}(I,\mathcal{H}_{Harr}^{q+1}(A,M)) \Rightarrow H_{GS}^{p+q}(I,A,M),\]
where $\mathcal{H}_{Harr}^{q}(A,M)$ is the natural system on $I$ whose value on $(\alpha:i \rightarrow j)$ is given by $Harr^{q}(A(i),\alpha^{*}M(j))$.

\chapter{Andr\'{e}-Quillen cohomology of diagrams of algebras}
\label{Chapter6}

In this chapter, let $\mathfrak{C}$ denote a category with limits,
and $I$ denote a small category.

We have already seen that for algebraic objects, we can get
cohomology from monads and comonads. In this chapter, we define a
cohomology for diagrams of algebras. Our approach can be described
as follows. First, we fix a small category $I$. A diagram of algebras is a functor $I \rightarrow \mathfrak{Alg}(T)$, where $T$ is
a monad on sets. For appropriate $T$, one gets a diagram of groups,
a diagram of Lie algebras, a diagram of commutative rings, etc. The
adjoint pair $\xymatrix{\mathfrak{Alg}(T) \ar@<0.5ex>[r] &
\ar@<0.5ex>[l] \mathfrak{Sets} }$ yields a comonad which we denote
by $\mathbb{G}$. We can also consider the category $I_{0}$, which
has the same objects as $I$, but only the identity morphisms. The
inclusion $I_{0} \subset I$ yields the functor $\mathfrak{Sets}^{I}
\rightarrow \mathfrak{Sets}^{I_{0}}$ which has a left adjoint given
by the left Kan extension. We also have the pair of adjoint functors
$\xymatrix{\mathfrak{Alg}(T)^{I} \ar@<0.5ex>[r] & \ar@<0.5ex>[l]
\mathfrak{Sets}^{I} }$ which comes from the adjoint pair
$\xymatrix{\mathfrak{Alg}(T) \ar@<0.5ex>[r] & \ar@<0.5ex>[l]
\mathfrak{Sets} }$. By gluing these diagrams together, one gets
another adjoint pair
\[\xymatrix{\mathfrak{Alg}(T)^{I} \ar@<0.5ex>[r] & \ar@<0.5ex>[l]
\mathfrak{Sets}^{I_{0}}. }\] This adjoint pair yields a comonad which
we denote by $\mathbb{G}_{I}$.  We will prove that
$\mathfrak{Alg}(T)^{I}$ is monadic in $\mathfrak{Sets}^{I_{0}}$ and
the right cohomology theory of diagrams of algebras is one which is
associated to the comonad $\mathbb{G}_{I}$. These cohomology
theories are denoted by $H^{*}_{\mathbb{G}_{I}}(A,M)$.

\section{Base change} Let $\mathfrak{C}$ be a category, and $X$ be
an object in $\mathfrak{C}$. An \emph{$X$-module} in $\mathfrak{C}$
is an abelian group object in the category $\mathfrak{C}/X$,
\[X -\mathfrak{mod} := \mathrm{Ab}(\mathfrak{C}/X).\]

\begin{theorem}
Let $f: X \rightarrow Y$ be a morphism in $\mathfrak{C}$, then there
exists a base-change functor $f^{*}: Y-\mathfrak{mod} \rightarrow X-\mathfrak{mod}$ via
pullbacks.
\end{theorem}

\begin{proof}
The product in the slice category is given by pullbacks. The functor we are going to use is $f^{*}:\mathfrak{C}/Y \rightarrow
\mathfrak{C}/X $ given by pullbacks.

\[\xymatrix{f^{*}(M) \ar[r] \ar[d] & M \ar[d]^{p} \\ X \ar[r]^{f} & Y}\]
If $M \in Y-\mathfrak{mod}$ then $f^{*}(M)$ has a canonical $X$-module
structure. In set-theoretic notation,
\[f^{*}(M) = \{(x,m) | x \in X,\textrm{ } m \in M, \textrm{ } f(x) = p(m) \},\]
\[f^{*}(M) \times_{X} f^{*}(M) = \{(x,m, m') | x \in X,\textrm{ } m,m' \in M, \textrm{ } f(x) = p(m) = p(m')\},\]
\[f^{*}(M) \times_{X} f^{*}(M) \simeq f^{*}(M \times_{Y} M).\]
Consider the following commuting diagram.
\[\xymatrix{f^{*}(M \times_{Y} M) \ar[r] \ar[d] \ar@{..>}[ddr]^(0.3){\exists !} & M \times_{Y} M \ar[d] \ar[ddr]^{mult}\\ X \ar[r]_(0.35){f} \ar@{=}[ddr] & Y \ar@{=}[ddr] \\ & f^{*}(M) \ar[r] \ar[d] & M \ar[d]\\ & X \ar[r]_{f} & Y}\]
The unique morphism $f^{*}(mult):f^{*}(M \times_{Y} M) \rightarrow
f^{*}(M)$ exists by the universal property of pullbacks. The
isomorphism $f^{*}(M) \times_{X} f^{*}(M) \simeq f^{*}(M \times_{Y}
M)$ and this unique morphism yield multiplication
\[f^{*}(mult):f^{*}(M) \times_{X} f^{*}(M) \rightarrow
f^{*}(M),\] which gives an abelian group object structure on
$f^{*}(M)$.
\end{proof}

\section{Derivations}
For $M \in X-\mathfrak{mod}$, one defines a \emph{derivation} from $X$ to $M$
to be a morphism $d:X \rightarrow M$ which is a section of the
canonical morphism $M \rightarrow X$. Let $\Der(X,M)$ denote the set
of derivations $d: X \rightarrow M$. This is a special case of \ref{pderivation} and there is an abelian group structure. We will require the following
useful theorem later.
\begin{theorem}\label{derttt}
If $X = \coprod_{\alpha \in I} X_{\alpha}$ and $M \in X-\mathfrak{mod}$, then
\[\Der(X,M) \cong \prod_{\alpha \in I} \Der(X_{\alpha},M_{\alpha}),\]
where $M_{\alpha}$ is the $X_{\alpha}$-module produced from $M$ by the base-change
functor from the morphism $i_{\alpha}: X_{\alpha} \rightarrow X$.
\end{theorem}

\begin{proof}
From the definition of the coproduct one has a morphism $i_{\alpha}: X_{\alpha} \rightarrow X$. Using this one gets $M_{\alpha} \in
X_{\alpha}$-$\mathfrak{mod}$ via the following pullback diagram.
\[ \xymatrix{ M  \ar[d]_{p} & M_{\alpha}  \ar[l]_{j_{\alpha}} \ar[d]^{p_{\alpha}}  \\
 X   & X_{\alpha} \ar[l]^{i_{\alpha}} }\]
Let $f$ be a section of $p$, this means that $pf=id_{X}$. Consider
the following diagram.
\[\xymatrix{ & & X_{\alpha} \ar@{..>}[ld]^{f_{\alpha}}
\ar[lld]_{fi_{\alpha}} \ar[ddl]^{id_{X_{\alpha}}} \\
 M  \ar@<1ex>[d]^{p}  & M_{\alpha}  \ar[l]_{j_{\alpha}} \ar[d]_{p_{\alpha}}  &  \\
 X  \ar@<1ex>[u]^{f}  &  X_{\alpha} \ar[l]^{i_{\alpha}} &  }
\]
The diagram commutes since $pfi_{\alpha} = id_{X}i_{\alpha} =
i_{\alpha}id_{X_{\alpha}}$. By the universal property of pullbacks
$p_{\alpha}f_{\alpha} = id_{X_{\alpha}}$. So if $f$ is a section of
$p$ then $f_{\alpha}$ is a section of $p_{\alpha}$.

Conversely, let $f_{\alpha}$ be a section of $p_{\alpha}$, this
means that $p_{\alpha}f_{\alpha}=id_{X_{\alpha}}$. By the definition
of the coproduct there exists a unique morphism $f$ such that the
following diagram commutes.
\[
\xymatrix{
 M  \\
 X  \ar@{..>}[u]^{f} &  X_{\alpha} \ar[l]^{i_{\alpha}} \ar[ul]_{j_{\alpha}f_{\alpha}} }
\]
This means that $fi_{\alpha}= j_{\alpha}f_{\alpha}$. Composing with
$p$ on the left gives us that $ pfi_{\alpha}= pj_{\alpha}f_{\alpha}
= i_{\alpha}p_{\alpha}f_{\alpha} = i_{\alpha}id_{X_{\alpha}} =
i_{\alpha}$ Thus the following diagram commutes.
\[
\xymatrix{ X \ar@<1ex>[r]^{id_{X}} \ar@<-1ex>[r]_{pf}  & X \\
X_{\alpha} \ar[u]^{i_{\alpha}} \ar[ur]_{i_{\alpha}}& & }\] The
universal property of the coproduct says that $pf = id_{X}$. Hence
$f$ is a section of $p$.
\end{proof}

We will require the following useful lemma later.

\begin{lemma}\label{derlma}
For all objects $Z \in \mathfrak{C}^{I}$, for $M \in G_{I}(Z)-\mathfrak{mod}$, and
$\alpha: i \rightarrow j$ in the small category $I$, one has
\[\Der(G(Z(i)),\alpha^{*}M(j)) = \prod_{m \in UZ(i)}
p_{j}^{-1}\alpha_{*}\gamma_{i}(m),\] where $p_{j}$ is the canonical
 morphism $p_{j}:M(j) \rightarrow GZ(j)$ and $\gamma_{i}$ is the inclusion $\gamma_{i}: UZ(i) \rightarrow GZ(i)$.
\end{lemma}

\begin{proof}
The derivations $\Der(G(Z(i)),\alpha^{*}M(j))$ are the sections of
$p_{\alpha}$ in the following pullback diagram.
\[\xymatrix{&\alpha^{*}M(j) \ar[r]  \ar[d]_{p_{\alpha}}  & M(j)  \ar[d]^{p_{j}}  \\UZ(i) \ar@{^(->}[r]_{\gamma_{i}} &GZ(i) \ar[r]_{\alpha_{*}} & GZ(j)}\]

By definition, $UZ(i)$ is the basis of the free object $GZ(i)$.
\begin{align*}
\Der(G(Z(i)),\alpha^{*}M(j)) &= \{s:UZ(i) \rightarrow M(j) |
\alpha_{*}\gamma_{i} = p_{j}s, \textrm{ } s\textrm{ is a set map.}
\}
\\ &= \prod_{m \in UZ(i)} p_{j}^{-1}\alpha_{*}\gamma_{i}(m).
\end{align*}
\end{proof}

\section{Natural system}
We require the following useful theorem.

\begin{theorem}\label{natsysttt}
Let $A \in \mathfrak{C}^{I}$ and $M \in A$-$\mathfrak{mod}$. If $\alpha : i \rightarrow j$ is a morphism in $I$ then $M(j) \in
A(j)-\mathfrak{mod}$ and \[ \mathfrak{Der}(A,M)(\alpha) = \Der(A(i),
\alpha^{*}M(j)), \] defines a natural system on $I$.
\end{theorem}

\begin{proof}
Start by fixing $A$ and $M$, then let $D(\alpha)$ denote $
\mathfrak{Der}(A,M)(\alpha)$. Let $\gamma, \alpha, \beta \in I$ such
that
\[
\xymatrix{ i' \ar[r]^{\gamma} & i \ar[r]^{\alpha} & j \ar[r]^{\beta}
& j'. }
\]
We are going to show that we have induced maps as follows.
\[
\xymatrix{ D(\alpha \gamma) & D(\alpha) \ar[l]_-{\gamma^{*}}
\ar[r]^-{\beta^{*}} & D(\beta \alpha) }.
\]
Let $s \in D(\alpha)$, then the following diagram commutes with $ps
= id_{A(i)}$, and $\alpha^{*}M(j)$ is a pullback, $\alpha^{*}M(j)
\in A(i)-\mathfrak{mod}.$

\[
\xymatrix{\alpha^{*}M(j) \ar[rr]  \ar@<1ex>[dd]^{p} && M(j) \ar[dd] \\
 & &\\
A(i) \ar@<1ex>[uu]^{s} \ar[rr]^{A(\alpha)} && A(j) }\]

Consider the following commuting diagram.
\[
\xymatrix{M(i) \ar@<-4ex>@/_4pc/[dddddd] \ar@{..>}[dd]_{\exists !} \ar[ddrr]^{M(\alpha)} \\
\\
 \alpha^{*}M(j) \ar@<-1ex>@/_2pc/[dddd]_{p} \ar@{..>}[dd]^{\exists ! \tau} \ar[rr] && M(j) \ar@/^2pc/[dddd] \ar@<-1ex>@{..>}[dd]_{\exists !} \ar[ddrr]^{M(\beta)} \\
&  \\
\alpha^{*}\beta^{*}M(j') \ar[rr] \ar[dd]^{p'}  && \beta^{*}M(j') \ar@<-1ex>[dd] \ar[rr] && M(j') \ar[dd]\\
&  & & && \\
A(i) \ar@<3ex>@/^2pc/[uuuu]^{s} \ar[rr]_{A(\alpha)} && A(j)
\ar[rr]_{A(\beta)} && A(j')}
\]
Let $s': A(i) \rightarrow \alpha^{*}\beta^{*}M(j')$ be the map $s' =
\tau s$. Hence \[p' \tau s = p s = id_{A(i)}.\] So define $\beta^{*}(s) = s'.$ Hence $s' \in
\Der(A(i), \alpha^{*}\beta^{*}M(j')) = \Der(A(i),
(\beta\alpha)^{*}M(j'))$.

Consider the following commutative diagram, with $s$ a section of
$p$.
\[
\xymatrix{
(\alpha \gamma)^{*}M(j) \ar[rr]  \ar@<1ex>[dd]^{p'} && \alpha^{*}M(j) \ar[rr]  \ar@<1ex>[dd]^{p} && M(j) \ar[dd] \\
&  & &  &\\
A(i')  \ar[rr]^{A(\gamma)} && A(i) \ar@<1ex>[uu]^{s}
\ar[rr]^{A(\alpha)} && A(j) }\] There exists a unique $s': A(i')
\rightarrow (\alpha \gamma)^{*}M(j)$ which is a section of $p'$
which would make the above diagram still commute. So define $\gamma^{*}(s) = s'$. Therefore $s' \in
Der(A(i'), (\alpha\gamma)^{*}M(j))$.
\end{proof}

\begin{corollary} For $q \geq 0$ there exists a natural system
$\mathcal{H}^{q}(A,M)$ on $I$ whose value on $(\alpha: i \rightarrow
j)$ is given by $H^{q}_{\mathbb{G}}(A(i),\alpha^{*}M(j))$.
\end{corollary}

This corollary allows us to define, for fixed $q \geq 0$, the
Baues-Wirsching cohomology $H^{*}_{BW}(I,\mathcal{H}^{q}(A,M))$ of
$I$ with coefficients in the natural system $\mathcal{H}^{q}(A,M)$.

Furthermore, we can consider a natural system on the category of
chain complexes $\mathfrak{Chain complex}$ as follows. To each
morphism $\alpha:i \rightarrow j \in I$ we assign the chain complex
$\Der(\mathbb{G}_{*}(A(i)),\alpha^{*}M(j))$. This gives us a
functor,
\[ D: \mathcal{FI} \rightarrow \mathfrak{Chain complex},\]
where $\mathcal{FI}$ denotes the \emph{category of factorizations}
in $\mathcal{I}$.

This natural system gives rise to a cosimplicial object in
$\mathfrak{Chain complex}$:
\[\xymatrix{ \prod_{i}D(id_{i}) \ar@<.5ex>[r]\ar@<-.5ex>[r] & \prod_{\alpha :i\rightarrow j} D(\alpha) \ar[r] \ar@<.75ex>[r]\ar@<-.75ex>[r] & \ldots}\]
which gives rise to a bicomplex described in the next in the next section.

\section{Bicomplex}
Let $\mathbb{G}$ be a comonad in $\mathfrak{C}$, let $A \in \mathfrak{C}^{I}$ and $M \in A-\mathfrak{mod}$. Then we
can construct the following bicomplex denoted by $C^{*,*}(I,A,M)$.
\[
C^{p,q}(I,A,M) = \prod_{\alpha:i_{0}\rightarrow \ldots \rightarrow
i_{p}}\Der(G^{q+1}(A(i_{0})),\alpha^{*}M(i_{p})).
\]
The map $C^{p,q}(I,A,M) \rightarrow C^{p+1,q}(I,A,M)$ is the map in
the Baues-Wirsching cochain complex, and the map $C^{p,q}(I,A,M)
\rightarrow C^{p,q+1}(I,A,M)$ is the product of maps in the
comonad cochain complex.
\[
\xymatrix@W=1pc{ \vdots & \vdots
\\\prod_{i}\Der(G^{3}(A(i)),M(i)) \ar^{\partial}[u]\ar^-{\delta}[r] &
\prod_{\alpha:i\rightarrow j}\Der(G^{3}(A(i)),\alpha^{*}M(j)) \ar^-{\delta}[r]\ar^{-\partial}[u] &\ldots
\\\prod_{i}\Der(G^{2}(A(i)),M(i)) \ar^-{\delta}[r] \ar^{\partial}[u]& \prod_{\alpha:i\rightarrow
j}\Der(G^{2}(A(i)),\alpha^{*}M(j)) \ar^-{\delta}[r]\ar^{-\partial}[u] & \ldots
\\\prod_{i}\Der(G(A(i)),M(i)) \ar^-{\delta}[r] \ar^{\partial}[u] & \prod_{\alpha:i\rightarrow
j}\Der(G(A(i)),\alpha^{*}M(j)) \ar^-{\delta}[r]  \ar^{-\partial}[u] & \ldots}
\]

This bicomplex lives in the category of abelian groups. We let
\emph{$H^{*}(I,A,M)$} denote the cohomology of the total complex of
$C^{*,*}(I,A,M)$.

We will need the following useful lemmas.

\begin{lemma}\label{GIprojttt}
If $A$ is $\mathbb{G}_{I}$-projective, then $A(i)$ is
$\mathbb{G}$-projective for all $i \in I$.
\end{lemma}

\begin{proof}
Consider $A = G_{I}(Z): I \rightarrow \mathfrak{C}$ where
$G_{I}(Z)(i) = \coprod _{x \rightarrow i} G(Z(x))$. Since $G(Z(x))$
is $\mathbb{G}$-projective, it follows that $\coprod _{x \rightarrow
i} G(Z(x))$ is $\mathbb{G}$-projective for all $i \in I$.
\end{proof}

\begin{lemma}\label{lemmaHIAM}
$H^{0}(I,A,M) \cong \mathfrak{Der}(A,M)$, furthermore, if $A$ is
$\mathbb{G}_{I}$-projective then $H^{n}(I,A,M) = 0$ for $n > 0$.
\end{lemma}

\begin{proof}
It is sufficient to consider the case when $A= G_{I}(Z)$. When $A =
G_{I}(Z)$, it is known that $A$ is $\mathbb{G}_{I}$-projective. By
lemma \ref{GIprojttt} and lemma \ref{exacttt}, one gets that the
vertical columns in our bicomplex are exact except in dimension 0.
There is a well known lemma for bicomplexes which tells us the
cohomology of the total complex is isomorphic to the cohomology of
the following chain complex.

\[\xymatrix{\prod_{i}\Der(A(i),M(i)) \ar[r] & \prod_{\alpha:i \rightarrow j}\Der(A(i),\alpha^{*}M(j)) \ar[r] &
\ldots}\] It is known that the cohomology of this cochain complex is
just $H_{BW}^{*}(I,\mathfrak{Der}(A,M))$.

To prove the first statement it is enough to show that \[0
\rightarrow \mathfrak{Der}(A,M) \rightarrow \prod_{i}\Der(A(i),M(i))
\rightarrow \prod_{\alpha:i \rightarrow j}\Der(A(i),\alpha^{*}M(j))\] is
exact. Let $\psi \in \prod_{i}\Der(A(i),M(i))$ and $(\alpha:i
\rightarrow j) \in I$, then $d \psi (\alpha : i \rightarrow j) =
\alpha_{*}\psi(i) - \alpha^{*}\psi(j).$ Therefore $d\psi(\alpha: i
\rightarrow j) =0$ if and only if $\alpha_{*} \psi(i) =
\alpha^{*}\psi(j)$. However $\alpha_{*} \psi(i) = \alpha^{*}\psi(j)$
if and only if $M(\alpha)\psi(i) = \psi(j)A(\alpha),$ i.e. the
following diagram commutes.
\[\xymatrix{A(i) \ar[d]^{A(\alpha)} \ar[r]^{\psi(i)} & M(i) \ar[d]^{M(\alpha)} \\ A(j) \ar[r]_{\psi(j)} &
M(j)}\] Hence $\psi \in \mathfrak{Der}(A,M)$. This tells us that the
sequence above is exact. Hence $H^{0}(I,A,M) = \mathfrak{Der}(A,M)$.

To prove the second statement, let us consider \begin{align*}
D(\alpha: i\rightarrow j) :&= \mathfrak{Der}(A(i),\alpha^{*}M(j))
\\ & = \Der(\coprod_{\beta:y \rightarrow i}GZ(y), \alpha^{*}M(j))\\
& = \prod_{\beta:y \rightarrow i}\Der(GZ(y),
\beta^{*}\alpha^{*}M(j)), \textrm{ by lemma \ref{derttt}.}
\end{align*}

Define $D_{y}$ for a fixed object $y \in I$ to be a natural system
on $I$ (using theorem \ref{natsysttt}) given by: \[D_{y}(\alpha:i
\rightarrow j) = \prod_{\beta:y \rightarrow i}\Der(GZ(y),
\beta^{*}\alpha^{*}M(j)).\]

So one has that \[D(i\rightarrow j) = \prod_{y}D_{y}(i\rightarrow
j).\] Hence, \[H^{*}_{BW}(I,D) = \prod_{y \in I}H^{*}_{BW}(I,D_{y}).\]

Now consider the cochain complex $C^{*}_{BW}(I,D_{y})$.

\begin{align*}
C^{*}_{BW}(I,D_{y}) &= \xymatrix{\prod _{i}D_{y}(i\rightarrow i)
\ar[r] & \prod_{\alpha: i \rightarrow j}D_{y}(i\rightarrow j) \ar[r]
& \ldots} \\ &= \xymatrix{\prod_{i} \prod_{\beta:y \rightarrow
i}\Der(GZ(y), \beta^{*}M(i)) \ar[r] & \\ \prod_{ \alpha : i
\rightarrow j} \prod_{\beta:y \rightarrow
i}\Der(GZ(y),\beta^{*}\alpha^{*}M(j)) \ar[r] & \ldots}
\end{align*}
$UZ(y)$ forms a basis of the free object $GZ(y)$, applying lemma
\ref{derlma}, one can rewrite the cochain complex as
\[C^{*}_{BW}(I,D_{y}) = \xymatrix{\prod_{y\rightarrow i} \prod_{m \in
UZ(y)}A_{\beta j(m)} \ar[r]& \prod_{\alpha:i \rightarrow
j}\prod_{\beta:y \rightarrow i}\prod_{m \in UZ(y)}A_{\alpha\beta
j(m)} \ar[r] & \ldots }, \] where $A_{\beta j(m)} = \textrm{preimage
of } \beta\gamma(m) \textrm{ in the projection } M(j) \rightarrow
GZ(j)$. This allows us to rewrite the cochain complex as
\[C^{*}_{BW}(I,D_{y}) = \prod_{m \in UZ(y)}C^{*}_{BW}(y/I, F_{m})\]
where $F_{m}: y/I \rightarrow Ab$ is a functor defined by
$F_{m}(\beta:y\rightarrow i) = A_{\beta j(m)}.$

Since the category $y/I$ contains an initial object
$id_{y}:y \rightarrow y$, by lemma \ref{natsyslma} the
cohomology vanishes in positive dimensions.
\end{proof}

\begin{theorem}
$H^{*}_{\mathbb{G}_{I}}(A,M) \cong H^{*}(I,A,M)$.
\end{theorem}

\begin{proof}
Consider the bicomplex $C^{*}(I,\mathbb{G}_{I}(A)_{*},M)$ shown below.
\[\xymatrix{\vdots & \vdots & \vdots & \\ C^{2}(I,G_{I}(A),M) \ar[u] \ar[r] & C^{2}(I,G^{2}_{I}(A),M) \ar[u] \ar[r] & C^{2}(I,G^{3}_{I}(A),M) \ar[u] \ar[r] & \cdots  \\ C^{1}(I,G_{I}(A),M) \ar[u] \ar[r] & C^{1}(I,G^{2}_{I}(A),M) \ar[u] \ar[r] & C^{1}(I,G^{3}_{I}(A),M) \ar[u] \ar[r] & \cdots  \\ C^{0}(I,G_{I}(A),M) \ar[u] \ar[r] & C^{0}(I,G^{2}_{I}(A),M) \ar[u] \ar[r] & C^{0}(I,G^{3}_{I}(A),M) \ar[u] \ar[r] & \cdots}\]

We are going to show that
\[H^{*}(I,A,M) \cong H^{*}(Tot(C^{\bullet}(I,\mathbb{G}_{I}(A)_{\bullet},M))) \cong  H^{n}_{\mathbb{G}_{I}}(A,M).\]
Since $G_{I}^{p}(A)$ is $\mathbb{G}_{I}$-projective, lemma \ref{lemmaHIAM} tells us that the vertical cohomology
\[H^{n}(C^{*}(I,G_{I}^{p}(A),M)) \cong \left\{
                                  \begin{array}{ll}
                                    \mathfrak{Der}(G_{I}^{p}(A),M), & \hbox{$n=0$,} \\
                                    0, & \hbox{otherwise,}
                                  \end{array}
                                \right.
\]so each column of the bicomplex $C^{*}(I,\mathbb{G}_{I}(A)_{*},M)$ is exact except at $C^{0}(I,G_{I}^{p}(A),M)$. Therefore by the spectral sequence argument
\begin{align*}
 H^{n}(Tot(C^{\bullet}(I,\mathbb{G}_{I}(A)_{\bullet},M))) \cong & H^{n}(\mathfrak{Der}(G_{I}(A),M) \rightarrow \mathfrak{Der}(G_{I}^{2}(A),M) \rightarrow \cdots ) \\ = &  H^{n}_{\mathbb{G}_{I}}(A,M).
\end{align*}

We are now going to compute the horizontal cohomology. From the definition of $C^{*}(I,A,M)$ we see that each row of the bicomplex $C^{*}(I,\mathbb{G}_{I}(A)_{*},M)$ is a product of cochain complexes of the form $\Der(G^{p} \mathbb{G}_{I}(A)_{*}(i),\alpha^{*}M(j))$.

Consider $\mathbb{G}_{I}(A)_{*} \rightarrow A$ which is an augmented simplicial object. For all objects $i \in I$ we have $\mathbb{G}_{I}(A)_{*}(i) \rightarrow  A(i)$ which is also an augmented simplicial object. Applying the forgetful functor $U: \mathfrak{Alg}(T) \rightarrow \mathfrak{Sets}$ we get $U\mathbb{G}_{I}(A)_{*}(i) \rightarrow U A(i)$ which is contractible in the category $\mathfrak{Sets}$. Then applying the free functor $F: \mathfrak{Sets} \rightarrow \mathfrak{Alg}(T)$ we get $G\mathbb{G}_{I}(A)_{*}(i) \rightarrow G A(i)$  which is contractible in the category $\mathfrak{Alg}(T)$. Repeated applications of the functors $U$ and $F$ give us $G^{p}\mathbb{G}_{I}(A)_{*}(i) \rightarrow G^{p} A(i)$ which is contractible in the category $\mathfrak{Alg}(T)$. For any arrow $\alpha:i \rightarrow j$ in $I$ we can apply the functor $\Der(-,\alpha^{*}M(j))$ to get a contractible cosimplicial abelian group $\Der(G^{p} A(i),\alpha^{*}M(j)) \rightarrow \Der(G^{p}\mathbb{G}_{I}(A)_{*}(i),\alpha^{*}M(j))$.
Therefore each row of the bicomplex $C^{*}(I,\mathbb{G}_{I}(A)_{*},M)$ is exact except at $C^{p}(I,G_{I}(A),M)$. Therefore
\[H^{n}(C^{p}(I,G_{I}(A)_{*},M)) \cong \left\{
                                  \begin{array}{ll}
                                    C^{p}(I,A,M), & \hbox{$n=0$,} \\
                                    0, & \hbox{otherwise.}
                                  \end{array}
                                \right.
\]
Therefore by the spectral sequence argument
\[ H^{n}(Tot(C^{\bullet}(I,\mathbb{G}_{I}(A)_{\bullet},M))) \cong H^{n}(C^{*}(I,A,M))= H^{n}(I,A,M).\]
\end{proof}

Now one has both a global cohomology, $H^{*}_{\mathbb{G}_{I}}(A,M)$,
and a local cohomology, $H^{*}_{\mathbb{G}}(A(i),M(i))$. One can ask how these
two are related; the answer is given by the local to global
spectral sequence.

\begin{theorem}\label{AQss} There exists a spectral sequence \[E^{pq}_{2} = H^{p}_{BW}(I,\mathcal{H}^{q}(A,M))
\Rightarrow H^{p+q}_{\mathbb{G}_{I}}(A,M),\] where
$\mathcal{H}^{q}(A,M)$ is a natural system on $I$ whose value on
$(\alpha: i \rightarrow j)$ is given by
$H^{q}_{\mathbb{G}}(A(i),\alpha^{*}M(j))$.
\end{theorem}

\begin{definition}
An \textit{extension} of $A$ by $M$ is an exact sequence of functors
\[\xymatrix{0 \ar[r] & M \ar^{q}[r] & X \ar^{p}[r] &  A \ar[r] & 0}\]
where $X: I \rightarrow \mathfrak{Com.alg}$ such that for all $i \in
I$ we get an extension of $A(i)$ by $M(i)$
\[\xymatrix{0 \ar[r] & M(i) \ar^{q(i)}[r] & X(i) \ar^{p(i)}[r] &  A(i) \ar[r] & 0}\]
\end{definition}

Two extensions $(X),(X')$ with $A,M$ fixed are said
to be \textit{equivalent} if there exists a map of diagrams $\phi: X
\rightarrow X'$ such that the following diagram commutes.
\[\xymatrix{0 \ar[r] & M \ar@{=}[d] \ar[r] & X \ar[r] \ar[d]^{\phi} & A \ar[r] \ar@{=}[d] & 0 \\ 0 \ar[r] & M \ar[r] & X' \ar[r] & A \ar[r] & 0}\]

We denote the set of equivalence classes of extensions of $A$ by $M$ by $\mathfrak{Ext}(A,M)$.

\begin{theorem}
$H^{1}_{\mathbb{G}_{I}}(A,M) \cong \mathfrak{Ext}(A,M)$.
\end{theorem}

\begin{proof}
Suppose we have a free resolution $P_{*}$ of $A$ and an
extension representing a class in $\mathfrak{Ext}(A,M)$.
\[\xymatrix{0 \ar[r] & M \ar^{i}[r] & X \ar^{u}[r] & A \ar[r] & 0}\]
The map $u$ is a surjection and $P_{0}$ is free, so there exists a
lift $h:P_{0}\rightarrow X$ which makes the following diagram
commute.
\[\xymatrix{0 \ar[r] & M \ar^{i}[r] & X \ar^{u}[r] & A \ar@{=}[d] \ar[r] & 0 \\
       \ldots \ar@<1ex>^{\varphi^{2}_{0}}[r] \ar@<-1ex>_{\varphi^{2}_{2}}[r] \ar[r] &
       P_{1} \ar@<0.5ex>^{\varphi^{1}_{0}}[r] \ar@<-0.5ex>_{\varphi^{1}_{1}}[r] & P_{0} \ar@{..>}_{h}[u] \ar^{\varepsilon}[r] & A \ar[r] & 0}\]
Then we can get a map
$d=i^{-1}(h\varphi^{1}_{0}-h\varphi^{1}_{1}):P_{1} \rightarrow M.$

$d$ is a derivation, and $d$ is also a 1-cocycle in
$\mathfrak{Der}(P_{*},M)$ and defines a class in $H^{1}_{\mathbb{G}_{I}}(A,M)$.
This class is independent of the choice of lifting $h$. This gives a
map $\Phi:\mathfrak{Ext}(A,M) \rightarrow H^{1}_{\mathbb{G}_{I}}(A,M)$.

Conversely, given a derivation $D:P_{1} \rightarrow M$ we let
\[X = \Coker(\xymatrix{P_{1} \ar@<-0.5ex>_{(\varphi^{1}_{0},0)}[r] \ar@<0.5ex>^{(\varphi^{1}_{1},0)}[r] & P_{0}\oplus M}).\]

The cokernel is in the category $A-\mathfrak{mod}$, and we
let $p:P_{0}\oplus M \rightarrow X$ be the canonical projection. If
$D$ is a 1-cocycle in $\mathfrak{Der}(P_{*},M)$ then we obtain an
extension in $\mathfrak{Ext}(A,M)$ where $i:M \rightarrow X$ is
given by $i(m)=p(0\oplus m)$ and $u:X \rightarrow A$ is given by
$u(p(y\oplus m))=\varepsilon(y)$.

\[\xymatrix{\ldots \ar@<1ex>^{\varphi^{2}_{0}}[r] \ar@<-1ex>_{\varphi^{2}_{2}}[r] \ar[r] &
       P_{1} \ar_{D}[d] \ar@<0.5ex>^{\varphi^{1}_{0}}[r] \ar@<-0.5ex>_{\varphi^{1}_{1}}[r] & P_{0} \ar[d] \ar^{\varepsilon}[r] & A \ar[r] & 0\\
       0 \ar[r] & M \ar^{i}[r] & X \ar^{u}[r] & A \ar@{=}[u] \ar[r] & 0 } \]

This procedure gives us an inverse to $\Phi$.
\end{proof}

\begin{definition}
A \textit{crossed extension} of $A$ by $M$ is an exact sequence of functors
\[\xymatrix{0 \ar[r] & M \ar^{\omega}[r] & C_{1} \ar^{\rho}[r] & C_{0} \ar^{\pi}[r] &  A \ar[r] & 0}\]
such that for all $i \in I$ we get a crossed extension of $A(i)$ by $M(i)$
\[ \xymatrix{0 \ar[r] & M(i) \ar^{\omega(i)}[r] & C_{1}(i) \ar^{\rho(i)}[r] & C_{0}(i) \ar^{\pi(i)}[r] &  A(i) \ar[r] & 0}\]

We let $\pi_{0} \mathfrak{ACross}(A,M)$ denote the connected components of the category of additively split crossed extensions of $A$ by $M$.
\end{definition}

\begin{lemma}
\[H^{2}_{\mathbb{G}_{I}}(A,M) \cong \pi_{0} \mathfrak{Cross}(A,M).\]
\end{lemma}

\begin{proof}

We are going to show that the crossed extensions are equivalent to the simplicial groups whose Moore complex is of length one. Given a crossed extension we have a crossed module \[\xymatrix{C_{1} \ar^{\partial}[r] & C_{0}.}\]

Let $X_{0} = C_{0}$ and $X_{1} = C_{1} \oplus C_{0}$ where addition is given by $(c_{1},c_{0})+(d_{1},d_{0}) = (c_{1}+d_{1},c_{0}+d_{0})$ and multiplication is given by $(c_{1},c_{0})(d_{1},d_{0}) = ( 0 ,c_{0}d_{0} + \partial(c_{1})d_{1} + c_{0}d_{1} + d_{0}c_{1})$. For all $\alpha:i \rightarrow j$ then we have $X_{1}(\alpha) (c_{0},d_{0}) = (C_{1}(\alpha)(c_{1}),C_{0}(\alpha)(c_{0}))$. This gives us that $X_{1}$ is a diagram of algebras. We set $d_{1}:X_{1} \rightarrow X_{0}$ to be $d_{1}(c_{1},c_{0})=c_{0}$ and $d_{0}:X_{1} \rightarrow X_{0}$ to be $d_{0}(c_{1},c_{0})= \partial(c_{1}) + c_{0}$. Then $d_{0}$ is a natural transformation.

We define the category $\mathfrak{C}$ to be the category whose objects are the elements of
$X_{0}$ and whose morphisms are the elements of $X_{1}$. The source of the morphism
$(c_{1},c_{0}) \in \mathfrak{C}$ is given by $d_{0}(c_{1},c_{0}) =
\partial(c_{1})+c_{0}$ and the target of $(c_{1},c_{0}) \in \mathfrak{C}$ is given by $d_{1}(c_{1},c_{0}) = c_{0}$. The
composable morphisms in $\mathfrak{C}$ are pairs of morphisms
$(c_{1},c_{0}), (c'_{1},c'_{0})$ such that $c'_{0}=\partial c_{1} +
c_{0}$. The nerve of the category $\mathfrak{C}$ is a
simplicial group whose Moore complex is \[\xymatrix{\ldots \ar[r] &
0 \ar[r] & \Ker d_{1} \ar[r] & C_{0}, }\] which is of length one.

Let $K_{*}$ be a simplicial object whose Moore complex is of length one. Then the Moore complex
\[\xymatrix{\Ker d_{1} \ar[r]^{d_{0}} & K_{0},}\]
is a crossed module.

The category of diagrams of algebras is exact and so the results of Glenn \cite{Glenn} tell us that $H^{2}_{\mathbb{G}_{I}}(A,M)$ classifies the simplicial groups whose Moore complexes are of length one.

\end{proof}

\section{Cohomology of diagrams of groups}
In the paper by Cegarra \cite{Cegarra}, the cohomology of diagrams of groups is described, which we denote by $H_{Cg}^{*}(G,M)$. A diagram of groups is a functor $G: I \rightarrow \mathfrak{Gp}$ where $I$ is a small category and $\mathfrak{Grp}$ is the category of groups. A $G$-module is a functor $M:I \rightarrow \mathfrak{Ab}$ such that for all objects $i \in I$ we have that $M(i) \in A(i)-\mathfrak{mod}$ and for all morphisms $(\alpha:i \rightarrow j) \in I$ we have that $M(\alpha)(gm) = G(\alpha)(g)\cdot M(\alpha)(m)$ for all $g \in G(i)$ and $m \in M(i)$.

A derivation of $G$ into $M$ is a natural transformation $d: G \rightarrow M$ such that $d(i):G(i) \rightarrow M(i)$ is a derivation of the group $G(i)$ into $M(i)$. We denote the abelian group of all derivations of $G$ into $M$ by $\Der_{I}(G,M)$. When $G$ is locally constant then $H_{Cg}^{n+1}(G,M) = R^{n}\Der_{I}(G,M)$ and the following spectral sequence exists. \[E^{p,q}_{2} = H^{p}_{BW}(I,\mathcal{H}^{q+1}(G,M))
\Rightarrow H^{p+q+1}_{Cg}(G,M),\] where
$\mathcal{H}^{q}(G,M)$ is a natural system on $I$ whose value on
$(\alpha: i \rightarrow j)$ is given by
$H^{q}(G(i),\alpha^{*}M(j))$.
So when $G$ is locally constant the cohomology described by Cegarra coincides with the Andr\'{e}-Quillen cohomology described above with a dimension shift.

\chapter{Andr\'{e}-Quillen cohomology of $\Psi$-rings and $\lambda$-rings} 
\label{Chapter7}

\section{Cohomology of $\Psi$-rings}

Let $I$ denote  the category with one
object associated to the multiplicative monoid of the
nonzero natural numbers. We can consider
$\Psi$-rings as diagrams of commutative rings; $\Psi$-rings are
functors from $I$ to the category of commutative rings
 \[ R: I \rightarrow \mathfrak{Com.rings.}\]
Therefore we can use the theory we developed in the previous chapter.

We are now going to construct the free $\Psi$-ring on one generator
$a$. Let $A$ be the free commutative ring generated by $\{a_{i} | i
\in \mathbb{N} \}$. Let the operations $\Psi^{i}: A \rightarrow A$
be given by $\Psi^{i}(a_{j})=a_{ij}$, for $i,j \in \mathbb{N}$. Then
$A$ is the \textit{free $\Psi$-ring on one generator}.

\begin{lemma}
If $R$ and $S$ are $\Psi$-rings, then $R \otimes S$ with $\Psi^{i}:R \otimes S \rightarrow R \otimes S$ given by $\Psi^{i}(r,s) = (\Psi^{i}(r),\Psi^{i}(s)) $ is the coproduct in
the category $\Psi-\mathfrak{rings}$.
\end{lemma}

\begin{proof}
The coproduct of two commutative rings is given by the tensor
product, so we only need to check the $\Psi$-operations. There is a
unique $\Psi$-ring structure on $R \otimes S$ such that
\[ R \rightarrow R \otimes S, \qquad r \mapsto r \otimes 1,\]
\[ S \rightarrow R \otimes S, \qquad s \mapsto 1 \otimes s,\]
are homomorphisms of $\Psi$-rings given by
\begin{align*}\Psi^{i}(r \otimes s) &=\Psi^{i}((r \otimes 1)(1 \otimes s)) \\&= \Psi^{i}(r \otimes 1)\Psi^{i}(1 \otimes s) \\ &= (\Psi^{i}(r) \otimes 1)(1 \otimes \Psi^{i}(s)) \\&= \Psi^{i}(r) \otimes \Psi^{i}(s).
\end{align*}
\end{proof}

\begin{corollary}
Let $A$ be the free commutative ring generated by
$\{a_{i},b_{i},\ldots,x_{i} | i \in \mathbb{N}\}$.
Let the operations $\Psi^{i}: A \rightarrow A$
be given by $\Psi^{i}(a_{j})=a_{ij}$, $\Psi^{i}(b_{j})=b_{ij}$, $\ldots$, $\Psi^{i}(x_{j})=x_{ij}$ for $i,j \in \mathbb{N}$. Then $A$ is
the \textit{free $\Psi$-ring} generated by $\{a, b, \ldots, x\}$.
\end{corollary}

It is well known that there is an adjoint pair of functors
\[\xymatrix{\mathfrak{Sets}  \ar@<1ex>[rr]^{F} && \mathfrak{Com.rings} \ar@<1ex>[ll]^{U}},\]
where $U$ is the forgetful functor and $F$ takes a set $S$ to the free commutative ring generated by $S$.
The adjoint pair gives rise to a comonad $\mathbb{G}$ on $\mathfrak{Com.rings}$ which is monadic and the
cohomology with respect to this comonad is the Andr\'{e}-Quillen cohomology of commutative rings.

The adjoint pair gives rise to another adjoint pair
\[\xymatrix{\mathfrak{Sets}  \ar@<1ex>[rr]^{F_{I}} && \mathfrak{Com.rings}^{I} \ar@<1ex>[ll]^{U_{I}}},\]
where $U_{I}$ is the forgetful functor and $F_{I}$ takes a set $S$ to the free $\Psi$-ring generated by $S$.
This adjoint pair yields a comonad $\mathbb{G}_{I}$ on $\mathfrak{Com.rings}^{I} = \Psi-\mathfrak{rings}$ which is monadic. Note that for any $R \in \Psi-\mathfrak{rings}$, we get that $G_{I}(R) = \bigsqcup_{i \in \mathbb{N}} G(R)$.
We define the cohomology of a $\Psi$-ring $R$ with coefficients in $M \in R-\mathfrak{mod}_{\Psi}$   to be
\[H^{*}_{\Psi}(R,M) := H^{*}_{\mathbb{G}_{I}}(R,M) = H^{*}_{\mathbb{G}_{I}}(R,\Der_{\Psi}(-,M)).\]
From theorem \ref{natsysttt} it follows that for any $n \geq 0$, there is a natural system on
$I$ as follows
\[D_{f} := H_{AQ}^{n}(R,M^{f}),\] where $M^{f}$ is an $R$-module with
$M$ as an abelian group with the following action of $R$
\[(r,m) \mapsto \Psi^{f}(r)m, \textrm{ for $r \in R, m \in M$}.\]
For any morphism $u \in I$, we have $u_{*}:D_{f} \rightarrow D_{uf}$ which is
induced by $\Psi^{u}: M^{f} \rightarrow M^{uf}$. For any morphism $v \in I$, we
have $v^{*}:D_{f} \rightarrow D_{fv}$ which is induced by $\Psi^{v}:
R \rightarrow R$.

Therefore theorem \ref{AQss} gives us the following theorem.

\begin{theorem} There exists a spectral sequence
\[E^{p,q}_{2} = H^{p}_{BW}(I,\mathcal{H}^{q}(R,M)) \Rightarrow H^{p+q}_{\Psi}(R,M),\]
where $\mathcal{H}^{q}(R,M)$ is the natural system on $I$ whose value on a morphism $\alpha$ in $I$
 is given by $H_{AQ}^{q}(R,M^{\alpha})$.
\end{theorem}

\begin{theorem}
Let $R$ be a $\Psi$-ring and $M \in R\mathfrak{-mod}_{\Psi}$, then
\begin{enumerate}
    \item $H^{0}_{\Psi}(R,M) \cong \Der_{\Psi}(R,M)$,
    \item $H^{1}_{\Psi}(R,M) \cong \Ext_{\Psi}(R,M)$,
    \item $H^{2}_{\Psi}(R,M) \cong \pi_{0} Cross_{\Psi}(R,M)$,
    \item If $R$ is a free $\Psi$-ring, then $H^{n}_{\Psi}(R,M)=0$ for $n \geq 1$.
\end{enumerate}
\end{theorem}

\section{Cohomology of $\lambda$-rings}

We are now going to construct the free $\lambda$-ring on one generator
$a$. Let $A$ be the free commutative ring generated by $\{a_{i} | i
\in \mathbb{N} \}$. Let the operations $\lambda^{i}: A \rightarrow A$
be given by $\lambda^{i}(a_{j})= P_{i,j}(a_{1},\ldots,a_{ij})$ for $i,j \in \mathbb{N}$. Then $A$ is the \textit{free $\lambda$-ring on one generator}.

\begin{lemma}
If $R$ and $S$ are $\lambda$-rings, then $R \otimes S$ with $\lambda^{i}:R \otimes S \rightarrow R \otimes S$ given by $\lambda^{i}(r,s) = P_{i}((\lambda^{1}(r),1),\ldots,(\lambda^{i}(r),1),(1,\lambda^{1}(s)),\ldots,(1,\lambda^{i}(s)))$ is the coproduct in
the category $\lambda-\mathfrak{rings}$.
\end{lemma}

It is known that there is an adjoint pair of functors
\[\xymatrix{\mathfrak{Sets}  \ar@<1ex>[rr]^{F} && \lambda-\mathfrak{rings} \ar@<1ex>[ll]^{U}},\]
where $U$ is the forgetful functor and $F$ takes a set $S$ to the free $\lambda$-ring generated by $S$.
The adjoint pair gives rise to a comonad $\mathbb{G}$ on $\lambda-\mathfrak{rings}$ which is monadic.
We define the cohomology of a $\lambda$-ring $R$ with coefficients in $M \in R-\mathfrak{mod}_{\lambda}$ to be
\[H^{*}_{\lambda}(R,M) := H^{*}_{\mathbb{G}}(R,M) = H^{*}_{\mathbb{G}}(R,\Der_{\lambda}(-,M)).\]

\begin{theorem}
Let $R$ be a $\lambda$-ring and $M \in R\mathfrak{-mod}_{\lambda}$,
then
\begin{enumerate}
    \item $H^{0}_{\lambda}(R,M) \cong \Der_{\lambda}(R,M)$,
    \item $H^{1}_{\lambda}(R,M) \cong \Ext_{\lambda}(R,M)$,
    \item $H^{2}_{\lambda}(R,M) \cong \pi_{0} Cross_{\lambda}(R,M)$,
    \item If $R$ is a free $\lambda$-ring, then $H^{n}_{\lambda}(R,M)=0$ for $n \geq 1$.
\end{enumerate}
\end{theorem}

\begin{proof}
Property 1 follows from lemma \ref{exacttt}, and property 4 follows from lemma \ref{Gderiv}. We are now going to prove property 2.

Suppose we have a free resolution $P_{*}$ of $R$ as a $\lambda$-ring and an
extension representing a class in $\Ext_{\lambda}(R,M)$.
\[\xymatrix{0 \ar[r] & M \ar^{i}[r] & X \ar^{u}[r] & R \ar[r] & 0}\]
The map $u$ is a surjection and $P_{0}$ is free, so there exists a
lift $h:P_{0}\rightarrow X$ which makes the following diagram
commute.
\[\xymatrix{0 \ar[r] & M \ar^{i}[r] & X \ar^{u}[r] & R \ar@{=}[d] \ar[r] & 0 \\
       \ldots \ar@<1ex>^{\varphi^{2}_{0}}[r] \ar@<-1ex>_{\varphi^{2}_{2}}[r] \ar[r] &
       P_{1} \ar@<0.5ex>^{\varphi^{1}_{0}}[r] \ar@<-0.5ex>_{\varphi^{1}_{1}}[r] & P_{0} \ar@{..>}_{h}[u] \ar^{\varepsilon}[r] & R \ar[r] & 0}\]
Then we can get a map
$d=i^{-1}(h\varphi^{1}_{0}-h\varphi^{1}_{1}):P_{1} \rightarrow M$

$d$ is a $\Psi$-derivation, and $d$ is also a 1-cocycle in
$\Der_{\Psi}(P_{*},M)$ and defines a class in $H^{1}_{\Psi}(R,M)$.
This class is independent of the choice of lifting $h$. This gives a
map $\Phi:\Ext_{\Psi}(R,M) \rightarrow H^{1}_{\Psi}(R,M)$.

Conversely, given a $\lambda$-derivation $D:P_{1} \rightarrow M$ we
let
\[X = \Coker(\xymatrix{P_{1} \ar@<-0.5ex>_{(\varphi^{1}_{0},0)}[r] \ar@<0.5ex>^{(\varphi^{1}_{1},0)}[r] & P_{0}\oplus M}).\]

The cokernel is in the category $R\mathfrak{-mod}_{\lambda}$, and we
let $p:P_{0}\oplus M \rightarrow X$ be the canonical projection. If
$D$ is a 1-cocycle in $\Der_{\lambda}(P_{*},M)$ then we obtain an
extension in $\Ext_{\lambda}(R,M)$ where $i:M \rightarrow X$ is
given by $i(m)=p(0\oplus m)$ and $u:X \rightarrow R$ is given by
$u(p(y\oplus m))=\varepsilon(y)$.

\[\xymatrix{\ldots \ar@<1ex>^{\varphi^{2}_{0}}[r] \ar@<-1ex>_{\varphi^{2}_{2}}[r] \ar[r] &
       P_{1} \ar_{D}[d] \ar@<0.5ex>^{\varphi^{1}_{0}}[r] \ar@<-0.5ex>_{\varphi^{1}_{1}}[r] & P_{0} \ar[d] \ar^{\varepsilon}[r] & R \ar[r] & 0\\
       0 \ar[r] & M \ar^{i}[r] & X \ar^{u}[r] & R \ar@{=}[u] \ar[r] & 0 } \]

This procedure gives us an inverse to $\Phi$.

We are now going to prove property 3 by showing that the crossed $\lambda$-extensions are equivalent to the simplicial groups whose Moore complex is of length one. Given a crossed $\lambda$-extension we have a crossed $\lambda$-module \[\xymatrix{C_{1} \ar^{\partial}[r] & C_{0}.}\]

Let $X_{0} = C_{0}$ and $X_{1} = C_{1} \oplus C_{0}$ where addition
is given by $(c_{1},c_{0})+(d_{1},d_{0}) =
(c_{1}+d_{1},c_{0}+d_{0})$ and multiplication is given by
$(c_{1},c_{0})(d_{1},d_{0}) = ( 0 ,c_{0}d_{0} + \partial(c_{1})d_{1}
+ c_{0}d_{1} + d_{0}c_{1})$. We let $\lambda^{n}(c_{0},d_{0}) =
(\sum_{j=1}^{i}
\Lambda^{j}(c_{1})\lambda^{i-j}(c_{0}),\lambda^{i}(c_{0})$. This
gives us that $X_{1}$ is a $\lambda$-ring. We set $d_{1}:X_{1}
\rightarrow X_{0}$ to be $d_{1}(c_{1},c_{0})=c_{0}$ and $d_{0}:X_{1}
\rightarrow X_{0}$ to be $d_{0}(c_{1},c_{0})=
\partial(c_{1}) + c_{0}$. Then $d_{0}$ is a $\lambda$-ring map.

We define the category $\mathfrak{C}$ to be the category whose objects are the elements of $X_{0}$  and whose morphisms are the elements of $X_{1}$. The source of the morphism $(c_{1},c_{0}) \in \mathfrak{C}$ is given by $d_{0}(c_{1},c_{0}) = \partial(c_{1})+c_{0}$ and the target of $(c_{1},c_{0}) \in \mathfrak{C}$ is given by  $d_{1}(c_{1},c_{0}) = c_{0}$. The composable morphisms in $\mathfrak{C}$ are pairs of morphisms $(c_{1},c_{0}), (c'_{1},c'_{0})$ such that $c'_{0}=\partial c_{1} + c_{0}$.
Hence the nerve of the category $\mathfrak{C}$ is a simplicial group whose Moore complex is \[\xymatrix{\ldots \ar[r] & 0 \ar[r] & \Ker d_{1} \ar[r] & C_{0}, }\] which is of length one.

Let $K_{*}$ be a simplicial object whose Moore complex is of length one. Then the Moore complex yields
\[\xymatrix{\Ker d_{1} \ar[r]^{d_{0}} & K_{0},}\]
which is a crossed $\lambda$-module.

The category of $\lambda$-rings is exact and so the results of Glenn \cite{Glenn} tell us that $H^{2}_{\lambda}(R,M)$ classifies the simplicial groups whose Moore complexes are of length one.
\end{proof}

\begin{lemma}
Let $R$ be a $\lambda$-ring and let $M \in R$-mod$_{\lambda}$. Then
there exist homomorphisms, for $n \geq 0$,
\begin{align*} &\varsigma_{n}:
H^{n}_{\lambda}(R,M) \rightarrow H^{n}_{\Psi}(R_{\Psi},M_{\Psi}),\\
 &\rho_{n}: H^{n}_{\lambda}(R,M) \rightarrow H^{n}_{AQ}(\underline{R},\underline{M}),
 \\ &\varrho_{n}: H^{n}_{\Psi}(R_{\Psi},M_{\Psi}) \rightarrow H^{n}_{AQ}(\underline{R},\underline{M}).
\end{align*}

\end{lemma}

\begin{proof}
Let $P_{*}$ be a projective resolution of $R$ in the category of
$\lambda$-rings. Then applying the Adams operations we get that
$(P_{*})_{\Psi}$ is a (not necessarily projective) resolution of
$R_{\Psi}$ in the category of $\Psi$-rings. We let $L_{*}$
 be a projective resolution of $R_{\Psi}$ in the category of
$\Psi$-rings. Since $L_{*}$ is projective, we can use the lifting
property to get a map $\alpha: L_{*} \rightarrow (P_{*})_{\Psi}$,
such that the following diagram commutes.
\[\xymatrix{(P_{*})_{\Psi} \ar[r] & R_{\Psi} \ar@{=}[d] \\ L_{*} \ar@{.>}_{\exists}^{\alpha}[u] \ar[r] & R_{\Psi}}\]

We then apply the functor $\Der_{\Psi}(-,M_{\Psi})$ to get the
commutative diagram.

\[\xymatrix{\Der_{\Psi}((P_{*})_{\Psi},M_{\Psi}) \ar[r] \ar^{\alpha^{*}}[d] & \Der_{\Psi}(R_{\Psi},M_{\Psi}) \ar@{=}[d] \\ \Der_{\Psi}(L_{*},M_{\Psi})  \ar[r] & \Der_{\Psi}(R_{\Psi},M_{\Psi})}\]

The inclusion $i: \Der_{\lambda}(R,M) \hookrightarrow
\Der_{\Psi}(R_{\Psi},M_{\Psi})$ gives us maps which make the
following diagram commute.

\[\xymatrix{\Der_{\lambda}(P_{*},M) \ar[r] \ar^{i}[d] & \Der_{\lambda}(R,M) \ar^{i}[d] \\ \Der_{\Psi}((P_{*})_{\Psi},M_{\Psi}) \ar[r] \ar^{\alpha^{*}}[d] & \Der_{\Psi}(R_{\Psi},M_{\Psi}) \ar@{=}[d] \\ \Der_{\Psi}(L_{*},M_{\Psi})  \ar[r] & \Der_{\Psi}(R_{\Psi},M_{\Psi}).}\]

This gives us homomorphisms
\[\xymatrix{\varsigma_{n}:H^{n}_{\lambda}(R,M) = H^{n}(\Der_{\lambda}(P_{*},M)) \ar[r]^{(\alpha^{*}i)^{*}}& H^{n}(\Der_{\Psi}(L_{*},M_{\Psi})) = H^{n}_{\Psi}(R_{\Psi},M_{\Psi}).}\]

The homomorphisms $\rho_{n}$ and $\varrho_{n}$ are induced by the
forgetful functors from $\lambda-\mathfrak{rings}$ and
$\Psi-\mathfrak{rings}$ respectively to $\mathfrak{Com.rings}$.
\end{proof}

\chapter{Applications}
\label{Chapter8} 

\section{K-theory}
The material covered in this section can be found in \cite{AtK} and
\cite{Hatch}.
\subsection{Vector bundles}
In this section we will develop the notion of complex vector
bundles. A lot of the basic theory for real vector bundles is the
same as for complex vector bundles, however we will only be
concerned with complex vector bundles in this chapter.

\begin{definition}
A \emph{complex vector bundle} consists of
\begin{enumerate}
    \item topological spaces $X$ (called the base space) and $E$ (called the total space.)
    \item a continuous map $p:E \rightarrow X$ (called the projection.)
    \item a finite dimensional complex vector space structure on each \[E_{x} = p^{-1}(x) \qquad \textrm{for } x \in
    X,\] (we call the $p^{-1}(x)$ the fibres)
\end{enumerate}
such that the following local triviality condition is satisfied.
There exists an open cover of $X$ by open sets $U_{\alpha}$ and for
each there exists a homeomorphism $\varphi_{\alpha} :
p^{-1}(U_{\alpha}) \rightarrow U_{\alpha} \times \mathbb{C}^{d} $
which takes $p^{-1}(b)$ to $\{b\} \times \mathbb{C}^{d}$ via a
vector space isomorphism for each $b \in U_{\alpha}$.
\end{definition}

\begin{example}
Let $E= X \times \mathbb{C}^{d}$, and $p$ be the projection onto the
first factor. We call this the \textit{product} or \textit{trivial}
bundle.
\end{example}

A \textit{homomorphism} from a complex vector bundle $p:E
\rightarrow X$ to another complex vector bundle $q: F \rightarrow X$
is a continuous map $\varphi: E \rightarrow F$ such that
\begin{enumerate}
    \item $q \varphi = p$,
    \item $\varphi: E_{x} \rightarrow F_{x}$ is a linear map of
    vector spaces for all $x \in X$.
\end{enumerate}
If $\varphi$ is a bijection and $\varphi^{-1}$ is continuous,  then
we say that $\varphi$ is an \emph{isomorphism} and that $E$ and $F$
are \emph{isomorphic}. We will let $Vect(X)$ denote the set of
isomorphism classes of complex vector bundles on $X$.

Let $E$ be a complex vector bundle over $X$. We get that
$dim(E_{x})$ is locally constant on $X$, furthermore it is a
constant function on each of the connected components of $X$.

For vector bundles $E,F$ we can define the following corresponding
bundles
\begin{itemize}
    \item $E \oplus F$, the direct sum of $E$ and $F$,
    \item $E \otimes F$, the tensor product of $E$ and $F$,
    \item $\lambda^{k}(E)$, the $k^{th}$ exterior power of $E$.
\end{itemize}

There exist the following natural isomorphisms
\begin{itemize}
    \item $E \oplus F \cong F \oplus E$,
    \item $E \otimes F \cong F \otimes E$,
    \item $E \otimes (F \oplus F') \cong (E \otimes F) \oplus (E \otimes F')$,
    \item $\lambda^{k}(E \oplus F) \cong \bigoplus_{i+j=k}(\lambda^{i}(E) \otimes \lambda^{j}(F))$.
\end{itemize}

\subsection{K-theory}
For any space $X$, we can consider the set $Vect(X)$ which has an
abelian semigroup structure where  addition is given by the direct
sum. There is also a multiplication, given by tensor products, which
is distributive over the addition of $Vect(X)$ (this makes $Vect(X)$
into a semiring.)

If $A$ is an abelian semigroup, we can associate an abelian group
$K(A)$ to $A$. Let $F(A)$ be the free abelian group generated by
$A$, and let $E(A)$ be the subgroup of $F(A)$ generated by elements
of them form $a + a' - (a \oplus a')$, where $a,a' \in A$ and
$\oplus$ is the addition in $A$. We define the abelian group $K(A) =
F(A)/E(A)$. If $A$ is a semiring, then $K(A)$ is a ring.

If $X$ is a space, then we will write $K(X)$ for the ring
$K(Vect(X))$. Let $f: X \rightarrow Y$ be a continuous map. Then
$f^{*}: Vect(Y) \rightarrow Vect(X)$ induces a ring homomorphism
$f^{*}:K(Y) \rightarrow K(X)$ which only depends on the homotopy
class of $f$.

We can define operations $\lambda^{k}:K(X) \rightarrow K(X)$ using
the exterior powers. These make $K(X)$ into a $\lambda$-ring. We can
then use these to define the Adams operations $\Psi^{k}:K(X)
\rightarrow K(X)$ which makes $K(X)$ into a $\Psi$-ring.

If $X$ is a compact space with distinguished basepoint, then we
define  $\widetilde{K}(X)$ to be the kernel of $i^{*}:K(X)
\rightarrow K(x_{0})$ where $i:x_{0} \rightarrow X$ is the inclusion
of the basepoint. Let $c: X \rightarrow
x_{0}$ be the collapsing map, then $c^{*}$ induces a natural splitting $K(X) \cong
\widetilde{K}(X) \oplus K(x_{0})$.

\begin{example}
$\widetilde{K}(S^{2n}) \cong \mathbb{Z}[y]/(y)^{2}$, where $y$ is the $n$-fold external
product $(H-1)*\ldots *(H-1)$ and $H$ is the canonical line bundle
of $S^{2} = \mathbb{C}\mathrm{P}^{1}$. Multiplication in
$\widetilde{K}(S^{2n})$ is trivial, and the $\lambda$-operations
 $\lambda^{k}: \widetilde{K}(S^{2n}) \rightarrow \widetilde{K}(S^{2n})$ are given by \[\lambda^{k}(x) = (-1)^{k-1}k^{n-1}x.\] Hence the $\Psi$-operations $\Psi^{k}: \widetilde{K}(S^{2n}) \rightarrow \widetilde{K}(S^{2n})$ are given by \[\Psi^{k}(x)=k^{n}x.\]

\end{example}

\section{Natural transformation}
Let $X,Y$ be topological spaces such that $\widetilde{K}(Y)=0$ and $\widetilde{K}(\Sigma X) = 0$. Let $f:
Y \rightarrow X$ be a continuous map, then we can consider the Puppe exact sequence
$$\xymatrix{Y \ar^{f}[r] & X \ar[r] & C_{f} \ar[r] & \Sigma Y \ar[r] & \Sigma X \ar[r] & \Sigma C_{f} \ar[r] & \ldots}$$
where $C_{f}$ is the mapping cone of $f$, and $\Sigma X$ is the
suspension of $X$. After applying the functor $\widetilde{K}(-)$ we
get the long exact sequence.
$$\xymatrix{\ldots \ar[r] &  \widetilde{K}(\Sigma X) \ar[r] & \widetilde{K}(\Sigma Y) \ar[r] & \widetilde{K}(C_{f}) \ar[r] & \widetilde{K}(X) \ar[r] & \widetilde{K}(Y)}$$
However, since $\widetilde{K}(\Sigma X) = 0$ and
$\widetilde{K}(Y) = 0$ we obtain the short exact
sequence.
$$\xymatrix{ 0 \ar[r] & \widetilde{K}(\Sigma Y) \ar[r] & K(C_{f}) \ar[r] & K(X) \ar[r] & 0}$$
This gives us the following proposition.

\begin{proposition}
If $X$ and $Y$ are topological spaces as above then there exist
natural transformations $\tau_{\lambda} : [Y,X] \rightarrow
Ext_{\lambda}(K(X),\widetilde{K}(\Sigma Y))$ and $\tau_{\Psi} : [Y,X]
\rightarrow Ext_{\Psi}(K(X),\widetilde{K}(\Sigma Y)).$
\end{proposition}

\begin{corollary}
If $X$ is a topological space such that $\widetilde{K}(\Sigma X) = 0$
then there exist natural transformations $\tau_{\lambda,n} :
\pi_{2n-1}(X) \rightarrow Ext_{\lambda}(K(X),\widetilde{K}(S^{2n}))$
and $\tau_{\Psi,n} : \pi_{2n-1}(X) \rightarrow
Ext_{\Psi}(K(X),\widetilde{K}(S^{2n})).$
\end{corollary}

\section{The Hopf invariant of an extension}
We are going to give a proof of the classical result of Adams which was first proved by Adams, and subsequently by Adams-Atiyah \cite{AA}. We are going to use the same approach as Adams-Atiyah; using $\Psi$-rings.

\begin{definition}
Consider the commutative ring $R$ which is free as an abelian group with generators $x$ and $y$, $R \cong \mathbb{Z}x \oplus \mathbb{Z}y$, where $x$ is the unit of the ring and $y^{2} = 0$. Let $M \cong \mathbb{Z}z$ be the $R$-module such that $y \cdot z = 0$. We can consider the square zero extensions of $R$ by $M$ in the category of commutative rings.
All the square zero extensions have the following form
\begin{equation}\label{aqext} \xymatrix{ 0 \ar[r] & M \ar[r]
& X \oplus \mathbb{Z}\gamma \ar[r] & R \ar[r] &0 }\end{equation} where $X \cong
\mathbb{Z}\alpha \oplus \mathbb{Z}\beta$ as an abelian group with $\alpha$ being the
image of the generator $z$, the image of the unit $\gamma$ is the unit $x$ and the image of
$\beta$ being the generator $y$. Since $M^{2}=0$ we get that $\alpha^{2}=0$. Since
$y^{2}=0$, we get that $\alpha \beta = 0$ and
$\beta^{2} = h \alpha$ for some integer $h$. We define $h$ to be the \textit{Hopf invariant} of the extension (\ref{aqext}).
\end{definition}

Let $f: S^{4n-1} \rightarrow S^{2n}$ be a continuous map. We define the \textit{Hopf invariant} of the map $f$ to be the Hopf invariant of the short exact sequence \[\xymatrix{ 0 \ar[r] & \widetilde{K}(S^{4n}) \ar[r] & K(C_{f}) \ar[r] & K(S^{2n}) \ar[r] & 0}\] obtained from applying the natural transformation $\tau_{\Psi}$ to $f$.

We are going to consider the extensions of $K(S^{2n})$ by $\widetilde{K}(S^{2n'})$ in the category of $\Psi$-rings.
We are going to prove the following theorem.
\begin{theorem}\label{psiextttt}
 \[Ext_{\Psi}(K(S^{2n}),\widetilde{K}(S^{2n'})) \cong  \left\{
                                                             \begin{array}{ll}
                                                               \mathbb{Z} \oplus
\mathbb{Z}_{G_{n,n'}} & \textrm{if } n\neq n'; \\
                                                               \mathbb{Z} \oplus \prod_{\textrm{p prime}} \mathbb{Z} & \textrm{if } n=n'.
                                                             \end{array}
                                                           \right.
   \]
where $G_{n,n'}$ denotes the greatest common divisor of all the integers in
the set \\$\{l^{n}-l^{n'} | l \in \mathbb{Z}, l\geq 2\}.$
\end{theorem}

\begin{corollary}\label{spsiext}
If $n \neq n'$ then,
\[Ext_{\lambda}(K(S^{2n}),\widetilde{K}(S^{2n'})) \cong
\{(h,\nu) \in \mathbb{Z} \oplus \mathbb{Z}_{G_{n,n'}}|h \equiv \nu
\frac{(2^{n}-2^{n'})}{G_{n,n'}} \textrm{ mod }2\}.\] If $n = n'$,
then
\begin{align*}Ext_{\lambda}(K(S^{2n}),\widetilde{K}(S^{2n'}))
\cong \{(h,\nu_{2},\nu_{3},\ldots) \in \mathbb{Z} \oplus
\prod_{\textrm{p prime}} \mathbb{Z}|& h \equiv \nu_{2} \textrm{ mod
}2, \\ &\nu_{p}\equiv 0 \textrm{ mod }p, \textrm{
}p>2\},\end{align*}\end{corollary}

All the $\Psi$-ring extensions of $K(S^{2n})$ by $\widetilde{K}(S^{2n'})$ have the form (\ref{aqext}).
The $\Psi$-operations on $\Psi^{k}:X \rightarrow X$ are given by
\[\psi^{k}(m,r)= (k^{n'}m + \nu_{k}r,k^{n}r),\]
for some $\nu_{k} \in \mathbb{Z}$.

$$\Psi^{k}(\Psi^{l}(m,r)) = (k^{n'}l^{n'}m + k^{n'}\nu_{l}r + \nu_{k}l^{n}r,k^{n}l^{n}r),$$
$$\Psi^{l}(\Psi^{k}(m,r)) = (l^{n'}k^{n'}m + l^{n'}\nu_{k}r + \nu_{l}k^{n}r,l^{n}k^{n}r).$$
Since the $\Psi$-operations commute, we get that
$$\nu_{l}r(k^{n'}-k^{n}) = \nu_{k}r(l^{n'}-l^{n}).$$
If $n=n'$ then there is no restriction on the choice of $\nu_{p}$ for $p$ prime. Otherwise we can rearrange the above to get that
$$\nu_{l} = \nu_{k}\frac{(l^{n'}-l^{n})}{(k^{n'}-k^{n})}.$$
By setting $k=2$ we get that for all $l \geq 2$
$$\nu_{l} = \nu_{2}\frac{(l^{n'}-l^{n})}{(2^{n'}-2^{n})}.$$
We can write all the $\nu_{l}$'s as multiples of $\nu_{2}$ since
$$\nu_{l} = \nu_{2}\frac{(l^{n'}-l^{n})}{(2^{n'}-2^{n})} = \nu_{2}\frac{(k^{n'}-k^{n})}{(2^{n'}-2^{n})}\frac{(l^{n'}-l^{n})}{(k^{n'}-k^{n})} = \nu_{k}\frac{(l^{n'}-l^{n})}{(k^{n'}-k^{n}).}$$

Since $\nu_{2}$ is an integer, we get that $\nu_{2} =
\frac{z(2^{n'}-2^{n})}{G_{n,n'}}$ for some integer z.

If we replace the generator $\beta$ by $\beta + N\alpha$, note that
$(\beta + N\alpha)^{2} = h\alpha$, then we have to replace $\nu_{k}$
by $\nu_{k}+N(k^{n'}-k^{n})$. We get that
$$\nu_{k}+N(k^{n'}-k^{n}) =
\nu_{2}\frac{k^{n'}-k^{n}}{2^{n'}-2^{n}} +N(k^{n'}-k^{n}) =
\frac{(\nu_{2}+N(2^{n'}-2^{n}))(k^{n'}-k^{n})}{(2^{n'}-2^{n})}.
$$
So we only have to be concerned with replacing $\nu_{2}$ by
$\nu_{2}+N(2^{n'}-2^{n})$, then our usual formula for $\nu_{k}$
holds. Hence we are replacing $\frac{z(2^{n'}-2^{n})}{G_{n,n'}}$ by
$$\frac{z(2^{n'}-2^{n})}{G_{n,n'}} + N(2^{n'}-2^{n}) = \frac{(z+NG_{n,n'})(2^{n'}-2^{n})}{G_{n,n'}}.$$

This proves theorem \ref{psiextttt}. The isomorphism depends on $n$ and $n'$. If we now
introduce the property that $\Psi^{p}(x) \equiv x^{p}$ mod $p$, we
get that $\nu_{2}r \equiv hr^{2}$ mod $2$ and $\nu_{p}r \equiv 0$
mod $p$ for $p \geq 3$. This proves corollary \ref{spsiext}.

\begin{proposition}
If there exists an extension in
$Ext_{\lambda}(K(S^{2n}),\widetilde{K}(S^{2n'}))$ whose Hopf
invariant is odd, then either $n=n'$ or $min(n,n') \leq
g^{2}_{|n-n'|}$, where $g^{p}_{j}$ denotes the multiplicity of the
prime p in the prime factorisation of the greatest common divisor of
the set of integers $\{(k^{j}-1 ) |\textrm{ } k \in
\mathbb{N}-\{1,qp | \forall q \in \mathbb{N} \}\}.$
\end{proposition}

\begin{proof}The case when $n=n'$ is clear. Assume that $n \neq n'$, then the special $\Psi$-ring extensions are given by a pair $(h,\nu)$ where $h$ is the Hopf invariant. By \ref{spsiext}, $h$ can only be odd if $2^{n}$ divides $G_{n,n'}$. Assume that $n<n'$, since the other case is analogous. The multiplicity of 2 in the prime factorisation of $G_{n,n'}$ is $n$ if $n \leq g^{2}_{|n-n'|}$ or $g^{2}_{|n-n'|}$ if $g^{2}_{|n-n'|} < n$. It follows that if $n \leq g^{2}_{|n-n'|}$ then $2^{n}$ divides $G_{n,n'}$.
\end{proof}

Note that $g^{2}_{2n-1}=1$ for all $n\in \mathbb{N}$. Since $(k^{2n}-1)=(k^{n}+1)(k^{n}-1)$ it follows that $g^{2}_{2n}=\left\{
                                                                                                                                \begin{array}{ll}
                                                                                                                                  3, & n \textrm{ odd} \\
                                                                                                                                  g^{2}_{n}+1, & n \textrm{ even.}
                                                                                                                                \end{array}
                                                                                                                              \right.
$

\begin{theorem}
If there exists an extension in
$Ext_{\lambda}(K(S^{2n}),\widetilde{K}(S^{2n'}))$ whose Hopf
invariant is odd, then one of the following is satisfied
\begin{enumerate}
  \item $n=n'$,
  \item $n=1$ or $n'=1$,
  \item $n'-n$ is even and either $n=2$ or $n'=2$,
  \item $n'>n\geq 3$ and $n' = n+2^{n-2}b$ for some $b \in \mathbb{N}_{0}$,
  \item $n>n'\geq 3$ and $n = n'+2^{n'-2}b$ for some $b \in \mathbb{N}_{0}$.
\end{enumerate}
\end{theorem}

\begin{proof}
1. is clear. \\2. follows from $g^{2}_{n} \geq 1$ for all
$\mathbb{N}$.\\3. follows from $g^{2}_{2n} \geq 3$ for all $n \in
\mathbb{N}$.
\\4. and 5. follows from $g^{2}_{|n-n'|}$ being 2 plus the multiplicity of 2 in the prime factorisation of $|n-n'|$.
\end{proof}

\begin{corollary}
If there exists an extension in
$Ext_{\lambda}(K(S^{2n}),\widetilde{K}(S^{2(n+k)}))$ whose Hopf
invariant is odd, then one of the following is satisfied
\begin{enumerate}
  \item $k=0$,
  \item $n=1$,
  \item $k$ is even and $n=2$,
  \item $n\geq 3$ and $k = n+2^{n-2}b$ for some $b \in \mathbb{N}_{0}$.
\end{enumerate}
\end{corollary}

\begin{lemma}
If there exists an extension in
$Ext_{\lambda}(K(S^{2n}),\widetilde{K}(S^{2an}))$ for $a \in
\mathbb{N}$ whose Hopf invariant is odd, then one of the following
is satisfied

\begin{enumerate}
  \item $n=1,2$ or $4$,
  \item $n=3$ and $a$ is even,
  \item $n \geq 5$ and $an =2n +  2^{n-2}b$ for some $b \in \mathbb{N}_{0}$.
\end{enumerate}
\end{lemma}

\begin{corollary}
If there exists an extension in
$Ext_{\lambda}(K(S^{2n}),\widetilde{K}(S^{4n}))$ whose Hopf
invariant is odd, then $n=1,2$ or $4$.
\end{corollary}

\begin{corollary}[Adams]
If $f: S^{4n-1} \rightarrow S^{2n}$ is a continuous map whose Hopf invariant is odd, then $n=1,2$ or $4$.
\end{corollary}

%

\section{Stable Ext groups of spheres}
\begin{proposition}
If $n> k+1$ then $G_{n,n+k} = G_{n+1,n+k+1}$.
\end{proposition}

\begin{proof}
Let $n>k+1$. We know that $G_{n,n+k} = G_{n+1,n+k+1}$ if and only if the multiplicity of any prime $p$ in the prime factorization of $G_{n,n+k}$ is $g^{p}_{k}$. For all primes $p>2$ we get that $p^{n} > 2^{k}-1$, so the multiplicity of $p$ in the prime factorisation of $G_{n,n+k}$ is $g^{p}_{k}$. We can easily see that $g^{2}_{k} \leq k+1$ for all $k$. It follows that the multiplicity of $2$ in the prime factorisation of $G_{n,n+k}$ is $g^{2}_{k}$.
\end{proof}

\begin{corollary}
If $n> k+1$ then
\[\Ext_{\lambda}(K(S^{2n}),\widetilde{K}(S^{2(n+k)})) \cong
\Ext_{\lambda}(K(S^{2(n+1)}),\widetilde{K}(S^{2(n+k+1)})).\]
\end{corollary}

The groups $\Ext_{\lambda}(K(S^{2n}),\widetilde{K}(S^{2(n+k)}))$
are independent of $n$ for $n>k+1$, we call these the \textit{stable
Ext groups of spheres} which we denote by $\Ext^{s}_{2k}$.

\begin{proposition}
There are natural transformations
\[\Upsilon_{k}: \pi^{s}_{2k-1} \rightarrow \Ext^{s}_{2k}, \]
where $\pi^{s}_{2k-1}$ denotes the stable homotopy groups of spheres.
\end{proposition}

For small $k$ these groups look as follows.

\begin{center}
\begin{tabular}{|l | l | l |}
  \hline
  k & $\pi^{s}_{2k-1}$ & $\Ext^{s}_{2k}$ \\ \hline
  1 & $\mathbb{Z}_{2}$ & $2\mathbb{Z} \oplus \mathbb{Z}_{2}$ \\
  2 & $\mathbb{Z}_{24}\oplus \mathbb{Z}_{3}$ & $2\mathbb{Z}  \oplus \mathbb{Z}_{24}$ \\
  3 & $0$               & $2\mathbb{Z}\oplus \mathbb{Z}_{2}$ \\
  4 & $\mathbb{Z}_{240}$ & $2\mathbb{Z}\oplus \mathbb{Z}_{240}$ \\
  5 & $\mathbb{Z}_{2} \oplus \mathbb{Z}_{2} \oplus \mathbb{Z}_{2}$ & $2\mathbb{Z} \oplus \mathbb{Z}_{2}$ \\
  6 & $\mathbb{Z}_{504}$ & $2\mathbb{Z}\oplus \mathbb{Z}_{504}$ \\
  7 & $\mathbb{Z}_{3}$ & $2\mathbb{Z}\oplus \mathbb{Z}_{2}$ \\
  8 & $\mathbb{Z}_{480}\oplus \mathbb{Z}_{2}$ & $2\mathbb{Z} \oplus \mathbb{Z}_{480}$ \\
  \hline
\end{tabular}
 \end{center}

\appendix

\chapter{Adams Operations}
\label{appa} 
\[ \Psi^{k}(r) = \sum_{i=1}^{k-1}(-1)^{i+1}\lambda^{i}(r)\Psi^{k-i}(r) + (-1)^{k+1}k\lambda^{k}(r) \]

\begin{align*}
  \Psi^{1}(r) =& r
 \\ \Psi^{2}(r) =& r^{2} - 2\lambda^{2}(r)
 \\ \Psi^{3}(r) =& r^{3} - 3 r \lambda^{2}(r) + 3 \lambda^{3}(r)
 \\ \Psi^{4}(r) =& r^{4} - 4 r^{2} \lambda^{2}(r) + 4r \lambda^{3}(r) + 2(\lambda^{2}(r))^{2} - 4\lambda^{4}(r)
 \\ \Psi^{5}(r)=& r^{5} - 5 r^{3} \lambda^{2}(r) + 5r^{2}\lambda^{3}(r) + 5r(\lambda^{2}(r))^{2}-5r\lambda^{4}(r)-5\lambda^{2}(r)\lambda^{3}(r) + 5\lambda^{5}(r)
 \\ \Psi^{6}(r) =& r^{6} - 6 r^{4}\lambda^{2}(r) + 6 r^{3}\lambda^{3}(r) + 9 r^{2}(\lambda^{2})^{2} - 6r^{2}\lambda^{4} - 12r\lambda^{2}(r)\lambda^{3}(r) \\ &+ 6r\lambda^{5}(r) - 2(\lambda^{2}(r))^{3} + 3(\lambda^{3}(r))^{2} + 6\lambda^{2}(r)\lambda^{4}(r) -6\lambda^{6}(r)
 \\ \Psi^{7}(r) =& r^{7} - 7r^{5}\lambda^{2}(r) + 7r^{4}\lambda^{3}(r) + 14r^{3}(\lambda^{2}(r))^{2} - 7r^{3}\lambda^{4}(r) - 21 r^{2}\lambda^{2}(r)\lambda^{3}(r)  \\&+ 7r^{2}\lambda^{5}(r) - 7r(\lambda^{2}(r))^{3} + 7r(\lambda^{3}(r))^{2} + 14r\lambda^{2}(r)\lambda^{4}(r) - 7r\lambda^{6}(r) \\&+ 7(\lambda^{2})^{2}\lambda^{3}(r) - 7\lambda^{3}(r)\lambda^{4}(r) - 7\lambda^{2}(r)\lambda^{5}(r) + 7\lambda^{7}(r)
 \\ \Psi^{8}(r) =& r^{8} - 8r^{6}\lambda^{2}(r) + 8r^{5}\lambda^{3}(r) + 20 r^{4}(\lambda^{2}(r))^{2} - 8r^{4}\lambda^{4}(r) - 32r^{3}\lambda^{2}(r)\lambda^{3}(r) \\ & + 8r^{3}\lambda^{5}(r) - 16r^{2}(\lambda^{2}(r))^{3} + 12r^{2}(\lambda^{3}(r))^{2} + 24r^{2}\lambda^{2}(r)\lambda^{4}(r) - 8r^{2}\lambda^{6}(r) \\&+ 24r(\lambda^{2}(r))^{2}\lambda^{3}(r) -16r\lambda^{3}(r)\lambda^{4}(r) - 16r\lambda^{2}(r)\lambda^{5}(r) + 8r\lambda^{7}(r) + 2(\lambda^{2}(r))^{4} \\&- 8\lambda^{2}(r)(\lambda^{3}(r))^{2} + 4(\lambda^{4}(r))^{2} - 8(\lambda^{2}(r))^{2}\lambda^{4}(r) + 8\lambda^{3}(r)\lambda^{5}(r) \\&+ 8\lambda^{2}(r)\lambda^{6}(r) - 8\lambda^{8}(r)
\end{align*}

\chapter{Universal Polynomials $P_{i}, P_{i,j}$}\label{appb}
For more information on the universal polynomials, refer to the thesis of Hopkinson \cite{hopkins}. He has several results and gives the polynomial $P_{i}$ upto $i=10$, as well as giving several formulas for the polynomial $P_{i,j}$.

\begin{itemize}
  \item $P_{1}(s_{1}; \sigma_{1}) = s_{1} \sigma_{1}$
  \item $P_{2}(s_{1}, s_{2}; \sigma_{1},\sigma_{2}) = s_{1}^{2}\sigma_{2} - 2 s_{2} \sigma_{2} + s_{2} \sigma_{1}^{2}$
  \item $P_{3}(s_{1},s_{2},s_{3}; \sigma_{1},\sigma_{2},\sigma_{3})
  =
  s_{1}^{3}\sigma_{3} + s_{1}s_{2}\sigma_{1}\sigma_{2} -
  3s_{1}s_{2}\sigma_{3} + s_{3}\sigma_{1}^{3} -
  3s_{3}\sigma_{1}\sigma_{2} + 3s_{3}\sigma_{3}$
  \item $P_{4}(s_{1},s_{2},s_{3},s_{4}; \sigma_{1},\sigma_{2},\sigma_{3},\sigma_{4}) = - 2 s_{1}s_{3}\sigma^{2}_{2} + 2 s_{4} \sigma_{2}^{2} + 4 s_{4} \sigma_{1} \sigma_{3} - 4 s_{1}^{2} s_{2} \sigma_{4} - 2 s_{2}^{2} \sigma_{1} \sigma_{3} - 4 s_{4} \sigma_{1}^{2} \sigma_{2} + 4 s_{1} s_{3} \sigma_{4} + s_{1}^{2} s_{2} \sigma_{1} \sigma_{3} + s_{1} s_{3} \sigma_{1}^{2} \sigma_{2} - s_{1} s_{3} \sigma_{1} \sigma_{3} + s_{1}^{4} \sigma_{4} + s_{2}^{2} \sigma_{2}^{2} + 2 s_{2}^{2}\sigma_{4} + s_{4}\sigma_{1}^{4} - 4s_{4} \sigma_{4}$
\end{itemize}

\begin{itemize}
  \item $P_{1,1}(s_{1}) = s_{1}$
  \item $P_{1,j}(s_{1},\ldots,s_{j}) = s_{j}$
  \item $P_{i,1}(s_{1},\ldots, s_{i}) = s_{i}$
  \item $P_{2,j}(s_{1},\ldots,s_{2j}) = \sum_{k=1}^{j-1}(-1)^{k+1}s_{j-k}s_{j+k} + (-1)^{j+1}s_{2j}$
\end{itemize}

Consider the polynomials
\begin{itemize}
\item $P_{2,4}(s_{1},s_{2},s_{3},s_{4},s_{5},s_{6},s_{7},s_{8}) = s_{3}s_{5} - s_{2}s_{6} + s_{1}s_{7} - s_{8}$
\item $P_{4,2}(s_{1},s_{2},s_{3},s_{4},s_{5},s_{6},s_{7},s_{8}) = s_{1}s_{3}s_{4} - 3s_{1}s_{2}s_{5} + s_{1}^{3}s_{5} - s_{4}^{2} + s_{3}s_{5} - s_{1}^{2}s_{6} + s_{1}s_{7} + 2s_{2}s_{6} -s_{8}$
\item $P_{5,2}(s_{1},s_{2},s_{3},s_{4},s_{5},s_{6},s_{7},s_{8},s_{9},s_{10}) = s_{1}^{4}s_{6} + s_{2}s_{4}^{2} + 3s_{1}s_{2}s_{7} + 3 s_{1}s_{3}s_{6} - 4s_{1}^{2}s_{2}s_{6} -2s_{1}s_{4}s_{5} -2s_{2}s_{3}s_{5}+s_{1}^{2}s_{3}s_{5} + s_{10}-s_{3}s_{7}+2s_{5}^{2}-s_{1}^{3}s_{7}-2s_{4}s_{6}+2s_{2}^{2}s_{6}+s_{1}^{2}s_{8}-s_{1}s_{9}-2s_{2}s_{8}$

\end{itemize}
So we can see that in general $P_{i,j} \neq P_{j,i}$.

\chapter{Universal Polynomial Partial Derivatives}\label{appc}

\begin{center}
\begin{tabular}{|l | l | l | l  | l |  }
  \hline
  $k$ & $\frac{\partial P_{1}(r,s)}{\partial \lambda^{k}(r)} $ & $\frac{\partial P_{2}(r,s)}{\partial \lambda^{k}(r)} $ & $\frac{\partial P_{3}(r,s)}{\partial \lambda^{k}(r)} $  & $\frac{\partial P_{4}(r,s)}{\partial \lambda^{k}(r)}$ \\ \hline
  1 & $\Psi^{1}(s)$ & $r(s^{2}-\Psi^{2}(s))$ & $\ddots$ & $\ddots$ \\
  2 & 0 & $\Psi^{2}(s)$ & $r(s^{3}-\Psi^{3}(s))$ & $\ddots$ \\
  3 & 0 & 0 & $\Psi^{3}(s)$ & $r(s^{4}-\Psi^{4}(s))$ \\
  4 & 0 & 0 & 0 & $\Psi^{4}(s)$ \\
   \hline
\end{tabular}
 \end{center}

\begin{conjecture} For all $i \in \mathbb{N}$
\[\frac{\partial P_{i}(r,s)}{\partial \lambda^{i}(r)} = \Psi^{i}(s), \textrm{ } \frac{\partial P_{i+1}(r,s)}{\partial \lambda^{i}(r)} = r(s^{i+1} - \Psi^{i+1}(s)) \]
\end{conjecture}

From the other universal polynomial, we get
\[
\frac{\partial P_{1,n}(r)}{\partial \lambda^{k}(r)} =\left\{
                                                        \begin{array}{ll}
                                                          1 & \hbox{$k=n$} \\
                                                          0 & \hbox{otherwise}
                                                        \end{array}
                                                      \right.
\]

\[
\frac{\partial P_{2,n}(r)}{\partial \lambda^{k}(r)} = \left\{
                                                        \begin{array}{ll}
                                                          0 & \hbox{$k=n$, or $k > 2n$} \\
                                                          (-1)^{k+1}\lambda^{2n-k}(r) & \hbox{otherwise}
                                                        \end{array}
                                                      \right.
\]

\[\frac{\partial P_{i,j}(r)}{\partial \lambda^{ij}(r)}= (-1)^{(i+1)(j+1)}
\]

\begin{center}
\begin{tabular}{|l | l | l |}
  \hline

  k & $\frac{\partial P_{4,2}(r)}{\partial \lambda^{k}(r)} $ & $\frac{\partial P_{5,2}(r)}{\partial \lambda^{k}(r)} $ \\ \hline
  3 & $r\lambda^{4}(r) + \lambda^{5}(r)$ & $\ddots$ \\
  4 & $r\lambda^{3}(r)-2\lambda^{4}(r)$ & $\ddots$ \\
  5 & $\Psi^{3}(r)-2\lambda^{3}(r)$ & $\ddots$ \\
  6 &  $-\Psi^{2}(r)$ & $\Psi^{4}(r) -r\lambda^{3}(r)+2\lambda^{4}(r)$   \\
  7 & $r$ & $-\Psi^{3}(r)+2\lambda^{3}(r)$   \\
  8 & -1 & $\Psi^{2}(r)$ \\
  9 & 0 & $-r$ \\
  10 & 0 & 1 \\ \hline

\end{tabular}
 \end{center}

\bibliography{Bibliography}

\begin{thebibliography}{20}
\providecommand{\natexlab}[1]{#1}
\providecommand{\url}[1]{\texttt{#1}}
\expandafter\ifx\csname urlstyle\endcsname\relax
  \providecommand{\doi}[1]{doi: #1}\else
  \providecommand{\doi}{doi: \begingroup \urlstyle{rm}\Url}\fi

\bibitem[Adams and Atiyah(1966)]{AA}
J.~Adams and M.~Atiyah.
\newblock \emph{K-theory and the Hopf invariant}.
\newblock Quant. J Math Oxford (2) 17, 1966.
\newblock pages 31-38.

\bibitem[Atiyah(1989)]{AtK}
M.~Atiyah.
\newblock \emph{K-theory}.
\newblock Addison Wesley Longman Publishing Co, 1989.
\newblock ISBN 0201093944.

\bibitem[Barr(2002)]{Barr}
M.~Barr.
\newblock \emph{Acyclic Models}, volume Vol. 17 of \emph{CRM Monograph Series}.
\newblock American Mathematical Society, 2002.
\newblock ISBN 0821828770.

\bibitem[Baues and Wirsching(1984)]{BW}
H.~Baues and G.~Wirsching.
\newblock Cohomology of small categories.
\newblock \emph{Journal of pure and applied algebra 38}, 1984.

\bibitem[Cartan and Eilenberg(1999)]{CE}
H.~Cartan and S.~Eilenberg.
\newblock \emph{Homological Algebra}.
\newblock Princeton University Press, Princeton, 13th edition, 1999.
\newblock ISBN 0-691-07977-3.

\bibitem[Cegarra(2003)]{Cegarra}
A.~Cegarra.
\newblock Cohomology of diagrams of groups. the classification of (co)fibred
  categorical groups.
\newblock \emph{Int. Math. J. vol. 3 no. 7}, pages 643--680, 2003.

\bibitem[Gerstenhaber and Schack(1982)]{Gerdiag}
M.~Gerstenhaber and S.~D. Schack.
\newblock On the deformation of diagrams of algebras.
\newblock \emph{Algebraists' Homage: Papers in Ring Theory and Related Topics,
  Contemporary Mathematics , Amer. Math. Soc., Providence, Volume 13}, pages
  193--197, 1982.

\bibitem[Glenn(1982)]{Glenn}
P.~Glenn.
\newblock Realisation of cohomology classes in arbitrary exact categories.
\newblock \emph{Journal of pure and applied algebra 25}, pages 33--105, 1982.

\bibitem[Godement(1958)]{Godement}
R.~Godement.
\newblock \emph{Topologie alg\'{e}brique et th\'{e}orie des faisceaux}.
\newblock Hermann, 1958.

\bibitem[Harrison(1962)]{Harr}
D.~Harrison.
\newblock Commutative algebras and cohomology.
\newblock \emph{Trans. Amer. Math. Soc. 104}, pages 191--204, 1962.

\bibitem[Hatcher(2001)]{Hatch}
A.~Hatcher.
\newblock \emph{Vector bundles and K-theory}, volume v2.2.
\newblock Online at author's website, 2001.

\bibitem[Hopkinson(2006)]{hopkins}
J.~R. Hopkinson.
\newblock \emph{Universal polynomials in lambda rings and the K-theory of the
  infinite loop space tmf}.
\newblock PhD thesis, Massachusetts Institute of Technology, 2006.

\bibitem[Kassell and Loday(1982)]{KL}
C.~Kassell and J.-L. Loday.
\newblock Extensions centrales d'alg\`{e}bres de lie.
\newblock \emph{Ann. Inst. Fourier, Grenoble}, pages 119--142, 1982.

\bibitem[Knutson(1973)]{Knut}
D.~Knutson.
\newblock \emph{$\lambda$-Rings and the Representation Theory of the Symmetric
  Group}, volume 308.
\newblock Springer, 1973.
\newblock ISBN 3-540-06184-3.

\bibitem[Loday(1998)]{Loday}
J.~Loday.
\newblock \emph{Cyclic Homology}, volume 301.
\newblock Springer, Berlin, 2nd edition, 1998.
\newblock ISBN 3-540-63074-0.

\bibitem[MacLane(1971)]{MacCW}
S.~MacLane.
\newblock \emph{Categories for the Working Mathematician}.
\newblock Springer-Verlag, Berlin, Heidelberg, New York., 1971.

\bibitem[MacLane(1994)]{Mac}
S.~MacLane.
\newblock \emph{Homology}.
\newblock Springer-Verlag, Berlin, Heidelberg, New York, 4th edition, 1994.
\newblock ISBN 3-540-58662-8.

\bibitem[Quillen(1970)]{Quillen}
D.~Quillen.
\newblock On the (co)homology of commutative rings.
\newblock \emph{Proc. Symp. Pure Math. XVII.}, pages 65--87, 1970.

\bibitem[Weibel(1997)]{Weibel}
C.~Weibel.
\newblock \emph{An introduction to homological algebra}.
\newblock Cambridge University Press, hardback edition, 1997.
\newblock ISBN 0-521-43500-5.

\bibitem[Yau(2005)]{Yau}
D.~Yau.
\newblock Cohomology of $\lambda$-rings.
\newblock \emph{Journal of algebra}, \penalty0 (284):\penalty0 37--51, 2005.

\end{thebibliography}
\addcontentsline{toc}{chapter}{Bibliography}
\bibliographystyle{abbrvnat} 
\end{document}